\documentclass[10pt,a4paper,psamsfonts, english]{amsart}

\usepackage[T1]{fontenc}
\usepackage{color}
\usepackage{amsthm}
\usepackage{bm}
\usepackage{amstext}
\usepackage{amssymb}
\usepackage[all]{xy}

\makeatletter
\numberwithin{equation}{section} 
\numberwithin{figure}{section} 

\usepackage{amsmath}
\usepackage{amsfonts}
\usepackage{amsthm}
\usepackage{amssymb}
\usepackage{amscd}
\usepackage[all]{xy}
\usepackage{enumerate}
\usepackage{mathrsfs}
\usepackage{graphicx, bm, color, ulem}
\input xyimport.tex
\usepackage{accents}

\makeatletter
\def\widebar{\accentset{{\cc@style\underline{\mskip15mu}}}}
\makeatother

\def\ft#1{{\mathsf #1}}
\def\chow{{\mathscr X}}
\def\hcoY{{\mathscr Y}}
\def\Hes{{\mathscr H}}
\def\UU{{\mathscr U}}
\def\Zpq{{\mathscr Z}}
\def\zpq{\hspace{-0.1cm}\Zpq}

\def\Vs{{\mathscr V}}
\def\hchow{{\mathaccent20\chow}}
\def\Prt{{\mathscr P}}
\def\Lrho{\text{\large{\mbox{$\rho\hskip-0pt$}}}}

\def\Lpi{\text{\large{\mbox{$\pi\hskip-0pt$}}}}

\def\Lp{\text{\large{\mbox{$p\hskip-1pt$}}}}
\def\Lq{\text{\large{\mbox{$q\hskip-1pt$}}}}

\font\eu=eusm10 at 10pt

\def\eF{\text{\eu F}}
\def\eG{\text{\eu G}}
\def\eQ{\text{\eu W}}
\def\eS{\text{\eu U}}

\input xyimport.tex

\newcommand{\vcorr}[3][1]{%
  \begingroup
    \tabcolsep=.5\tabcolsep
    \sbox0{%
      \begin{tabular}[b]{@{}l}%
        #3%
         \tabularnewline
      \end{tabular}%
    }%
    \settoheight{\dimen0 }{%
      \rotatebox{#2}{%
        \copy0 %
        \kern-\tabcolsep
      }%
    }%
    \rule{0pt}{#1\dimen0}%
    \setlength{\wd0 }{1em}%
    \setlength{\ht0 }{1em}%
    \rotatebox{#2}{\usebox{0}}%
  \endgroup
}

\newenvironment{fcaption}{\begin{list}{}{
\setlength{\leftmargin}{35pt}
\setlength{\rightmargin}{35pt}
\setlength{\labelsep}{5pt}
}}{\end{list}}

\newtheorem{thm}{Theorem}[subsection]
\newtheorem{prop}[thm]{Proposition}

\newtheorem{lem}[thm]{Lemma}

\theoremstyle{definition}
\newtheorem{defn}[thm]{Definition}
\newtheorem*{ackn}{Acknowledgements}

\newtheorem{conv}[thm]{Convention}
\theoremstyle{remark}
\newtheorem*{rem}{Remark}

\makeatother

\usepackage{babel}

\global\long\def\sA{\mathcal{A}}
 \global\long\def\sB{\mathcal{B}}
 \global\long\def\sC{\mathcal{C}}
 \global\long\def\sD{\mathcal{D}}
 \global\long\def\sE{\mathcal{E}}
 \global\long\def\sF{\eF}
 \global\long\def\sG{\eG}
 \global\long\def\sH{\mathcal{H}}
 \global\long\def\sI{\mathcal{I}}
 \global\long\def\sJ{\mathcal{J}}
 \global\long\def\sK{\mathcal{K}}

 \global\long\def\sO{\mathcal{O}}
 \global\long\def\sP{\mathcal{P}}
 \global\long\def\sS{\mathcal{S}}
 
 \global\long\def\sQ{\mathcal{Q}}
 \global\long\def\sT{\mathcal{T}}
 
 \global\long\def\sV{\mathcal{V}}

 \global\long\def\mC{\mathbb{C}}

 \global\long\def\mP{\mathbb{P}}
 \global\long\def\mQ{\mathbb{Q}}
 \global\long\def\mZ{\mathbb{Z}}
 
 \global\long\def\Ima{\mathrm{Im}\,}

 \global\long\def\Bs{\mathrm{Bs}\,}

 \global\long\def\Hom{\mathrm{Hom}}
 \global\long\def\id{\mathrm{id}}
 \global\long\def\Ext{\mathrm{Ext}}

 \global\long\def\pr{\mathrm{pr}}
 \global\long\def\Sing{\mathrm{Sing}\,}
 \global\long\def\Supp{\mathrm{Supp}\,}
 \global\long\def\SL{\mathrm{SL}\,}
 
 \global\long\def\rank{\mathrm{rank}\,}

 \global\long\def\rG{\mathrm{G}}
 
 \global\long\def\nS{\textsc{S}}
 \global\long\def\nT{\textsc{T}}
 \global\long\def\nU{\textsc{U}}

\def\hW{{\mathaccent20W}}

\def\hB{{\mathaccent20B}}

\def\mylabel#1{\label{#1}}

\title{Derived categories of \\
linear sections of ${\ft S}^2 \mP^3$}

\begin{document}

\begin{center}
\textbf{\Large Derived Categories of 
Artin-Mumford double solids}
\par\end{center}{\Large \par}

$\;$

\begin{center}
Shinobu Hosono and Hiromichi Takagi 
\par\end{center}

\vspace{5pt}



$\;$

\begin{fcaption} {\small  \item 

Abstract. 
We consider 
the derived category of 
an Artin-Mumford quartic double solid 
blown-up at ten ordinary double points.
We show that it has a semi-orthogonal decomposition
containing the derived category of 
the Enriques surface of a Reye congruence.
This answers affirmatively  
a conjecture by Ingalls and Kuznetsov. 
}\end{fcaption}

\vspace{0.5cm}


\markboth
{Hosono and Takagi}{Derived categories of Artin-Mumford double solids}


\section{{\bf Introduction}}

\date{2015, 6/07, at home.}

Throughout this paper, we work over $\mC$, the complex number field.
We fix a four dimensional vector space $V$, and denote by $V^*$ the dual vector space of $V$,
and by $\mP(V)$
the projectivization of $V$.

\subsection{Classical backgrounds and motivations}

An {\it Artin-Mumford quartic double solid} 
is the double cover $Y$ of the $3$-dimensional projective space
branched along a quartic surface.
$Y$ is singular at ten ordinary double points.
Artin and Mumford showed
that it is unirational but irrational \cite{AM}.
They established its irrationality 
by showing its desingularization has a nonzero torsion in the Brauer group.
In \cite{Co}, Cossec pointed out 
that an Artin-Mumford quartic double solid is paired with
the so-called {\it Enriques surface of a Reye congruence}, which we denote by $X$ in this paper. The history of Enriques surfaces of Reye congruences
goes back to the work by Reye \cite{Reye} in the 19th century,
and later they were intensively studied by Fano \cite{Fano1, Fano2}.

The relationship between $X$ and $Y$
can be described in the following diagram (the section \ref{LHPD}):
\begin{equation}
\label{eq:XYbasic}
\xymatrix{Z\ar[r]\ar[d] & Y\\
X\subset \mathrm{G}(2,4),}
\end{equation}
where $Z\to \mathrm{G}(2,4)$ is the blow-up along $X$ and
$Z\to Y$ is a generically conic bundle.
$X\subset \mathrm{G}(2,4)$ is called a congruence since
classically a congruence means  
a $2$-dimensional family of lines in $\mP^3$.
$X\subset \mathrm{G}(2,4)$
is called the {\it Fano model} of the Enriques surface of
a Reye congruence.
Beauville recalculated in \cite{Be2}
the torsion part of 
the Brauer group of 
$X$ by comparing $H^2(X,\mZ)$ and $H^3(\widetilde{Y},\mZ)$ by using the diagram
(\ref{eq:XYbasic}),
where $\widetilde{Y}$ is the blow-up of $Y$ at the ten ordinary double points.
 
Recently, Ingalls and Kuznetsov \cite[Thm.~4.3]{IK} showed
that the derived categories of 
$X$ and $\widetilde{Y}$ have respective
semi-orthogonal decompositions
with a certain triangulated subcategory in common.
They also conjectured [ibid., Conj.~4.2]
that 
the derived category of $\widetilde{Y}$ 
has a semi-orthogonal decomposition containing the derived category of $X$.
The main result of this paper is an affirmative
solution to this conjecture.

\subsection{Statement of the main result}
\label{state}
Let us denote by $y_1, \dots, y_{10}$ the
ten ordinary double points of $Y$ and 
by
$F_i\simeq \mP^1\times \mP^1$ $(i=1,\dots, 10)$
the exceptional divisors over $y_i$
of the blow-up $\widetilde{Y}\to Y$.
We also denote by $\iota_i\colon F_i\hookrightarrow \widetilde{Y}$
the closed embeddings.
We define a categorical resolution $\sD_{Y}$ 
of $\sD^b({Y})$ 
in the sense of Kuznetsov \cite{Lef}.
For this, 
we
choose a dual Lefschetz decomposition of $\sD^b(F_i)$;
\[
\sD^b(F_i)=\langle \sA_1(-1,-1), \sA_0\rangle,
\]
where 
\[
\text{$\sA_0=\langle \sO_{F_i}(0,-1), \sO_{F_i}(-1,0), \sO_{F_i}\rangle$,
$\sA_1=\langle \sO_{F_i}\rangle$},
\]
with $(-1,-1)$ meaning the twist by $\sO_{F_i}(-1,-1)$.
The categorical resolution 
\[
\sD_{Y}\subset \sD^b(\widetilde{Y})
\] 
of $\sD^b(Y)$
with respect to this dual Lefschetz decomposition is defined to be 
the left orthogonal to the subcategory
$\langle \{{\iota_i}_*\sO_{F_i}(-1,-1)\}_{i=1}^{10}\rangle.$ 
Therefore we have the following semi-orthogonal decomposition of 
$\sD^b(\widetilde{Y})$:
\begin{equation*}
\label{eq:DY}
\sD^b(\widetilde{Y})=
\langle \{{\iota_i}_*\sO_{F_i}(-1,-1)\}_{i=1}^{10},\sD_Y\rangle.
\end{equation*}
Let 
$\sO_{\widetilde{Y}}(1)$ be the pull-back of $\sO_{\mP^3}(1)$
by the composite $\widetilde{Y}\to Y\to \mP^3$.

\begin{thm}
\label{thm:main}
Suppose that the branch locus $H$ of $Y\to \mP^3$ does not contain a line.
Then there exists a fully faithful Fourier-Mukai functor $\Phi_1\colon \sD^b(X)\to \sD_{Y}$ 
giving the following semi-orthogonal decomposition of the categorical resolution $\sD_Y$\,$:$ 
\begin{equation*}
\label{eq:main1}
\sD_{Y}=\langle \Phi_1(\sD^b(X)), \sO_{\widetilde{Y}}(1), \sO_{\widetilde{Y}}(2)\rangle.
\end{equation*} 
\end{thm}
The method of the proof of Theorem \ref{thm:main} is similar to that of our previous work \cite{HoTa4}, which is based on our fundamenatal papers \cite{Geom, DerSym}.
By our method, we also reproduce \cite[Thm.~4.3]{IK}
in the subsection \ref{sub:IK}.

Recently, in \cite{IKP},
the authors studied
smooth projective $3$-folds
whose derived categories include
those of Enriques surfaces, and discuss
irrationality of such $3$-folds in view of the theory of homological mirror symmetry. 
Theorem \ref{thm:main} and the affirmative solution to \cite[Conj.~1]{Or2}
would imply immediately that
the motif of $X$ with rational coefficient 
is a direct summand of 
that of $\widetilde{Y}$.

The key idea to our proof of Theorem \ref{thm:main}
is to construct the kernel of
the Fourier-Mukai functor $\Phi_1$.
Taking  
the images by $Z\to Y$ of fibers of
the blow-up $Z\to \mathrm{G}(2,4)$ and modifying them slightly on $\widetilde{Y}$,
we may construct a family of curves in $\widetilde{Y}$
parameterized by $X$ ($\Delta_1\to X$ defined in the subsection \ref{sub:IK}).
A naive candidate of the kernel
is the ideal sheaf in $\sO_{\widetilde{Y}\times X}$
of this family of curves.
It, however, turns out that this does not give a fully faithful functor.
Our kernel of $\Phi_1$ is a certain modification of this ideal sheaf.

\subsection{Orthogonal linear section and homological projective duality}
\label{sub:intro3}
In our proof of Theorem \ref{thm:main},
we compute the images of several objects by the functor $\Phi_1$
along the theory of {\it homological projective duality}
due to Kuznetsov in \cite{HPD1}.

Let us consider the pair $X$ and $Y$ from this viewpoint.
Then $X$ is a subvariety of the second symmetric product $\chow:={\ft S}^2 \mP(V)$ of $\mP(V)$ embedded by the Chow form into $\mP({\ft S}^2 V)$.
$Y$ is a subvariety of the double cover $\hcoY$ of the dual projective
space $\mP({\ft S}^2 V^*)$ branched along the quartic hypersurface which is
the locus of corank $\geq 1$ quadrics in $\mP(V)$.
We have studied $\chow$ and $\hcoY$ as (double covers of) symmetric 
determinantal loci and obtained their basic results in \cite{Geom}.
For $X$ and $Y$, there exists a four-dimensional vector subspace
$L_4\subset {\ft S}^2 V^*$ such that $Y$ is the pull-back of
$P_3:=\mP(L_4)\simeq \mP^3$ by the double covering $\hcoY\to \mP({\ft S}^2 V^*)$,
and $X=\chow\cap \mP(L_4^{\perp})$,
where $L_4^{\perp}$ is the orthogonal space to $L_4$
with respect to the dual pairing of ${\ft S}^2 V$ and ${\ft S}^2 V^*$. 
We remark that the embedding $X\subset \mP(L_4^{\perp})\simeq \mP^5$ is
different from the Fano model;
it is called the {\it Cayley model} of the Reye congruence.
From this viewpoint, the pair $X$ and $Y$ is an example of 
{\it orthogonal linear sections}, which were initially 
employed systematically by Mukai \cite{Mukai} in his studies of Fano $3$-folds and $K3$ surfaces.

The main result and the constructions in this paper
may be regarded a step to establish
that
$\chow$ and $\hcoY$ are homologically projective dual.
We refer the precise definition of homological projective duality to 
\cite{HPD1}.
Here we only mention that,
if two varieties $\Sigma$ and $\Sigma^*$ 
with respective morphisms to the dual projective spaces $\mP^N$ and $(\mP^N)^*$
are homological projective dual,
then the relationships between the derived categories of {\it any} orthogonal linear sections are established simultaneously.
Therefore, once $\chow$ and $\hcoY$ are shown to be homologically projective dual to each other,
Theorem \ref{thm:main} will follow immediately from [ibid., Thm.~6.3].

We construct the kernel of the functor $\Phi_1$ by restricting 
to $\widetilde{Y}\times X$ 
a certain rank two reflexive sheaf $\sP$
on $\Vs\subset \widetilde{\hcoY}\times \hchow$,
where $\widetilde{\hcoY}$ and $\hchow$ are
suitable desingularizations of 
$\hcoY$ and $\chow$ respectively constructed in \cite{Geom}, and $\Vs$ 
is the pull-back of the universal family of hyperplane sections.
The sheaf $\sP$ is constructed naturally from
the geometry of Grassmannians.
We also construct a locally free resolution
of ${\iota_{\Vs}}_*\sP$ (Theorem \ref{thm:newker}),
where $\iota_{\Vs}$ is the natural closed embedding 
$\Vs\hookrightarrow \widetilde{\hcoY}\times \hchow$.
This resolution enables us to 
compute the images of several objects by the functor $\Phi_1$.

The definition of homological projective duality 
requires
special types of semi-orthogonal decomposition
called (dual) Lefschetz decompositions in 
(noncommutative resolutions of) the derived categories
of varieties.
As a strong evidence of 
the homological projective duality between $\hcoY$ and $\chow$,
we can read off (dual) Lefschetz collections
in the derived categories of $\widetilde{\hcoY}$ and $\hchow$
from the locally free resolution of ${\iota_{\Vs}}_*\sP$
(\cite[Cor.~3.3 and 5.11]{DerSym}).

\subsection{Relations with previous works}
Homological projective duality is 
a powerful guiding principle in describing derived categories
of projective varieties but it is often hard to show a plausible candidate
of a pair of variety $\Sigma$ and $\Sigma^*$ is actually homologically projective dual.
On the contrary, for such a candidate $\Sigma$ and $\Sigma^*$,
it is usually less hard to establish the relationships between
the derived categories of individual orthogonal linear sections.
One such example is the Grassmann-Pfaffian derived equivalence \cite{BC}.
We gave another example in \cite{HoTa4} after the fundamental works
\cite{HoTa1}--\cite{DerSym}, 
by showing the derived equivalence between
the smooth Calabi-Yau threefold of a Reye congruence and
its orthogonal linear section which is also a Calabi-Yau threefold. 

\subsection{Structure of the paper}
We state general results in the section \ref{basic}.
In the section \ref{Orthogonal},
we review the constructions of (double) symmetric loci and
introduce $X$ and $Y$ as orthogonal linear sections of 
$\chow$ and $\hcoY$.
In the section \ref{Orthogonal}, we also introduce another pair 
$W$ and $S$ of orthogonal linear sections,
where $W$ is a $3$-dimensional linear section of $\chow$,
and $S$ is a $2$-dimensional linear section of $\hcoY$.
The plausible homological projective duality of $\chow$ and $\hcoY$
indicates a close relation between the derived categories of 
$W$ and $S$. We establish this relation in the section \ref{section:mut} (Theorem \ref{thm:main2}). 
$W$ and $S$ also play crucial roles in our proof of Theorem \ref{thm:main}.
In the section \ref{BirYZ1},
we review the constructions of birational models of $\hcoY$ and 
introduce generically conic bundles over them, which are used in the construction of the sheaf $\sP$ in the section \ref{section:Corr}.
We obtain a locally free resolution of ${\iota_{\Vs}}_*\sP$
in the appendix \ref{app:B}.
In the section \ref{section:BC},
we prove Theorem \ref{thm:main} and give a new proof of
the result by Ingalls and Kuznetsov.

\begin{ackn}
We are grateful to Professor Igor Dolgachev for his comments on Reye congruences,
and to Professor Yukinobu Toda for his answer to our question about derived category. 

We have been involved for long in the research of orthogonal linear sections 
under strong influences of Professor Shigeru Mukai's work. 
We thank Professor Shigeru Mukai for valuable comments on the subjects
related to the orthogonal linear sections.
We would like to dedicate this article to him 
with our deep gratitude and admiration.

This paper is supported in part by 
Grant-in Aid for Young Scientists (B 20740005, H.T.). 
\end{ackn}

\begin{conv}
\label{conv:chow}
Throughout the article,
we consider several varieties $\Sigma$ with morphisms $\Sigma\to \mP({\ft S}^2 V)$, $\Sigma\to \mP({\ft S}^2 V^*)$, or $\Sigma\to \mathrm{G}(2,V)$.
We denote by (without suffix)
\[
H, M, L
\]
the pull-backs of 
$\sO_{\mP({\ft S}^2 V)}(1)$,
$\sO_{\mP({\ft S}^2 V^*)}(1)$,
and
$\sO_{\mathrm{G}(2,V)}(1)$, respectively
if 
$\Sigma$ has a morphism to $\mP({\ft S}^2 V)$, $\mP({\ft S}^2 V^*)$, 
or $\mathrm{G}(2,V)$.
\end{conv}

\section{{\bf{Basic general results}}}
\label{basic}

\begin{thm}{\rm {\bf (Grothendieck-Verdier duality)}}
\label{cla:duality}
Let $f\colon X\to Y$ be a proper morphism of smooth varieties $X$~and~$Y$.
Set $n:=\dim X-\dim Y$. We have the following functorial isomorphism\,$:$
For $\mathcal{F}^{\bullet}\in\sD^{b}(X)$ and $\mathcal{E}^{\bullet}\in\sD^{b}(Y)$,
\[
Rf_{*}R\sH om(\mathcal{F}^{\bullet},Lf^{*}\sE^{\bullet}\otimes\omega_{X/Y}[n])\simeq R\sH om(Rf_{*}\mathcal{F}^{\bullet},\mathcal{E}^{\bullet}).\]
 In particular, if $\mathcal{E}^{\bullet}$ and $\mathcal{F}^{\bullet}$
are locally free $($we write them simply $\mathcal{E}$ and~$\mathcal{F})$
and if $R^{\bullet}f_{*}\mathcal{F}=f_*\mathcal{F}$, then \[
R^{\bullet+n}f_{*}(\mathcal{F}^{*}\otimes f^{*}\sE\otimes\omega_{X/Y})\simeq\sE xt^{\bullet}(f_{*}\mathcal{F},\mathcal{E}).\]

\end{thm}

\begin{proof} See \cite[Thm.~3.34]{Huy}. \end{proof}

For our proof of the fully faithfulness of $\Phi_1$
in Theorem \ref{thm:main}, we use the following fundamental result:

\begin{thm} \label{thm:D} Let $X$ and $Y$ be smooth projective
varieties and $\sP$ a coherent sheaf on $X\times Y$ flat over $X$.
Then the Fourier-Mukai functor $\Phi_{\sP}\colon\sD^{b}(X)\to\sD^{b}(Y)$
with the kernel $\sP$ 
is fully faithful
if and only if the following two conditions are satisfied\,$:$ \global\long\def\labelenumi{\textup{(\roman{enumi})}}
 
\begin{enumerate}
\item For any point $x\in X$, it holds $\Hom(\sP_{x},\sP_{x})\simeq\mC$,
and 
\item if $x_{1}\not=x_{2}$, then $\Ext^{i}(\sP_{x_{1}},\sP_{x_{2}})=0$
for any $i$. 
\end{enumerate}
Moreover, under these conditions, $\Phi_{\sP}$ is an equivalence
of triangulated categories if and only if $\dim X=\dim Y$ and $\sP\otimes\pr_{1}^{*}\omega_{X}\simeq\sP\otimes\pr_{2}^{*}\omega_{Y}$.

In particular, if $\dim X=\dim Y$, $\omega_{X}\simeq\sO_{X}$ and
$\omega_{Y}\simeq\sO_{Y}$, then $\Phi_{\sP}$ is fully faithful if
and only if it is an equivalence. \end{thm}

\begin{proof} See \cite[Thm.~1.1]{BO}, \cite[Thm.~1.1]{B},
\cite[Cor.~7.5 and Prop.~7.6]{Huy}. \end{proof}

We also need the following results of 
the derived categories of del Pezzo surfaces:
\begin{thm}
\label{thm:delPezzo}
Any exceptional collection in the derived category of
a del Pezzo surface is contained in a full exceptional collection,
and any full exceptional collection has the same length.
Moreover, any exceptional object is a locally free sheaf.
\end{thm}
\begin{proof}
See \cite{Ru}, \cite[Thm.~6.11, \S 7]{KO}.
\end{proof}

\section{{\bf Orthogonal linear sections of symmetric determinantal loci}}
\label{Orthogonal}
\subsection{Symmetric determinantal loci}
We quickly review some basic definitions and properties of
symmetric determinantal loci from \cite{Geom}.

\begin{defn}
We define $\nS_r\subset \mP({\ft S}^2 V^*)$ to be the locus 
of quadrics in $\mP(V)$ of rank at most $r$.
Taking a basis of $V$, $\nS_r$ is defined by $(r+1)\times (r+1)$ minors
of the generic $(n+1)\times (n+1)$ symmetric matrix.
We call $\nS_r$ the {\it symmetric determinantal locus of rank at most $r$}.
We call a point of $\nS_r\setminus \nS_{r-1}$ {\it a rank $r$ point}.

Similarly we define the symmetric determinantal locus $\nS_r^*$ in 
the dual projective space $\mP(\ft{S}^2V)$.   
\end{defn}

We have introduced the Springer type resolution $\Lp_{\widetilde{\nS}_r}\colon \widetilde{\nS}_r\to \nS_r$, which is a projective bundle over $\rG(n+1-r,V)$.
Using this, we have derived several properties of $\nS_r$,
for which we refer to \cite[\S 2.1]{Geom}.

In case $r$ is even, we have defined the double cover $\nT_r$ of $\nS_r$
by the following construction.

Let \begin{equation}
0\to{\eQ}_{\frac r2}^{*}\to V^{*}\otimes\sO_{\mathrm{G}(n-\frac r2+1,V)}\to\eS_{n-\frac r2+1}^{*}\to0\label{eq:Q*S}\end{equation}
 be the dual of the universal exact sequence on $\mathrm{G}(n-\frac r2+1,V)$,
where $\eQ_{\frac r2}$ is the universal quotient bundle of rank $\frac r2$
and $\eS_{n-\frac r2+1}$ is the universal subbundle of rank $n-\frac r2+1$. 
For brevity, we often omit the subscripts writing them by $\eS$ and $\eQ$.
Taking the second symmetric product, we obtain the following surjection: $\ft{S}^{2}V^{*}\otimes\sO_{\mathrm{G}(n-\frac r2-1,V)}\to\ft{S}^{2}\eS^{*}.$
Let $\sE^{*}$ be the kernel of this surjection, and consider the
following exact sequence: \begin{equation}
0\to\sE^{*}\to\ft{S}^{2}V^{*}\otimes\sO_{\mathrm{G}(n-\frac r2+1,V)}\to\ft{S}^{2}\eS^{*}\to0.\label{eq:sE0}\end{equation}
Now we set 
\[
\nU_r:=\mP(\sE^{*}),
\]
and denote by $\Lrho_{\nU_r}$ the projection
$\nU_r\to\mathrm{G}(n-\frac r2+1,V)$. By (\ref{eq:sE0}), $\nU_r$
is contained in $\mathrm{G}(n-\frac r2+1,V)\times\mP(\ft{S}^{2}V^{*})$.
Considering $\mP(\ft{S}^{2}V^{*})$
as the parameter space of quadrics in $\mP(V)$,
we see that the fiber of $\sE^{*}$ over $[\Pi]\in \mathrm{G}(n-\frac r2+1,V)$ 
parameterizes quadrics
in $\mP(V)$ containing the $(n-\frac r2)$-plane $\mP(\Pi)$.
Therefore \[
\nU_r=\{([\Pi],[Q])\mid\mP(\Pi)\subset Q\}\subset\mathrm{G}(n-\frac r2+1,V)\times\mP(\ft{S}^{2}V^{*}). \]
 Note that $Q$ in $([\Pi],[Q])\in\nU_r$ is a quadric of rank at most $r$
since quadrics contain $(n-\frac r2)$-planes only when their ranks are 
at most $r$. 
Hence the
symmetric determinantal locus $\nS_r$ is the image of the natural projection $\nU_r\to\mP(\ft{S}^{2}V^{*})$.
\begin{defn}
We let \[
\xymatrix{ & \nU_r\;\ar[r]^{\;\;\Lpi_{\nU_r}\;\;} & \;\nT_r\;\ar[r]^{\;\;\Lrho\,_{\nT_r}\;\;} & \;\nS_r}
\]
be the Stein factorization of $\nU_r\to\nS_r$. 
We have seen that $\nT_r\to \nS_r$ is a double cover branched along $\nS_{r-1}$.
We call $\nT_r$ {\it the double symmetric determinantal loci of rank at most $r$}.
\end{defn}

In this paper, we only consider $\nS^*_2$ and $\nT_4$ with $n=3$.
For them, we introduce the following set of notation:

\noindent
$\bullet\ \chow:=\nS^*_2,\, \hchow:=\widetilde{\nS}^*_2,\, f\colon \hchow\to \chow,\,
g\colon \hchow\to \rG(2,V), \sF:=\eS_2.$
We note that $\hchow\simeq \mP({\ft S}^2 \sF)$ (\cite[\S 2.1, \S 3.2]{Geom}),
and $\nS^*_1=v_2(\mP(V))$, where $v_2(\mP(V))$ is the second Veronese variety.

\noindent $\bullet \ \UU:=\widetilde{\nS}_4,\, \hcoY:=\nT_4,\, \Zpq:=\nU_4,$
$\xymatrix{\Zpq\;\ar[r]^{\;\;\Lpi_{\Zpq}\;\;} & 
\;\hcoY\;\ar[r]^-{\Lrho_{\hcoY}} & \;\nS_4=\mP({\ft S}^2 V^*)}.$
We maintain the notation $\nS_1,\nS_2,\nS_3\subset \mP({\ft S}^2 V^*)$.

\vspace{5pt}

We denote by $L_r$ a linear subspace of ${\ft S}^2 V^*$
of dimension $r+1$,
and by
$L_r^{\perp}$ the orthogonal space to $L_r$
with respect to the dual pairing of ${\ft S}^2 V$ and ${\ft S}^2 V^*$.
Then we consider the following mutually orthogonal subspaces:
\[
P_r:=\mP(L_r)\subset \mP({\ft S}^2 V^*),
\]
and
\[
P_r^{\perp}:=\mP(L_r^{\perp})\subset \mP({\ft S}^2 V).
\]
\subsection{Enriques surface $X$ and Artin-Mumford quartic double solid $Y$} 
\label{section:XY}

Let us take a $3$-plane 
$P_3\subset \mP({\ft S}^2 V^*)$
and set
\begin{align*}
X&:=\chow\cap P_3^{\perp},\\
Y&:= \text{the pull-back in $\hcoY$ of $P_3\simeq \mP^3$.}
\end{align*}
We say that $X$ and $Y$ are {\it orthogonal} to each other.

We put the following generality assumptions: 
\begin{equation}
\label{eq:XYcond}
\begin{cases}
\text{$X$ is a smooth surface,}\\
\text{$\nS_2\cap P_3$ consists of $10$ points 
in $\nS_2\setminus \nS_1$,}\\
\text{$\Sing (\nS_3\cap P_3)=\nS_3\cap P_3$,\ and}\\
\text{$\nS_3\cap P_3$ does not contain a line}.
\end{cases}
\end{equation}
Here we note that $\nS_2$ is a determinantal variety of degree 10
since $\nS_2={\ft S}^2 \mP(V^*)=\mP(V^*)\times \mP(V^*)/\mZ_2$.
We also
remark that the last condition of (\ref{eq:XYcond}) will be needed to
show the flatness of the kernel of $\Phi_1$
(Propositions \ref{lem:DeltaDflat} and \ref{prop:flatP1}).

$X$ is an Enriques surface of degree $10$.
Since $X$ is smooth,
$X$ is disjoint from $\Sing \hchow$, and hence
we can consider $X$ to be contained in ${{\hchow}}$.
As a subvariety of $\chow$ or $\hchow$,
$X$ parameterizes the sets of two distinct points $(x_1, x_2)$
of $\mP(V)$ such that $x_1$ and $x_2$ are mutually orthogonal 
with respect to all the quadratic forms corresponding to 
the quadrics in $P_3$.
Moreover, $X$ is mapped by $g\colon \hchow\to \mathrm{G}(2,V)$
onto its image isomorphically (we prove this in the proof of 
Proposition \ref{prop:flop} (1)), and hence 
we can also consider $X$ to be contained in $\mathrm{G}(2,V)$.
By the existence of this embedding into $\mathrm{G}(2,V)$,
$X$ is called the {\it Enriques surface of a Reye congruence}, or
simply a Reye congruence \cite{Co}.
In the proof of Proposition \ref{prop:flop} (1) below,
we will describe the lines in $\mP(V)$ which are parameterized by $X$.

Now we turn to $Y$.
Let us define \[
{\rm H}:=\nS_3\cap P_3.
\] 
\begin{prop}
\begin{enumerate}[$(1)$]
\item The singularities of ${\rm H}$ are $10$ ordinary
double points at $\nS_2\cap P_{3}$.
\item $y\in Y$ corresponds to a pair $(Q_{y},q_{y})$ of quadric
$Q_{y}$ in $P_{3}$ and a connected family $q_{y}$ of lines in $Q_{y}$
where $\Lq_{\hcoY}(y)=[Q_{y}]\in \mP({\ft S}^2 V^*)$.
\item
$Y$ is a del Pezzo threefold with only $10$ ordinary double points
at \[
\left\{ y_{1},...,y_{10}\right\} :=\nS_2\cap Y.\]
\end{enumerate}
\end{prop}
\begin{proof}
(1) follows since $P_{3}$ intersects transversely with
$\nS_2$ at 10 points in $\nS_2\setminus \nS_1$. 
(2) follows from the definition of $Y$ and the
Stein factorization $\Zpq\to\hcoY\to\mP({\ft S}^2 V^*)$. We have 
(3) since $Y$ is a double cover
of $P_{3}$ branched along ${\rm H}$. 
\end{proof}
$Y$ is called an {\it Artin-Mumford (AM) quartic double solid}, which is
discovered in \cite{AM} as an example of irrational unirational $3$-fold.

Let 
\[
Z:=\Lpi_{\Zpq}^{-1}(Y)\subset \Zpq.
\]
Then, by \cite[Prop.~3.7]{Geom},
the 2nd and 3rd assumptions of (\ref{eq:XYcond}),
$Z$ is a smooth fourfold, and
$Z\to Y$ is a $\mP^1$-fibration outside
$10$ points $y_1,\dots, y_{10}$ and
the fibers over these points are the unions of two planes 
\[
\mathrm{P}^{(1)}_i\cup_{1pt} \mathrm{P}^{(2)}_i.
\] 

As we see in the introduction,
the blow-up 
\[
\widetilde{Y}\to Y
\]
at $10$ ordinary double points
$y_1,\dots, y_{10}$
 plays a central role 
for Theorem \ref{thm:main}.

\vspace{5pt}

In the next subsection,
we introduce another pair of orthogonal linear sections $W$ and $S$
of $\chow$ and $\hcoY$ respectively,
which play important roles in the proof of Theorem \ref{thm:main}.

\subsection{Enriques-Fano threefold $W$ and del Pezzo surface $S$ of degree two}\label{subsection:SW}
Let us take a $2$-plane $P_2\subset \mP({\ft S}^2 V^*)$
and set
\begin{align*}
W&:=\chow\cap P_2^{\perp}, \\
S&:= \text{the pull-back in $\hcoY$ of $P_2\simeq \mP^2$.}
\end{align*}
We say that $W$ and $S$ are {\it orthogonal} to each other.

We put the following generality assumptions: 
\begin{equation}
\label{eq:P2}
\begin{cases}
\text{$\Sing \chow\cap P_2^{\perp}$ consists of $8$ points $w_1,\dots,w_8$,}\\
\text{$\Sing W=\Sing \chow\cap P_2^{\perp}$, and}\\
\text{$\Gamma:=\nS_3\cap P_2$ is a smooth plane quartic curve.}
\end{cases}
\end{equation}
Here we note that $\deg \Sing \chow=\deg v_2(\mP(V))=8$ (\cite[\S 2.1]{Geom}). 

If $\Sing \chow\cap P_2^{\perp}$
consists of $8$ points, then 
\[
\text{$W$ has only $\frac 12(1,1,1)$-singularities
at $w_1,\dots,w_8$}
\]
since $P_2^{\perp}$ cuts $\Sing \chow$ transversely and
$\chow$ has only $\frac 12(1,1,1)$-singularities along $\Sing \chow$
(\cite[\S 2.1]{Geom}). 
$W$ is an example of {\it Enriques-Fano threefolds}, which is a $\mQ$-Fano threefold containing an Enriques surface as a hyperplane section, and was initially studied by Fano \cite{Fano3}.
See \cite{Conte, CM, Ba, San1} for modern treatments of this class of $3$-folds.
We define 
\[
\text{$\hW:=$ the pull-back in $\hchow$ of $W$}.
\] 
$\hW$ is a smooth threefold
and the birational morphism $\hW\to W$ is the blow-up at $w_1,\dots,w_8$. 
By \cite[(2.4)]{Geom},
we derive 
\begin{equation}
\label{eq:2HL}
2(H-L)=\sum_{i=1}^{8} E_i
\end{equation}
on $\hW$ since the restriction of the $f$-exceptional divisor is 
$\sum_{i=1}^{8} E_i$.
$S$ is the double cover of $P_2$ branched along the smooth quartic curve $\Gamma$
and then is a smooth del Pezzo surface of degree two.

Let \[
Z_S:=\Lpi_{\Zpq}^{-1}(S)\subset \Zpq,
\]
which is a smooth threefold with
a $\mP^1$-bundle structure over $S$ by the choice of $P_2$ as
in (\ref{eq:P2}).
The fiber over a point $s\in S$ parameterizes
lines contained in the rank three or four quadric corresponding to the image of $s$ on $P_2$.  

\section{{\bf Derived categories of Enriques-Fano threefolds and degree two del Pezzo surfaces}}
\label{section:mut}

Let $W$ and $S$ be as in the subsection \ref{subsection:SW}.
In this section, we establish a relationship between
the derived categories of $W$ and $S$ as indicated by
the plausible homological projective duality of $\chow$ and $\hcoY$
(Theorem \ref{thm:main2}).
 
\subsection{Linear duality between $\hchow$ and $\Zpq$}
\label{LHPD}

We rewrite the exact sequence (\ref{eq:sE0}) by new notation:
\begin{equation}
0\to\sE^{*}\to\ft{S}^{2}V^{*}\otimes\sO_{\mathrm{G}(2,V)}\to\ft{S}^{2}\sF^{*}\to0.\label{eq:sE0new}\end{equation}
This means that
the fibers of ${\ft S}^2 \sF$ and $\sE$ over a point of $\mathrm{G}(2,V)$
are the orthogonal spaces to each other when we consider them
as subspaces in ${\ft S}^2 V$ and ${\ft S}^2 V^*$ respectively.
In this sense, the pair ${\ft S}^2 \sF$ and $\sE$ is an example of 
{\it orthogonal bundles}.

In \cite[\S 8]{HPD1}, Kuznetsov establishes 
the homological projective duality between
a projective bundle $\mP(\sV)$ over a smooth base $S$
and its orthogonal bundle $\mP(\sV^{\perp})$ for a 
globally generated vector bundle $\sV$ on $S$.
He call this duality {\it linear duality} in \cite{ICM}.
This situation is ubiquitous and is quite useful to understand several
relationships of derived categories.
We do not review his result but we show in the following proposition that 
this framework is suitable for describing the relationships
between $X=\mP({\ft S}^2 \sF)\cap f^{-1}(P_3^{\perp})$
and $Z=\mP(\sE)\cap (\Lrho_{\hcoY}\circ \Lpi_{\Zpq})^{-1}(P_3)$
(as in the subsection \ref{section:XY}),
and between 
$\hW=\mP({\ft S}^2 \sF)\cap f^{-1}(P_2^{\perp})$ 
and $Z_S=\mP(\sE)\cap (\Lrho_{\hcoY}\circ \Lpi_{\Zpq})^{-1}(P_2)$
(as in the subsection \ref{subsection:SW}),
respectively. In particular, in the assertion (1) of the following proposition,
we derive the classical diagram (\ref{eq:XYbasic}) from the framework
presented in the section \ref{Orthogonal}.
The equivalence $\sD^b(\hW)\simeq \sD^b(Z_S)$
given in the assertion (2) plays important roles
in our proofs of \ref{thm:main2} in the next subsection.

\begin{prop}
\label{prop:flop}
\begin{enumerate}[$(1)$]
\item
$X$ is mapped isomorphically onto its image by $g\colon \hchow\to \mathrm{G}(2,V)$. We denote by $X$ its image in $\mathrm{G}(2,V)$. 
$Z\to \mathrm{G}(2,V)$ is the blow-up along $X\subset \mathrm{G}(2,V)$ and
the fiber of $Z\to \mathrm{G}(2,V)$ over a point $[l]\in X$
represents exactly the pencil $\subset P_3$ of quadrics containing $l$.
Moreover,
$X$ may be characterized as a subvariety of $\mathrm{G}(2,V)\,;$
\begin{equation*}
\label{eq:subset1}
\text{$X=\{[l]\in \mathrm{G}(2,V)\mid$ quadrics $\in P_3$
containing $l$
form a pencil$\}$.}
\end{equation*}
\item
There exists the following diagram\,$:$
\[
\xymatrix{
& \hW\ar[dl]\ar@{-->}[r]^{\text{flop}} & Z_S\ar[dr]\\
W & & & S,}
\]
where $\hW\dashrightarrow Z_S$ is an Atiyah's flop, namely,
a flop whose exceptional curves are mutually disjoint, and
have normal bundles $\sO_{\mP^1}(-1)^{\oplus 2}$.
The flop induces an equivalence of the derived categories
$\sD^b(\hW)\simeq \sD^b(Z_S)$.
\end{enumerate}
\end{prop}

\begin{proof}
The second claim of (2) follows immediately from [ibid, Cor.~8.3]
or \cite{BO}. Therefore we have only to show (1) and the first claim of (2).

Let $s\in \mathrm{G}(2,V)$ be any point, and
$F_s$ and $E_s$ the fibers of ${\ft S}^2 \sF$ and $\sE$ at $s$, respectively.
By the exact sequence (\ref{eq:sE0new}) defining $\sE$,
we see that $F_s\subset {\ft S}^2 V$ and $E_s\subset {\ft S}^2 V^*$ 
are mutually orthogonal. Thus we will write $F_s=E^{\perp}_s$.
We recall that
$L_{r+1}$ is the $(r+1)$-dimensional subspace of ${\ft S}^2 V^*$ such that $\mP(L_{r+1})=P_r$.
Note that $s\in \mathrm{G}(2,V)$ is contained in 
the image of $(\Lrho_{\hcoY}\circ \Lpi_{\Zpq})^{-1}(P_r)$ under the morphism
$\Zpq\to \mathrm{G}(2,V)$ if and only if 
$\dim (E_s\cap L_{r+1})\geq 1$, and a similar assertion holds for $\hchow$.

We show the following key equality: 
\begin{equation}
\label{eq:perp}
\dim (E_s^{\perp} \cap L_{r+1}^{\perp})=2-r+\dim (E_s\cap L_{r+1}).
\end{equation}
Indeed, we have 
$\dim (E_s^{\perp} \cap L_{r+1}^{\perp})=\dim {\ft S}^2 V-\dim (E_s+L_{r+1})=
\dim {\ft S}^2 V-\dim E_s-\dim L_{r+1}+\dim (E_s\cap L_{r+1})=2-r+\dim (E_s\cap L_{r+1})$.

\vspace{5pt}

For (1), we set $r=3$.
Since $\dim (E_s^{\perp} \cap L_{4}^{\perp})\geq 0$ for any $s\in \mathrm{G}(2,V)$, we have $\dim (E_s\cap L_{4})\geq 1$ by (\ref{eq:perp}).
This implies that $Z\to \mathrm{G}(2,V)$ is surjective.
Therefore $Z\to \mathrm{G}(2,V)$ is generically finite since
$\dim Z=\dim \mathrm{G}(2,V)$, and moreover,
this is birational since the fiber of $Z\to \mathrm{G}(2,V)$ over $s$
is linear in $\mP(E_s)$, and hence is one point if it is $0$-dimensional.
Note that $Z\to \mathrm{G}(2,V)$ has a positive dimensional fiber
over $s$ if and only if 
$\dim (E_s\cap L_{4})\geq 2$.
By (\ref{eq:perp}),
this is equivalent to that
$\dim (E_s^{\perp} \cap L_{4}^{\perp})\geq 1$, namely,
$s$ is contained in the image of 
$X=\hchow\cap f^{-1}(L^{\perp}_4)$ under the morphism $\hchow\to \mathrm{G}(2,V)$.
We show that $\dim (E_s^{\perp} \cap L_{4}^{\perp})\geq 2$
is impossible.
Indeed, if $\dim (E_s^{\perp} \cap L_{4}^{\perp})\geq 2$,
then
$(E_s^{\perp} \cap L_{4}^{\perp})$
would intersect the $f$-exceptional divisor, 
and $X\subset \chow$ would be singular, a contradiction to the assumption in
(\ref{eq:XYcond}). 
Therefore $X\subset \hchow$ isomorphically mapped by $g$ onto its image (which we also denote by $X$) and any positive dimensional fiber of
$Z\to \mathrm{G}(2,V)$ is $1$-dimensional and then is a line.
Then
we conclude
$Z\to \mathrm{G}(2,V)$ is the blow-up along $X$ 
by \cite[Theorem 2.3]{A}. We have also proved
other assertions of (1).

\vspace{5pt}

For (2), we set $r=2$.
By (\ref{eq:perp}),
$\dim (E_s^{\perp} \cap L_3^{\perp})=\dim (E_s \cap L_3)$.
Therefore 
the images on $\mathrm{G}(2,V)$ of $\hW$ and 
$Z_S$ are equal, which we denote by $\overline{W}$,
and the dimensions of 
the fibers of $\hW\to \overline{W}$ and $Z_S\to \overline{W}$ over a point $s$
are the same. We show the assertion by several steps.

\vspace{5pt}

The arguments in the following steps are more or less standard
in explicit birational geometry. Here we give an outline.
In Steps 1--3, we will
show that $\hW\to \overline{W}$ and $Z_S\to \overline{W}$
are flopping contractions.
In the remaining steps (Steps 4--6), 
we will show 
$Z_S\to \overline{W}$ is actually an
Atiyah's flopping contraction.
Then 
so is $\hW\to \overline{W}$ 
by symmetry of a flop.
Let $E'_i\subset Z_S$ be the strict transform of $E_i$ ($1\leq i\leq 8$).
To show that 
$Z_S\to \overline{W}$ is an
Atiyah's flopping contraction,
a key point 
is showing that the map $Z_S\to S$ induces an isomorphism 
$E'_i\simeq S$. For this, it suffices to show that
the induced morphism $E'_i\simeq S$ is birational (Step 4) and finite (Step 5)
by the Zariski main theorem. 

\vspace{5pt}

\noindent {\bf Step 1.}
$\hW\to \overline{W}$ and $Z_S\to \overline{W}$ are birational and crepant.
\vspace{5pt}

Note that $-K_{\hW}=L$ and $-K_{Z_S}=L$
since $-K_{\hchow}=3H+L$ and 
$-K_{\Zpq}=6M+L$ by the canonical bundle formula of
projective bundle (we follow Convention \ref{conv:chow}).
By standard computation, we see that $(-K_{\hW})^3>0$ and $(-K_{Z_S})^3>0$. 
Therefore
both $\hW\to \overline{W}$ and 
$Z_S\to \overline{W}$ are generically finite and crepant, and 
by a similar argument to the one showing
that $Z\to \mathrm{G}(2,V)$ is birational,
we see that they are birational.

\vspace{5pt}

\noindent {\bf Step 2.}
Any non-trivial fibers of 
$\hW\to \overline{W}$ and $Z_S\to \overline{W}$ are copies of $\mP^1$.

\vspace{5pt}

Indeed, if otherwise, $\mP^2$ would appear as a non-trivial fiber of 
$\hW\to \overline{W}$. We may disprove this situation 
similarly to the argument in the proof of (1) above.

\vspace{5pt}

\noindent {\bf Step 3.}
$\hW\to \overline{W}$ and $Z_S\to \overline{W}$ are nontrivial
morphisms.

\vspace{5pt}

We have only to exhibit
positive dimensional fibers of $\hW\to \overline{W}$.
Note that the image on $W$ of each $f|_{\hW}$-exceptional divisor $E_i$
represents a double points $2w_i$ with some $w_i\in \mP(V)$.
Therefore the image $\overline{E}_i\subset \overline{W}$ of $E_i$
is equal to $\{[l]\in \mathrm{G}(2,V)\mid w_i\in l\}$.
For each distinct $i$ and $j$, $\overline{E}_i$ and $\overline{E}_j$
intersects at one point
$[l_{ij}]$, where $l_{ij}$ is the line joining $w_i$ and $w_j$.
On the other hand, $E_i$ and $E_j$ are disjoint.
Therefore $\hW\to \overline{W}$ has a positive dimensional fiber
over $[l_{ij}]$.

\vspace{5pt}

\noindent {\bf Step 4.}
$E'_i\to S$ is birational. 

\vspace{5pt}

Indeed, for the fiber $q$ 
of $Z_S\to S$ over a general rank four point $s$,
$E'_i\cap q$ represents lines 
on the rank four quadric
belonging to one connected family
and passing through $x_i$.
There is only one such a line, thus $E'_i\cap q$ consists of one point,
which means $E'_i\to S$ is birational.

\vspace{5pt}

\noindent {\bf Step 5.}
$E'_i\to S$ is finite.

\vspace{5pt}

Indeed, if otherwise, then $E'_i$ contains one fiber $q$ of 
$Z_S\to S$ over a point $s\in S$.
Since $q$ parameterizes lines on the quadric $Q$ corresponding to $s$ and
passing through $w_i$, $Q$ must be of rank three, and $w_i$ is the vertex of 
$Q$. By the orthogonality between $W$ and $P_2$,
$Q$ corresponds to a hyperplane section $H_Q\subset \mP({\ft S}^2 V)$
containing $W$. Since $w_i$ is the vertex of $Q$, 
$H_Q$ is tangent to
$v_2(\mP(V))$ at $w_i$ by 
the projective duality between 
$\Hes$ and $v_2(\mP(V))$.
However, by choosing two hyperplane $H_1$ and $H_2$ such that
$W=\chow \cap H_Q \cap H_1 \cap H_2$, we see that
$v_2(\mP(V))\cap H_Q\cap H_1\cap H_2$ is singular at $w_i$,
a contradiction to the assumption (\ref{eq:P2}).

\vspace{5pt}

\noindent {\bf Step 6.}
$Z_S\to \overline{W}$ is an
Atiyah's flopping contraction.

\vspace{5pt}

Since any flopping curve intersects some $E_i$ and 
is contained in no $E_i$'s,
any flopped curve is contained in some $E'_i$.
By Step 3, we see that $E'_i$ contains at least seven flopped curves
corresponding to seven flopping curves over $[l_{ij}]$'s with $j\not =i$.
Since $E'_i$ is a smooth del Pezzo surface of degree two by Steps 4 and 5,
we see that exceptional curves of $\hW\to \overline{W}$ are only $28$
flopping curves over $[l_{ij}]$'s, and the flopped curve over $[l_{ij}]$
is a $(-1)$-curve on $E'_i$. Then we see that its normal bundle on $Z_S$
is $\sO_{\mP^1}(-1)^{\oplus 2}$.  
\end{proof}

\begin{rem}~

\begin{enumerate}
\item
By Proposition \ref{prop:flop} (1), we immediately see that $\rho(Z)=2$.
Thus $Z\to Y$ is a Mori fiber space, and then
$Y$ is $\mQ$-factorial by \cite[Lem.~5-1-5]{KMM}
(see also \cite{E} for another proof of this fact.
We are grateful to Professor I.~Cheltsov for this information).
\item
The image on $S$ of the $28$ flopped curves
are mapped by the double cover $S\to P_2$ to
the famous $28$ bitangent lines of the quartic curve $\nS_3\cap P_2$.
\end{enumerate}
\end{rem}

\vspace{5pt}

\subsection{Derived categories of $W$ and $S$}

We have obtained the equivalence $\sD^b(\hW)\simeq \sD^b(Z_S)$
in Proposition \ref{prop:flop} (2).
We will establish the relation between $\sD^b(W)$ and $\sD^b(S)$
through this equivalence by relating $\sD^b(W)$ and $\sD^b(\hW)$,
and $\sD^b(Z_S)$ and $\sD^b(S)$, respectively.

The relation between 
$\sD^b(W)$ and $\sD^b(\hW)$ is given by 
a categorical resolution 
of $\sD^b({W})$.
Denote by $j_i\colon E_i\hookrightarrow \hW$
the closed embeddings.
We define a categorical resolution $\sD_{W}$ 
of $\sD^b({W})$ similarly to the case of $Y$ as in Theorem \ref{thm:main}.
We choose a dual Lefschetz decomposition of $\sD^b(E_i)$;
\[
\sD^b(E_i)=\langle \sB_1(-2), \sB_0\rangle,
\]
where 
\[
\sB_0=\langle \sO_{E_i}, \sO_{E_i}(1)\rangle,\,
\sB_1=\langle \sO_{E_i}\rangle.
\]
The categorical resolution 
\[
\sD_{W}\subset \sD^b(\hW)
\] 
of $\sD^b(W)$ with respect to this dual Lefschetz decomposition is defined to be 
the left orthogonal to the subcategory
$\langle\{ {j_i}_*\sO_{E_i}(-2)\}_{i=1}^{8}\rangle$. 
Therefore we have the following semi-orthogonal decomposition of 
$\sD^b(\hW)$:
\begin{equation}
\label{eq:DW}
\sD^b(\hW)=\langle \{ {j_i}_*\sO_{E_i}(-2)\}_{i=1}^{8},\sD_W\rangle.
\end{equation}

To establish the relation between $\sD^b(Z_S)$ and $\sD^b(S)$,
we review the results of 
the derived categories of quadric fibrations \cite{Quad}.
Let $\sC\ell_0$ and $\sC\ell_1$ be the sheaves of the even and the odd part of the Clifford algebra on $P_2$, respectively,
associated with the family of quadrics over $P_2$.
By [ibid.], 
$\sD^b(P_2,\sC\ell_0)$ admits the following semi-orthogonal decomposition:
\[
\sD^b(P_2,\sC\ell_0)=\langle \sD^b(w_1),\dots, \sD^b(w_8),\sC\ell_3,\sC\ell_4
\rangle,
\]
where $\sC\ell_3:=\sC\ell_1\otimes \sO_{P_2}(1)$ and
$\sC\ell_4:=\sC\ell_0\otimes \sO_{P_2}(2)$.

The relation between $\sD^b(Z_S)$ and $\sD^b(S)$
is given by the following proposition,
which can be shown by applying the same arguments as the proofs of 
\cite[Thm.~1.1]{Lines} and \cite[Thm.~5.5 and Cor.~5.8]{IK}
to $Z_S\to S$. The arguments are simplified since 
$Z_S\to S$ is a $\mP^1$-bundle. So we omit the proof of it.

\begin{prop}
\label{prop:cliff}
The pull-back functor $\sD^b(S)\to \sD^b(Z_S)$ is fully faithful.
We denote by $\langle \sC_S,\sO_{Z_S}(M)\rangle$
the image of $\sD^b(S)$ by the pull-back functor.
There exist a fully faithful functor $\sD^b(P_2,\sC\ell_0)\to \sD^b(Z_S)$
such that the image $\sS_3$ of $\sC\ell_3$ is isomorphic to
$\sG(M)$, and the image $\sS_4$ of $\sC\ell_4$
is a non-trivial extension of 
$\sO_{Z_S}(L+M)$ by $\sO_{Z_S}(2M)$.
Moreover, $\sD^b(Z_S)$ admits 
the following semi-orthogonal decomposition\,$:$
\begin{equation}
\label{eq:0st}
\sD^b(Z_S)=\langle \sC_S,\sO_{Z_S}(M), \sD^b(w_1),\dots, \sD^b(w_8), \sS_3,\sS_4\rangle,
\end{equation}
where we denote the image of $\sD^b(w_i)$ by the same symbol.
\end{prop}

We will compare the decompositions
(\ref{eq:DW}) and (\ref{eq:0st}) by the flop equivalence 
$\sD^b(\hW)\simeq \sD^b(Z_S)$ and series of mutations.
Then we obtain
\begin{thm}
\label{thm:main2}
We define the triangulated subcategories $\sC_S$ and $\sC_{W}$
of $\sD^b(S)$ and $\sD_{W}$, respectively 
by the following semi-orthogonal decompositions\,$:$
\begin{equation}
\label{eq:WS}
\begin{cases}
\sD^b(S)=\langle \sC_S, \sO_S(M)\rangle,\\
\sD_{W}=\langle \sO_{\hW}(-H), \sO_{\hW}(-L),
\sF, \sC_{W}\rangle,
\end{cases}
\end{equation}
where we follow Convention $\ref{conv:chow}$, and
we denote by $\sF$ the pull-back of $\sF$ by the morphism $g|_{\hW}$.
Then there exists an equivalence $\sC_{W}\simeq \sC_S$. 
\end{thm}

\noindent {\bf Proof of Theorem \ref{thm:main2}.}
Recall that $-K_{Z_S}=L$ (see Step 1 of the proof of Proposition \ref{prop:flop} (2)).
Then, by mutating 
\[
\sD^b(w_1),\dots, \sD^b(w_8), \sS_3,\sS_4
\] 
in (\ref{eq:0st}) to the left end, we obtain
\begin{equation}
\label{eq:1st}
\sD^b(Z_S)=\langle \sD^b(w_1),\dots, \sD^b(w_8), \sS_3(-L),\sS_4(-L),\sC_S,\sO_{Z_S}(M)\rangle,
\end{equation}
where we denote the image of $\sD^b(w_i)$ by the same symbol again.

Now we will construct another semi-orthogonal decomposition (\ref{eq:mutfin}) of $\sD^b(Z_S)$.
(\ref{eq:DW}) and
(\ref{eq:WS}) immediately give the following semi-orthogonal decomposition
of $\sD^b(\hW)$: 
\[
\sD^b(\hW)=\langle \{\sO_{E_i}(-2)\}_{1\leq i\leq 8}, 
\sO_{\hW}(-H), \sO_{\hW}(-L), \sF, \sC_{W}\rangle.
\]
By the flop equivalence $\sD^b(\hW)\simeq \sD^b(Z_S)$, 
we transform this decomposition to $\sD^b(Z_S)$ as follows:
\begin{lem}
\label{lem:ZSdec}
\begin{equation}
\label{eq:ZS}
\sD^b(Z_S)=\langle \delta, 
\sO_{Z_S}(-H), \sO_{Z_S}(-L), \sF, \sC_{Z_S}\rangle,
\end{equation}
where we denote by $\delta$ and $\sC_{Z_S}$ the images of 
the categories $\{\sO_{E_i}(-2)\}_{1\leq i\leq 8}$ and
$\sC_W$, respectively, and by $H$ the strict transform of $H$ on $Z_S$.
\end{lem}
\begin{proof}
We have only to verify
the images of
the sheaves $\sO_{\hW}(-H), \sO_{\hW}(-L), \sF$ are
the sheaves $\sO_{Z_S}(-H), \sO_{Z_S}(-L), \sF$, respectively.
The assertion is clear for $\sO_{\hW}(-L), \sF$ since they are the pull-backs
of the corresponding sheaves on $\overline{W}\subset \mathrm{G}(2,V)$.
Let $p\colon \widehat{W}'\to \widehat{W}$ be the blow-up
along the flopping curves and
$q\colon \widehat{W}'\to Z_S$ be the blow-up
along the flopped curves.
Let $l_i\subset \hW$ ($1\leq i\leq 28$)
be the flopping curves and $G_i$ the $p$-exceptional divisors over $l_i$.
Then we have $p^*H=q^*H-\sum_{i=1}^{28} G_i$ since $H\cdot l_i=1$.
Therefore we have
\[
R^{\bullet}q_*p^*\sO_{\hW}(-H)=
\sO_{Z_S}(-H)\otimes R^{\bullet}q_*\sO_{\widehat{W}'}(\sum _{i=1}^{28} G_i).
\]
This implies the assertion for $\sO_{\hW}(-H)$ since
$q_*\sO_{\widehat{W}'}(\sum _{i=1}^{28} G_i)\simeq \sO_{Z_S}$ and
\[
\text{$R^{\bullet}q_*\sO_{\widehat{W}'}(\sum _{i=1}^{28} G_i)=0$ for
$\bullet>0$}
\] by a standard fact on the blow-up.
\end{proof}

We will perform a series of mutations
starting from (\ref{eq:ZS}) (cf.~\cite[\S 5.3]{IK}).

\begin{lem}
\label{lem:HML}
It holds that $M=-H+3L$ on $\hW$ and $Z_S$.
\end{lem}

\begin{proof}
It suffices to show the assertion on $Z_S$ since
$\hW$ and $Z_S$ are isomorphic in codimension one.
Transforming (\ref{eq:2HL}) on $Z_S$,
we obtain
\begin{equation}
\label{eq:HLF}
2(H-L)=\sum_{i=1}^{8} E'_i,
\end{equation}
where $E'_i$ is the strict transform of $E_i$.
Let $q$ be a general fiber of $Z_S\to S$. Then 
$L\cdot q=-K_{Z_S}\cdot q=2$ since
$Z_S\to S$ is a $\mP^1$-bundle, and
$E'_i\cdot q=1$ for any $i$ since $E'_i$ is a section of $Z_S\to S$
(Steps 4 and 5 in the proof of Proposition \ref{prop:flop} (2)).
Thus, by (\ref{eq:HLF}),
we have $H\cdot q=6$, and then
$-H+3L$ is the pull-back of a divisor $D$ on $S$.
Since $E'_i\simeq S$,
we have $(-H+3L)|_{E'_i}=D$.
Note that the flop induces the blow-up $E'_i\to E_i\simeq \mP^2$ at seven points.
We denote by $l_1,\dots, l_7$ the exceptional curves, which are also flopped curves on $Z_S$. We also denote by $l$ the pull-back of a line on $\mP^2$.
Then $H\cdot l_i=-1$ and $L\cdot l_i=0$.
We also note that, on $\hW$,
$H|_{E_i}=\sO_{\mP^2}$ and $L|_{E_i}=\sO_{\mP^2}(1)$.
Therefore, on $Z_S$, we have
$H|_{E_i}=\sum_{i=1}^{7} l_i$ and 
$L|_{E_i}=l$.
Consequently,
we have $D=3l-\sum_{i=1}^{7} l_i=-K_S=M$.
\end{proof}

For simplicity of notation, we denote the images of $\delta$ and $\sC_{Z_S}$
by the same symbols in the sequel.
We consider the following sequence of mutations:

\begin{itemize}
\item
We mutate $\sF$ to the left of $\sO_{Z_S}(-L)$. Then we have
\[
\sD^b(Z_S)=\langle \delta, \sO_{Z_S}(-H),
\sS_3(H-5L), \sO_{Z_S}(-L), \sC_{Z_S}\rangle,
\]
where we note $\sG^*(-L)=\sS_3(H-5L)$ by Proposition \ref{prop:cliff}
and Lemma \ref{lem:HML}.
\item
We mutate the block $\delta$ to the right of
$\sO_{Z_S}(-H)$. Then we have
\[
\sD^b(Z_S)=\langle \sO_{Z_S}(-H), \delta,
\sS_3(H-5L), \sO_{Z_S}(-L), \sC_{Z_S}\rangle.
\]
\item
We mutate $\sO_{Z_S}(-H)$ to the right end. Then, by $-K_{Z_S}=L$, we have
\[
\sD^b(Z_S)=\langle \delta,
\sS_3(H-5L), \sO_{Z_S}(-L), \sC_{Z_S}, \sO_{Z_S}(-H+L)\rangle.
\]
\item
We mutate $\sC_{Z_S}$ to the right of $\sO_{Z_S}(-H+L)$. Then we have
\[
\sD^b(Z_S)=\langle \delta,
\sS_3(H-5L), \sO_{Z_S}(-L), \sO_{Z_S}(-H+L),\sC_{Z_S}\rangle.
\]
\item
We mutate $\sO_{Z_S}(-H+L)$ to the left of 
$\sO_{Z_S}(-L)$. 
We will show
\begin{equation}
\label{eq:left}
\mathbb{L}_{\sO_{Z_S}(-L)} \sO_{Z_S}(-H+L)=\sS_4(H-5L).
\end{equation}
Since
$\sS_4(H-5L)$ is a nontrivial extension 
of $\sO_{Z_S}(-L)$ by $\sO_{Z_S}(-H+L)$ \cite[Cor.~3.3]{Lines},
it suffices to show
\begin{equation}
\label{eq:H-F}
\text{$H^{\bullet}(Z_S, \sO_{Z_S}(H-\sum_{i=1}^{8} E'_i))=0$ for $\bullet\not =1$
and $\simeq \mC$ for $\bullet=1$}
\end{equation}
since $-H+2L=H-\sum_{i=1}^{8} E'_i$ by (\ref{eq:HLF}).
Consider the exact sequence;
\[
0\to \sO_{Z_S}(H-\sum_{i=1}^{8} E'_i)\to
\sO_{Z_S}(H)\to \oplus_{i=1}^{8} \sO_{E'_i}\to 0.
\]
By a similar argument to the proof of Lemma \ref{lem:ZSdec},
we have
$H^{\bullet}(Z_S,\sO_{Z_S}(H))\simeq H^{\bullet}(\hW,\sO_{\hW}(H))$,
where the latter is zero for $\bullet>0$, and is $7$-dimensional for $\bullet=0$. 
We also see that
$H^0(Z_S, \sO_{Z_S}(H-\sum_{i=1}^{8} E'_i))=0$.
Indeed, it suffices to show 
$H^0(\hW, \sO_{\hW}(H-\sum_{i=1}^{8} E_i))=0$ since $Z_S$ and $\hW$ are
isomorphic in codimension one.
If $H^0(\hW, \sO_{\hW}(H-\sum_{i=1}^{8} E_i))\not =0$,
then there exists a hyperplane in $P_2^{\perp}\simeq \mP^6$
passing through $w_1,\dots, w_8$, which is a contradiction since
they are linearly independent.

Therefore, since 
$H^0(E'_i, \sO_{E'_i})\simeq \mC$ and $H^{\bullet}(E'_i, \sO_{E'_i})=0$ for
$\bullet>0$, we obtain (\ref{eq:H-F}).
Thus we have shown (\ref{eq:left}) and we obtain
\[
\sD^b(Z_S)=\langle \delta,
\sS_3(H-5L), \sS_4(H-5L), \sO_{Z_S}(-L), \sC_{Z_S}\rangle.
\]
\item
We mutate $\sC_{Z_S}$ to the left of $\sO_{Z_S}(-L)$. Then we have
\[
\sD^b(Z_S)=\langle \delta,
\sS_3(H-5L), \sS_4(H-5L), \sC_{Z_S}, \sO_{Z_S}(-L)\rangle.
\]
\item
Finally, we twist $-H+4L$. Then, by Lemma \ref{lem:HML}, we obtain
\begin{equation}
\label{eq:mutfin}
\sD^b(Z_S)=\langle \delta,
\sS_3(-L), \sS_4(-L), \sC_{Z_S}, \sO_{Z_S}(M)\rangle.
\end{equation}

\end{itemize}  

Now, by comparing (\ref{eq:1st}) and (\ref{eq:mutfin}),
we obtain an equivalence $\sC_W\simeq \sC_{Z_S}\simeq \sC_S$ as desired.
We have finished our proof of Theorem \ref{thm:main2}.
$\hfill\square$

\section{{\bf Birational geometry of $\hcoY$ and $\Zpq$}}
\label{BirYZ1}

\subsection{Birational geometry of $\hcoY$}

We quickly review the birational geometry of $\hcoY$
obtained in \cite[\S 4]{Geom} with slightly different notation.

Let $\widehat{\hcoY}:=\mathrm{G}(3,\wedge^2 V)$
and $\Prt_{\rho},\, \Prt_{\sigma}\subset \widehat{\hcoY}$
the smooth subvarieties parameterizing $\rho$-planes and $\sigma$-planes
respectively in $\rG(2,V)\subset \mP(\wedge^2 V)$
(see \cite[\S 4.1]{Geom} for the definitions of 
$\rho$-planes and $\sigma$-planes).
In \cite[\S 4.1, \S 4.5]{Geom},
we have shown that $\Prt_{\rho}\simeq \mP(V)$ and $\Prt_{\sigma}\simeq \mP(V^*)$. 
We denote by $\hcoY_0$
the Hilbert scheme of conics in $\rG(2,V)$.

\begin{thm}
There is a commutative diagram of birational maps as follow\,$:$
\[
\xymatrix{&   \hcoY_0\ar[d]\\
& \widetilde{\hcoY}\ar[dl]_{\Lp_{\widetilde{\hcoY}}}\ar[dr]^{\Lrho_{\widetilde{\hcoY}}} & \\
 \widehat{\hcoY} & &\hcoY,}
\]
where
\begin{itemize}
\item
$\Lp_{\widetilde{\hcoY}}\colon \widetilde{\hcoY}\to \widehat{\hcoY}$
is the blow-up along $\Prt_{\rho}$, for which we denote by $F_{\rho}$ the exceptional divisor,
\item
$\Lrho_{\widetilde{\hcoY}}\colon \widetilde{\hcoY}\to \hcoY$ is an extremal divisorial contraction, which is described in detail in \cite[\S 4.6]{Geom},
\item
$\hcoY_0\to \widetilde{\hcoY}$ is the blow-up 
along the strict transform of $\Prt_{\sigma}$.
\end{itemize}
\end{thm}

\vspace{5pt}

In the subsequent part of this section, 
we will construct generically conic bundles
$\widehat{\Zpq}\to \widehat{\hcoY}$ and
$\widetilde{\Zpq}\to \widetilde{\hcoY}$,
which are birational to
$\Zpq\to \hcoY$.
These bundles lead us to construct the kernel
of the sheaf $\sP$ (mentioned in the subsection \ref{sub:intro3}) with locally free resolutions
in the next section \ref{section:Corr} and Appendix \ref{app:B}.

\subsection{Generically conic bundle $\Lpi_{\widehat{\Zpq}}\colon \widehat{\Zpq}\to \widehat{\hcoY}$}

We construct a family $\widehat{\Zpq}$ of
$\tau$-conics, and $\rho$- and $\sigma$-planes
over $\widehat{\hcoY}$
(for the definitions of $\tau$-, $\rho$-, $\sigma$-conics,
we refer to \cite[\S 4.1]{Geom}).
 
Let $\sS$ be the universal subbundle of rank three on $\widehat{\hcoY}$.
We consider
\[
\mP(\sS)
=\{([\mathrm{P}],t)\mid t\in \mathrm{P}\}\subset
\widehat{\hcoY}\times \mP(\wedge^2 V).
\] 
The natural map ${\Lpi_{\mP(\sS)}}\colon \mP(\sS) \to 
\widehat{\hcoY}$
is nothing but the universal 
family of planes in $\mP({\wedge^2} V)$;
\begin{equation}
\label{eq:diag3}
\xymatrix{
& \mP(\sS)\ar[dl]_{\Lrho_{\mP(\sS)}}\ar[dr]^{\Lpi_{\mP(\sS)}}&\\
{\mP({\wedge^2 V})} & & \widehat{\hcoY}.}
\end{equation}

We restrict the diagram (\ref{eq:diag3})
to $\mathrm{G}(2,V)\subset \mP({\wedge^2} V)$ and
set \[
\widehat{\Zpq}:=\mP(\sS)\cap (\widehat{\hcoY}\times \mathrm{G}(2,V))
\subset
\widehat{\hcoY}\times \mathrm{G}(2,V),
\]
Then we obtain 
\begin{equation}
\label{eq:diagr4}
\xymatrix{
& {\widehat{\Zpq}}\ar[dl]_{\Lrho_{\widehat{\Zpq}}}\ar[dr]^{\Lpi_{\widehat{\Zpq}}}&\\
{\mathrm{G}(2,V)} & & \widehat{\hcoY},}
\end{equation}
which is
clearly a family of
$\tau$-conics, and $\rho$- and $\sigma$-planes
over $\widehat{\hcoY}$.

Let $\sQ$ be the universal quotient bundle of rank three on $\widehat{\hcoY}$.
We note that
\[
H^0(\widehat{\hcoY}\times \mathrm{G}(2,V),
{\sQ} 
\boxtimes \sO_{\mathrm{G}(2,V)}(1))
\simeq \wedge^2 V\otimes
\wedge^2 V^*
\simeq
\Hom(\wedge^2 V, \wedge^2 V).
\]
Therefore
$H^0(\widehat{\hcoY}\times \mathrm{G}(2,V), 
{\sQ} 
\boxtimes \sO_{\mathrm{G}(2,V)}(1))$
has a unique nonzero $\SL(V)$-invariant section
up to constant corresponding to  
the identity of $\Hom(\wedge^2 V, \wedge^2 V)$.
Since
\[
\widehat{\Zpq}=\{([\mathrm{P}],[l])\mid [l]\in \mathrm{P}\}\subset
\widehat{\hcoY}\times \mathrm{G}(2,V),
\]
where $l$ is a line in $\mP(V)$ and $\mathrm{P}$ is a plane in
$\mP(\wedge^2 V)$,
it is standard to see the following proposition, for which we omit
a proof:
\begin{prop}\label{prop:ci}
$\widehat{\Zpq}$ is the complete intersection in
$\widehat{\hcoY}\times \mathrm{G}(2,V)$ with respect to
the unique nonzero $\SL(V)$-invariant section
of $H^0(\widehat{\hcoY}\times \mathrm{G}(2,V), 
{\sQ} 
\boxtimes \sO_{\mathrm{G}(2,V)}(1))$
up to constant.
\end{prop}
We set \[
\Zpq_{\rho}:=\Lpi_{\widehat{\Zpq}}^{-1}(\Prt_{\rho})\simeq
\mP(\sS|_{\Prt_{\rho}}),\
\Zpq_{\sigma}:=\Lpi_{\widehat{\Zpq}}^{-1}(\Prt_{\sigma})\simeq
\mP(\sS|_{\Prt_{\sigma}}).
\]
Then 
the subfamilies $\Zpq_{\rho}\to \Prt_{\rho}$
and
$\Zpq_{\sigma}\to \Prt_{\sigma}$
are the universal family of 
$\rho$- and $\sigma$-planes respectively.

We also prepare
some properties of $\widehat{\Zpq}$ for later use.
\begin{prop}
\mylabel{cla:G25}
The fiber of $\Lrho_{\widehat{\Zpq}}\colon \widehat{\Zpq}\to \mathrm{G}(2,V)$
over a point $[V_2]\in \mathrm{G}(2,V)$
parameterizes planes in $\mP(\wedge^2 V)$ containing $[\wedge^2 V_2]$.
In particular, $\Lrho_{\widehat{\Zpq}}$ is a $\mathrm{G}(2,5)$-bundle
and $\widehat{\Zpq}$ is smooth.
\end{prop}
This assertion is almost clear, so we omit a proof.

Restricting the diagram (\ref{eq:diagr4}) over $\Prt_{\rho}$, we have
\begin{equation*}
\xymatrix{
& {{\Zpq}_{\rho}}\ar[dl]_{\Lrho_{{\Zpq}_{\rho}}}\ar[dr]^{\Lpi_{{\Zpq}_{\rho}}}&
\\ {\mathrm{G}(2,V)} & & \Prt_{\rho},}
\end{equation*}
where 
we set 
\[
\Lrho_{\Zpq_{\rho}}:=\Lrho_{\widehat{\Zpq}}|_{\Zpq_{\rho}},\,
\Lpi_{\Zpq_{\rho}}:=\Lpi_{\widehat{\Zpq}}|_{\Zpq_{\rho}}.
\]

\begin{prop}
\label{cla:P1bdl}
$\Lrho_{\Zpq_{\rho}}
\colon \Zpq_{\rho}\to \mathrm{G}(2,V)$
is a $\mP^1$-bundle.
As a subbundle of $\Lrho_{\widehat{\Zpq}}\colon \widehat{\Zpq}\to \mathrm{G}(2,V)$,
a fiber of $\Lrho_{\Zpq_{\rho}}$ is a conic in $\mathrm{G}(2,5)$.

\end{prop}
\begin{proof}
Let $[\wedge^2 V_2]$ be a point of $\mathrm{G}(2,V)$.
Then a $\rho$-plane ${\rm P}_{V_1}$ contains 
$[\wedge^2 V_2]$ if and only if $V_1\subset V_2$.
Therefore $\Lrho_{\Zpq_{\rho}}^{-1}([\wedge^2 V_2])$
is isomorphic to the line $\mP(V_2)\subset \mP(V)\simeq \Prt_{\rho}$,
and this is a conic 
in $\mathrm{G}(2,\wedge^2 V/\wedge^2 V_2)$
since $\Prt_{\rho}= v_2(\mP(V))$.
\end{proof}

\subsection{Generically conic bundle $\Lpi_{\widetilde{\Zpq}}\colon\widetilde{\Zpq}\to \widetilde{\hcoY}$}
\label{Z2Y2}

We may also construct 
a generically conic bundle $\widetilde{\Zpq}\to \widetilde{\hcoY}$
from
$\widehat{\Zpq}\to \widehat{\hcoY}$.

We recall the diagrams (\ref{eq:diag3}) and (\ref{eq:diagr4}).
In the $\mP^5$-bundle
$\mP(\wedge^2 V) \times \widetilde{\hcoY}$ over $\widetilde{\hcoY}$,
we consider a $\mP^2$-bundle
\[
\mP(\Lrho_{\widetilde{\hcoY}}^*\sS)=\mP(\sS) \times_{\widehat{\hcoY}} \widetilde{\hcoY}\to \widetilde{\hcoY}.
\]
$\mP(\Lrho_{\widetilde{\hcoY}}^*\sS)\to \mP(\sS)$ 
is the blow-up
along the pull-back of
$\Prt_{\rho}$ in 
$\mP(\sS)$ and
the exceptional divisor of
$\mP(\Lrho_{\widetilde{\hcoY}}^*\sS)\to \mP(\sS)$ is 
 $\Zpq_{\rho}\times_{\sP_{\rho}} F_{\rho}
=\mP(\Lrho_{\widetilde{\hcoY}}^*\sS|_{F_{\rho}})$
since $\Lrho_{\widetilde{\hcoY}}\colon \widetilde{\hcoY}\to \widehat{\hcoY}$ is the blow-up along $\Prt_{\rho}$.

Let $\widetilde{\Zpq}$ be the strict transform of $\widehat{\Zpq}$,
which is nothing but the blow-up of $\widehat{\Zpq}$ along $\Zpq_{\rho}$.
Then $\widetilde{\Zpq}$ is smooth
by Propositions \ref{cla:G25} and \ref{cla:P1bdl}.

$\widetilde{\Zpq}$ can be obtained from $\Zpq_2$ in \cite[\S 4.3]{HoTa4}
by restricting $\Zpq_2$ over a point of $\mP(V)$.
By [ibid., Prop.~4.3.3 and 4.3.4], we deduce the following,
where we denote the transform of $\Prt_{\sigma}$ on
$\widetilde{\hcoY}$ also by
$\Prt_{\sigma}$
since 
$\widehat{\hcoY}$ and $\widetilde{\hcoY}$ are isomorphic near
$\Prt_{\sigma}$:

\begin{prop}
\label{prop:cobu}
The induced morphism $\Lpi_{\widetilde{\Zpq}}\colon \widetilde{\Zpq}\to \widetilde{\hcoY}$ is a conic bundle over $\widetilde{\hcoY}\setminus \Prt_{\sigma}$
and the fiber over a point $y\in \widetilde{\hcoY}\setminus \Prt_{\sigma}$
can be identified with the conic corresponding to $y$.
\end{prop}

In the subsections \ref{section:P} and \ref{sub:resol},
we also need 
\[
\widetilde{\Zpq}^t:=\widehat{\Zpq}\times_{\widehat{\hcoY}} \widetilde{\hcoY}
\subset \mP(\Lrho_{\widetilde{\hcoY}}^*\sS),
\]
which is the total transform of 
$\widehat{\Zpq}$
by the blow-up $\mP(\Lrho_{\widetilde{\hcoY}}^*\sS)\to \mP(\sS)$
since it contains $\Zpq_{\rho}\times_{\sP_{\rho}} F_{\rho}$.
$\widetilde{\Zpq}^t$ is reduced since  
$\widehat{\Zpq}$ is smooth by Proposition \ref{cla:G25}.

\vspace{5pt}

Here we summerize the constructions of generically conics bundles
in the following diagram:
\begin{equation*}
\label{eq:genconic}
\xymatrix{\mP(\sS)  & \mP(\Lrho_{\widetilde{\hcoY}}^*\sS)\ar[l]  \\
\widehat{\Zpq}\ar[d]\ar@{^{(}->}[u] & \widetilde{\Zpq}^t\ar[l]\ar[d]\ar@{^{(}->}[u] & \widetilde{\Zpq}\ar[dl]\ar@{_{(}->}[l] \\
\widehat{\hcoY} & \widetilde{\hcoY}.\ar[l]}
\end{equation*}

\section{{\bf Plausible kernel $\sP$ inducing homological projective duality}}
\label{section:Corr}

\subsection{Locally free sheaves $\widetilde{\sS}_L$, $\widetilde{\sQ}$, and $\widetilde{\sT}$ on $\widetilde{\hcoY}$.}
We recall that
$\sS$ and $\sQ$ are universal sub- and quotient bundles on
$\widehat{\hcoY}=\mathrm{G}(3,\wedge^2 V)$,
respectively.
We write down the universal exact sequence on $\widehat{\hcoY}$ as follows:
\begin{equation}
\label{eq:univ36}
0\to \sS\to \wedge^2 V\otimes \sO_{\widehat{\hcoY}}\to \sQ\to 0.
\end{equation}
Taking the $\SL(V)$-action into account,
we define 
\[
\sS_L:=\sS\otimes \wedge^4 V^*,\ \sS^*_L:=\sS^*\otimes \wedge^4 V
\]
(note that $\wedge^4 V$ corresponds to $L$ in \cite{DerSym}).

In \cite[\S 4]{DerSym}, we have introduced three important sheaves $\widetilde{\sS}_L$,
$\widetilde{\sQ}$, and $\widetilde{\sT}$
on $\widetilde{\hcoY}$,
which will appear in the locally free resolutions in
Theorems \ref{thm:newker} and \ref{thm:resolY}.
We set  
\[\widetilde{\sS}_L:=
\Lrho_{\widetilde{\hcoY}}^*\sS_L,\
\widetilde{\sQ}:=\Lrho_{\widetilde{\hcoY}}^*\sQ.
\]
We define $\widetilde{\sT}$ as the dual of the 
following locally free sheaf $\widetilde{\Omega}$ fitting
into the
exact sequence;
\begin{equation}
\label{eq:T*}
0\to \widetilde{\Omega}\to V^*\otimes \sO_{\widetilde{\hcoY}} 
\to {\iota_{F_{\rho}}}_*(\Lrho_{\widetilde{\hcoY}}|_{F_{\rho}})^*\sO_{\mP(V)}(1)\to 0,
\end{equation}
where
$\iota_{F_{\rho}}\colon F_{\rho}\hookrightarrow \widetilde{\hcoY}$ is 
the natural closed embedding, and
the map $V^*\otimes \sO_{\widetilde{\hcoY}} 
\to {\iota_{F_{\rho}}}_*(\Lrho_{\widetilde{\hcoY}}|_{F_{\rho}})^*\sO_{\mP(V)}(1)$ is induced from the natural map
$V^*\otimes \sO_{\mP(V)} 
\to \sO_{\mP(V)}(1)$.
We note that $\widetilde{\sT}$ is denoted by $\widetilde{\mathfrak{Q}}$ in \cite{DerSym}.

\subsection{Family of hyperplane sections $\Vs$}

Let $\Vs\subset \widetilde{\hcoY}\times {\hchow}$
be the pull-back of the universal family of hyperplane sections
in $\mP({\ft S}^2 V^*)\times \chow$.
We can consider $\Vs$ to be both the family over $\widetilde{\hcoY}$
of the pull-backs of
hyperplane sections of $\chow$,
and
the family over $\hchow$ of the pull-backs of
hyperplanes in $\mP({\ft S}^2 V^*)$.
We denote by $\iota_{\Vs}$ the natural closed embedding
$\Vs\hookrightarrow \widetilde{\hcoY}\times \hchow$.

\subsection{Definition of $\sP$}\label{section:P}

We set
\begin{equation}
\label{eq:lP}
(\Delta')^t:=\widetilde{\Zpq}^t\times_{\mathrm{G}(2,V)} \hchow.
\end{equation}
Since
$\widetilde{\Zpq}^t\subset \widetilde{\hcoY}\times \mathrm{G}(2,V)$
and
$(\widetilde{\hcoY}\times \mathrm{G}(2,V))\times_{\mathrm{G}(2,V)}
\hchow\simeq \widetilde{\hcoY}\times \hchow$,
we may consider
$(\Delta')^t$ to be contained in $\widetilde{\hcoY}\times \hchow$.
As such, it holds that
\begin{equation}
\label{eq:lxPy}
(\Delta')^t=\{(y, x)\mid [l_x]\in \mathrm{P}_y\}\subset
\widetilde{\hcoY}\times \hchow,
\end{equation}
where $[l_x]=g(x)\in \mathrm{G}(2,V)$ and
$\mathrm{P}_y$ is the plane of $\mP(\wedge^2 V)$ corresponding to
$\Lrho_{\widetilde{\hcoY}}(y)\in \widehat{\hcoY}$.
Namely,
$(\Delta')^t$ is
the pull-back of $\widehat{\Zpq}$
by $\widetilde{\hcoY}\times \hchow\to
\mathrm{G}(3,\wedge^2 V)\times \mathrm{G}(2,V)$.
Then, by Proposition \ref{prop:ci},
$(\Delta')^t$ is the complete intersection in
$\widetilde{\hcoY}\times \hchow$ with respect to
the unique nonzero $\SL(V)$-invariant section
of $H^0(\widetilde{\hcoY}\times \hchow, 
\widetilde{\sQ} 
\boxtimes \sO_{\hchow}(L))$
up to constant.

\begin{defn}[$\Delta'$, $\sJ$ and $\sP$]
\begin{enumerate}[(1)]
\item
We define 
\[
\Delta':=(\Delta')^t\cap \Vs,
\]
and $\sJ\subset \sO_{\Vs}$ to be its ideal sheaf.
In other words,
$\sJ$ is the image of 
the composite of 
the map
$\widetilde{\sQ}^* 
\boxtimes \sO_{\hchow}(-L)\to
\sO_{\widetilde{\hcoY}\times \hchow}$ corresponding to the subscheme 
$(\Delta')^t$
and the natural map
$\sO_{\widetilde{\hcoY}\times \hchow}\to
{\iota_{\Vs}}_*\sO_{\Vs}$.
\item
The definition of $\sJ$
induces a surjection
\[
\widetilde{\sQ}^* 
\boxtimes \sO_{\hchow}|_{\Vs}\to \sJ(L).
\]
The kernel of this map,
namely, the second sygyzy of $\sJ(L)$
is a coherent sheaf of rank $2$ on $\Vs$, and, by
\cite[Proposition 1.1]{H}, is reflexive.
Now we define $\sP$ to be the dual of this kernel.
$\sP$ is also a reflexive sheaf of rank $2$ on $\Vs$. 
The definition gives the following exact sequence:
\begin{equation}
\label{eq:P*}
0\to \sP^*\to 
\widetilde{\sQ}^* 
\boxtimes \sO_{\hchow}|_{\Vs}\to \sJ(L)\to 0.
\end{equation}
\end{enumerate}
\end{defn}

It turns out that 
the coherent sheaf ${\iota_{\Vs}}_*\sP$ admits 
a nice locally free resolution.

\begin{thm}
\label{thm:newker}
The coherent sheaf ${\iota_{\Vs}}_*\sP$ admits the following locally free resolution$:$
\begin{eqnarray}
\label{eqnarray:onVs'}
0\to
\sO_{\widetilde{\hcoY}} 
\boxtimes \Omega^1_{\hchow/\mathrm{G}(2,V)}(-L)
\to
\widetilde{\sT}\boxtimes \sF^*(-H)
\to\\
\widetilde{\sQ} 
\boxtimes \sO_{\hchow}\oplus 
\widetilde{\sS}^*_L\boxtimes
\sO_{{\hchow}}(L-H)
\to
{\iota_{\Vs}}_*\sP\to 0,
\nonumber
\end{eqnarray}
where 
$\Omega^1_{\hchow/\mathrm{G}(2,V)}$ is the relative cotangent bundle
for the morphism $\hchow \to \mathrm{G}(2,V)$.
\end{thm}
We will show this theorem in the appendix \ref{app:B}.
We remark that (dual) Lefschetz collections in
$\sD^b(\widetilde{\hcoY})$ and $\sD^b(\hchow)$
can be read off from 
the locally free resolution (\ref{eqnarray:onVs'}).
We have shown this fact in \cite[Cor.~3.3 and 5.11]{DerSym}.
Thus we expect $\sP$ will
induce the homological projective duality
between (suitable noncommutative resolutions of) $\hcoY$ and $\chow$
with respect to these (dual) Lefschetz collections.

In our proof of Theorem \ref{thm:main},
we need the restrictions of 
$\sP$ and their locally free resolutions 
(\ref{eqnarray:onVs'}) over fibers of $\Vs\to \hchow$ and 
$\Vs\to \widetilde{\hcoY}$. Here we only state the results postponing their derivations
to the appendix \ref{app:B} (Proposition \ref{prop:restP}). 

For $x\in \hchow$ and $y\in \widetilde{\hcoY}$,
we denote by $\Vs_x$ and $\Vs_y$ 
the fibers of $\Vs\to \hchow$ over $x$ and 
$\Vs\to \widetilde{\hcoY}$ over $y$, respectively.
Let \[
\iota_x\colon \Vs_x\hookrightarrow \hchow,\
\iota_y\colon \Vs_y\hookrightarrow \widetilde{\hcoY}
\] be
the natural inclusions.
We set 
\[
\sP_x:=\sP|_{\Vs_x},\ \sP_y:=\sP|_{\Vs_y}. 
\]

\begin{prop}
\label{prop:restP'}
The locally free resolution $(\ref{eqnarray:onVs'})$
restricts to $\widetilde{\hcoY}$ for any $x\in \hchow$. Namely, the following sequence is exact on $\widetilde{\hcoY}:$
\begin{eqnarray}
\label{eqnarray:onVsx}
0\to
\sO_{\widetilde{\hcoY}}^{\oplus 2} 
\to
\widetilde{\sT}^{\oplus 2}
\to
\widetilde{\sQ} \oplus 
\widetilde{\sS}^*_L
\to
{\iota_x}_*\sP_x\to 0.
\end{eqnarray}
Similarly, for $y\in \widetilde{\hcoY}\setminus \Prt_{\sigma}$,
the following sequence is exact on $\hchow:$ 
\begin{eqnarray}
\label{eqnarray:onVsy}
0\to
\Omega^1_{\hchow/\mathrm{G}(2,V)}(-L)
\to
\sF^*(-H)^{\oplus 4}
\to\\
\sO_{\hchow}^{\oplus 3}\oplus 
\sO_{{\hchow}}(L-H)^{\oplus 3}
\to
{\iota_y}_*\sP_y\to 0.
\nonumber
\end{eqnarray}
\end{prop}

\subsection{Components of $\Delta'$}
\label{sub:Delta'}
By the very definition of $(\Delta')^t$, it contains
\begin{equation}
\label{eq:defDelta}
\Delta:=\widetilde{\Zpq}\times_{\mathrm{G}(2,V)} \hchow.
\end{equation}
We may consider $\Delta$ to be contained in $\widetilde{\hcoY}\times \hchow$
as we did for $(\Delta')^t$.
Now we give descriptions of $\Delta$, which 
will be needed in our proofs of Theorems \ref{thm:main}
and \ref{thm:IK}
(Lemmas \ref{prop:Delta2descrip} and \ref{Delta1fiber}).
It is convenient to understand the definition of $\Delta$
step by step as summerized in the following diagram:

\vspace{5pt}

\begin{equation}
\label{eq:constDelta}
\xymatrix{
\widetilde{\Zpq}\times \hchow\ar[d]_{\footnotesize{\text{blow-up}}} & \Delta:=\widetilde{\Zpq}\times_{\mathrm{G}(2,V)}\hchow\ar@{_{(}->}[l]
\ar[d]_{\footnotesize{\text{blow-up}}}\ar@{^{(}->}[r] & 
\widetilde{\hcoY}\times\hchow\\
\widehat{\Zpq}\times \hchow\ar[d]_{\footnotesize{\text{$\mathrm{G}(2,5)$-bundle}}} & \widehat{\Delta}:=\widehat{\Zpq}\times_{\mathrm{G}(2,V)} \hchow\ar@{_{(}->}[l]\ar[d]_{\footnotesize{\text{$\mathrm{G}(2,5)$-bundle}}} \\
\mathrm{G}(2,V)\times \hchow\ar[d] & \Delta_G:=\mathrm{G}(2,V)\times_{\mathrm{G}(2,V)} \hchow\simeq \hchow\ar@{_{(}->}[l]\ar[d]\\
   \mathrm{G}(2,V) \times \mathrm{G}(2,V) & \Delta_{0}:=
\mathrm{G}(2,V)\times_{\mathrm{G}(2,V)} \mathrm{G}(2,V)\ar@{_{(}->}[l]}
\end{equation}
We note that $\widehat{\Delta}$ is
a $\mathrm{G}(2,5)$-bundle over $\Delta_G$
since so is $\widehat{\Zpq}$ over $\mathrm{G}(2,V)$ by Proposition \ref{cla:G25}.
In particular, $\widehat{\Delta}$ is smooth.
Then we see that  
${\Delta}$
is the blow-up of
$\widehat{\Delta}$ along 
$\Zpq_{\rho}\times_{\mathrm{G}(2,V)}\hchow$
since $\widetilde{\Zpq}\to \widehat{\Zpq}$ is the blow-up along
$\Zpq_{\rho}$.
Therefore
\begin{equation}
\label{eq:deltasmooth}
\text{${\Delta}$ is smooth}
\end{equation}
since
$\Zpq_{\rho}\times_{\mathrm{G}(2,V)}\hchow$
 is 
a $\mP^1$-subbundle of $\widehat{\Delta}\to \hchow$ over $\hchow$ by Proposition \ref{cla:P1bdl}.
In particular, we have shown 
\begin{prop}
\label{prop:flat}
The fiber of $\Delta$ 
over a point $x\in \hchow$ is isomorphic to the blow-up of $\mathrm{G}(2,5)$ along a conic.
In particular, $\Delta\to \hchow$ is flat.
\end{prop}
Considering
$\Delta\subset \widetilde{\hcoY}\times \hchow$,
we describe 
the natural morphisms
${\Delta}\to \hchow$
and ${\Delta}\to \widetilde{\hcoY}$ in the following proposition:
\begin{prop}
\label{prop:Deltasmooth}
\begin{enumerate}[$(1)$]
\item
The fiber of ${\Delta}\to \hchow$
over a point $x\in \hchow$ parameterizes $\tau$- and $\rho$-conics containing $[l_x]$
and $\sigma$-planes containing $[l_x]$, where $[l_x]:=g(x)\in \mathrm{G}(2,V)$.
In particular, $\Delta\subset \Vs$.
\item
The fiber of ${\Delta}\to \widetilde{\hcoY}$
over a point $y\in \widetilde{\hcoY}$
is the $\mP^2$-bundle $g^{-1}(q_y)$
if $y\not\in \Prt_{\sigma}$,
where $q_y$ is the conic corresponding to $y$,
and
is the $\mP^2$-bundle $g^{-1}({\rm{P}}_y)$
if $y\in \Prt_{\sigma}$, where ${\rm{P}}_y$ is 
the $\sigma$-plane corresponding to $y$.
In particular,
${\Delta}\to \widetilde{\hcoY}$ is flat
over $\widetilde{\hcoY}\setminus \Prt_{\sigma}$.
\end{enumerate}
\end{prop}
\begin{proof}
It is easy to derive the asserions
from Proposition \ref{prop:cobu}.
We only note that 
$\Delta\subset \Vs$ by the first asserion of (1) and 
\cite[Prop.~4.20]{Geom}.
\end{proof}

We note that the diagram (\ref{eq:constDelta}) is useful
to construct a locally free resolution of
the ideal sheaf of $\Delta$ in
$\widetilde{\hcoY}\times \hchow$ (Theorem \ref{thm:resolY}).

Now we will obtain the irreducible decomposition of $\Delta'$.
By Proposition \ref{prop:Deltasmooth} (1),
we have $\Delta\subset (\Delta')^t\cap \Vs=\Delta'$.
We will see that $\Delta$ is an irreducible component of $\Delta'$.
We define a subvariety $D$ of 
$F_{\rho}\times \hchow$, which is shown to be another 
irreducible component of $\Delta'$.
We denote a point of $F_{\rho}$ by $[q_{V_1}]$, where 
$q_{V_1}$ is a $\rho$-conic
contained in the $\rho$-plane $\rm{P}_{V_1}$ for some $[V_1]\in \mP(V)$,
and
a point of $\hchow$ by $[\eta]$, where $\eta$ is the $0$-dimensional
subscheme of $\mP(V)$.  
We define the subvariety $D$
of $F_{\rho}\times \hchow$ by
\[
D:=\{([q_{V_1}], [\eta])\mid \Supp \eta \ \text{contains}\ [V_1]\}.
\]  
By the definitions of $D$ and $(\Delta')^t$, we immediately see that
$D\subset (\Delta')^t$.
By \cite[Prop.~4.20]{Geom}, we also see that
$D\subset \Vs$. Therefore we have
\[
D\subset (\Delta')^t\cap\Vs=\Delta'.
\]
We give a description of $D$.
Recall that $E_f$ is the exceptional divisor of $f\colon \hchow\to \chow$.

\begin{lem}
\label{D}
The fiber of 
the natural morphism $D\to F_{\rho}$ over a point $[q_{V_1}]\in F_{\rho}$
is the blow-up of $\mP(V)$ at $[V_1]$.
This fiber is also a $\mP^1$-bundle over the $\rho$-plane
${\rm P}_{V_1}$.
Moreover, $D$ is a $\mP^5$-bundle over 
the double cover of $\hchow$ branched along $E_f$.
In particular, 
$D$ is irreducible and smooth, and
$D\to F_{\rho}$ and 
$D\to \hchow$ are flat.
\end{lem}
\begin{proof}
The first assertion is clear from the definition of $D$.

We consider 
the natural projection from $D$ to $\hchow$. 
Let $[\eta]\in \hchow$.
If $\Supp \eta$ consists of two points $x_1, x_2\in \mP(V)$,
then
the fiber of $D\to \hchow$ over $[\eta]$ is the union of
the fibers of $F_{\rho}\to \mP(V)$ over $x_1$ and $x_2$.
If $\Supp \eta$ consists of one point $x \in \mP(V)$,
then
the fiber of $D\to \hchow$ over $[\eta]$ is 
the fiber of $F_{\rho}\to \mP(V)$ over $x$.
In particular, $D\to \hchow$ is surjective.
The Stein factorization of 
$D\to \hchow$ factors through 
the double cover of $\hchow$ branched along $E_f$.
\end{proof}

\begin{prop}
$\Delta'=\Delta\cup D$ is the irreducible decomposition.
\end{prop}

\begin{proof}
The natural projection $(\Delta')^t\to \widetilde{\Zpq}^t$ is
a $\mP^2$-bundle since so is $\hchow\to \mathrm{G}(2,V)$.
Therefore
it holds that 
\[
(\Delta')^t=(\widetilde{\Zpq}\times_{\mathrm{G}(2,V)} \hchow) \cup
(\mP(\Lrho_{\widetilde{\hcoY}}^*\sS|_{F_{\rho}})\times_{\mathrm{G}(2,V)} \hchow)\]
since
$\widetilde{\Zpq}^t=\widetilde{\Zpq}\cup \mP(\Lrho_{\widetilde{\hcoY}}^*\sS|_{F_{\rho}})$.
We set
\[
D':=
(\mP(\Lrho_{\widetilde{\hcoY}}^*\sS|_{F_{\rho}})\times_{\mathrm{G}(2,V)} \hchow)\cap \Vs.
\]
It is clear by definition that
$\Delta\cup D\subset \Delta'=\Delta\cup D'$.
Thus we have only to show that $D'\subset \Delta\cup D$.
We may consider $D'$ is contained in $F_{\rho}\times \hchow$.
Then, by (\ref{eq:lxPy}), it holds that \[
D'=\{([q_{V_1}], [\eta])\mid 
[l_{\eta}]\in \mathrm{P}_{V_1}, ([q_{V_1}], [\eta])\in \Vs\},
\]
where $l_{\eta}$ is the line of $\mP(V)$ determined by the subscheme $\eta$.
We set $\Supp \eta=\{x_1,x_2\}$, where $x_1$ and $x_2$ may be equal.
Let $Q$ be the quadric determined by $q_{V_1}$ as in \cite[Prop.~4.20]{Geom}
and $B$ a symmetric bi-linear form defining $Q$.
Then
$([\eta], [q_{V_1}])\in \Vs$ means
that $B(x_1,x_2)=0$ by \cite[Prop.~4.20]{Geom}.
Since $[V_1]\in \Sing Q$ and $[V_1]\in l_{\eta}$,
we have $\rank B|_{l_{\eta}}=0, 1$.
If $\rank B|_{l_{\eta}}=0$,
then $l_{\eta}\subset Q$.
Therefore $[l_{\eta}]\in q_{V_1}$ and 
then $([q_{V_1}], [\eta])\in \Delta$ by Proposition \ref{prop:Deltasmooth} (1)
.
If $\rank B|_{l_{\eta}}=1$,
then $B(x_1,x_2)=0$ implies that
$x_1$ or $x_2=[V_1]$, namely,
$([\eta], [q_{V_1}])\in D$.
\end{proof}

\section{{\bf Proof of Theorem \ref{thm:main}}}
\mylabel{section:BC}
Using the coherent sheaf $\sP$,
we define the functor $\Phi_1$ appearing in Theorem \ref{thm:main}.
\begin{defn}
Noting $\widetilde{Y}\times X \subset \Vs$, 
we set
\[
\sP_1:=\sP|_{\widetilde{Y}\times X},
\]
and 
\[
\Phi_1:=\Phi_{\sP_1}\colon \sD^b(X)\to \sD^b(\widetilde{Y}).
\]
\end{defn}

In the subsection \ref{subsection:ff},
we prove the fully faithfulness of the functor $\Phi_1$
verifying
the conditions of Theorem \ref{thm:D} for $\sP_1$.
The arguments of this subsection are the same as those in \cite[\S 8]{HoTa4},
and originally go back to \cite{BC}
except technical differences in Propositions \ref{cla:tower} and \ref{cla:Hom}.

In the subsection \ref{section:full},
we show the fullness of the decomposition in 
Theorem \ref{thm:main}, which completes the proof of the theorem.
This part does not appear in \cite{HoTa4} since
in the Calabi-Yau case, 
fully faithfulness automatically implies equivalence (cf.~Theorem \ref{thm:D}).
In the course of this part,
we also essentially see the relationship between Theorem \ref{thm:main} and 
the main result of \cite{IK} (Lemma \ref{lem:IK}). 
In the proof,
we follow a philosophy of homological projective duality;

let $W$ be a codimension three linear section of $\chow$
containing $X$,
and $S$ the linear section of $\hcoY$ orthogonal to $W$;
\begin{align*}
X&\subset W\subset \chow\\
\hcoY\supset Y&  \supset S.
\end{align*}
Then a strong relationship is expected among the derived categories of
$X$, $Y$, and $W$ and $S$. 

This philosophy is indicated in \cite{K-g7} in a special case,
and in \cite{HPD1} generally. We closely follow the arguments in 
\cite{K-g7, HPD1}.

\subsection{Fully faithfulness of $\Phi_1$}
\label{subsection:ff}

We will show that $\sP_1$ is flat over $X$ in the appendix \ref{app:B}
(Proposition \ref{prop:flatP1}). 
We denote by $\sP_{x;1}$ the restriction of 
$\sP_1$ over the fiber of $\widetilde{Y}\times X\to X$ over $x$.

We verify the condition (ii) for $\sP_1$, i.e.,
the following vanishing:
\begin{equation}
\label{eq:aim}
\Ext^{\bullet}(\sP_{x_1;1}, \sP_{x_2;1})=0\
\text{for any two distinct points $x_1$ and $x_2$ of $X$}.
\end{equation}
The starting point is the following vanishing.
We follow the notation of Proposition \ref{prop:restP'}, and
denote by $(-t)$ the tensor product of
$\sO_{\widetilde{\hcoY}}(-tM)$.

\begin{prop}
\label{prop:Extvan}
For any two points $x_1$ and $x_2$ of $\hchow$,
it holds \[
\Ext^{\bullet}({\iota_{{x_1}}}_*\sP_{x_1}, {\iota_{{x_2}}}_*\sP_{x_2}(-t))=0\,
(1\leq t\leq 5).
\] 
\end{prop}

\begin{proof}
The vanishing can be derived in a standard way from
(\ref{eqnarray:onVsx}) and \cite[Thm.~5.1]{DerSym}.
\end{proof}

To step forward, we need cut out $\sP_{x_i;1}$ from $\sP_{x_i}$ ($i=1,2$) in an appropriate way as in the following proposition, which will be derived in
the appendix \ref{app:B} from
Proposition \ref{prop:flatP1}
(the subsection \ref{cutting}):

\begin{prop}
\mylabel{cla:tower}
There exists a ladder of complete intersections 
of $\widetilde{\hcoY}$
by members of $|M_{\widetilde{\hcoY}}|;$ 
\[
Y_1\subset Y_2\subset \cdots \subset Y_6\subset Y_7\
(\dim Y_i=2+i),
\]
and coherent sheaves $\sP_{x_i;j}$ on $Y_j$ for $i=1,2$ and $1\leq j\leq 7$
satisfying the following conditions$:$
\begin{enumerate}[$(1)$]
\item
$Y_1:=\widetilde{Y}$ and $Y_{7}:=\widetilde{\hcoY}$.
\item
$Y_6:=\Vs_{x_1}$ and 
$\sP_{x_1;6}:=\sP_{x_1}$, which is a reflexive sheaf on
$Y_6$ by Proposition $\ref{prop:restP}$.
\item
$Y_5:=\Vs_{x_1}\cap \Vs_{x_2}$,
where the intersection is taken in $\widetilde{\hcoY}$. 
We denote
by $\iota^i_{Y_5}$ the embedding $Y_5\hookrightarrow \Vs_{x_i}$
for $i=1,2$.
$\sP_{x_1;5}:={\iota^{1*}_{Y_5}} \sP_{x_1}$ and 
$\sP_{x_2;5}:={\iota^{2*}_{Y_5}} \sP_{x_2}$ are reflexive on $Y_5$.
\item
We denote
by $\iota_{Y_j}$ the embedding $Y_j\hookrightarrow Y_{j+1}$ for $j\leq 4$.
$\sP_{x_i;j}:=\iota_{Y_j}^*\sP_{x_i;j+1}$ 
is reflexive on $Y_j$ for $j\leq 4$ and $i=1,2$.
\end{enumerate}
\end{prop}

In particular, the choices of $Y_{7}$ and $Y_{6}$ there turn out to be
crucial in Steps 1 and 2 in the following arguments. 

\vspace{5pt}

\noindent\textbf{Step 1 (from $\widetilde{\hcoY}$ to $Y_{5}$).}
In this step, we show \begin{equation}
\Ext_{Y_{5}}^{\bullet-1}(\sP_{x_{1};5},
\sP_{x_{2};5}(-t+1))=0\,(1\leq t\leq 5).\label{eq:step1}\end{equation}

Note that
\[
\Ext^{\bullet}_{Y_7}({\iota_{{x_1}}}_*\sP_{x_1}, {\iota_{{x_2}}}_*\sP_{x_2}(-t))\simeq
\Ext^{\bullet-1}_{Y_6}(\sP_{x_1}, \iota_{{x_1}}^*{\iota_{{x_2}}}_*\sP_{x_2}(-t+1))
\]
by
applying the Grothendieck-Verdier duality (Theorem \ref{cla:duality}) to
the embedding $\iota_{{x_1}}$. The l.h.s. of this equality
is zero for $1\leq t\leq 5$ by Proposition \ref{prop:Extvan}.
Moreover,
since $\iota_{{x_1}}^*{\iota_{{x_2}}}_*\sP_{x_2} \simeq
{\iota^{1}_{V_5}}_*\iota^{2*}_{V_5}\sP_{x_2}\simeq 
{\iota^{1}_{V_5}}_*\sP_{x_2;5}$,
we have
\begin{eqnarray*}
\Ext^{\bullet-1}_{Y_6}(\sP_{x_1}, \iota_{{x_1}}^*{\iota_{{x_2}}}_*\sP_{x_2}(-t+1))\simeq
\Ext^{\bullet-1}_{Y_6}(\sP_{x_1}, {\iota^{1}_{V_5}}_*\sP_{x_2;5}
(-t+1)) \simeq \\
\Ext^{\bullet-1}_{Y_5}({\iota^{1*}_{V_5}}\sP_{x_1}, \sP_{x_2;5}(-t+1))\simeq
\Ext^{\bullet-1}_{Y_5}(\sP_{x_1;5}, \sP_{x_2;5}(-t+1)).
\end{eqnarray*}
Thus (\ref{eq:step1}) follows.

\vspace{5pt}

\noindent\textbf{Step 2 (from $Y_{5}$ to $Y_{4}$).}
In this step, we show \begin{equation}
\Ext_{Y_{4}}^{\bullet-1}(\sP_{x_{1};4},\sP_{x_{2};4}(-t+1))=0\,(1\leq t\leq 4).\label{eq:step2}\end{equation}

Since $\sP_{x_{1};4}\simeq\iota_{Y_{4}}^{*}\sP_{x_{1};5}$, we have
\begin{equation}
\begin{aligned}\Ext_{Y_{4}}^{\bullet-1}(\sP_{x_{1};4},\sP_{x_{2};4}(-t+1)) & \simeq\Ext_{Y_{4}}^{\bullet-1}(\iota_{Y_{4}}^{*}\sP_{x_{1};5},\sP_{x_{2};4}(-t+1))\\
 & \simeq\Ext_{Y_{5}}^{\bullet-1}(\sP_{x_{1};5},\iota_{Y_{4}*}\sP_{x_{2};4}(-t+1)).\end{aligned}
\label{eqnarray:D}\end{equation}
 From (\ref{eqnarray:D}) and the
exact sequence \[
0\to\sP_{x_2;5}(-1)\to\sP_{x_{2};5}\to\iota_{Y_{4}*}\sP_{x_{2};4}\to0\]
we obtain the exact sequence \begin{equation*}
\begin{aligned}\Ext_{Y_{5}}^{\bullet-1}(\sP_{x_{1};5},\sP_{x_{2};5}(-t+1)) & \to\Ext_{Y_{4}}^{\bullet-1}(\sP_{x_{1};4},\sP_{x_{2};4}(-t+1))\\
 & \to\Ext_{Y_{5}}^{\bullet}(\sP_{x_{1};5},\sP_{x_2;5}(-t)),\end{aligned}
\label{eqnarray:E}\end{equation*}
 where,
from (\ref{eq:step1}), the first term vanishes for $1\leq t\leq 5$ and 
the third term vanishes for $1\leq t\leq 4$.
Thus (\ref{eq:step2}) follows.

\vspace{5pt}

\noindent\textbf{Step 3 (from $Y_{4}$ to \dots to $\widetilde{Y}$).}
In this step, we finish the proof of (\ref{eq:aim}). 

In a similar way to the argument of Step 2, we may prove
\begin{equation*}
\Ext_{Y_{i}}^{\bullet-1}(\sP_{x_{1};i},\sP_{x_{2};i}(-t+1))=0\,(1\leq t\leq i)\label{eq:step4}\end{equation*} for $i=3, 2,1$.
In particular, we have $\Ext_{Y}^{\bullet-1}(\sP_{x_{1};1},\sP_{x_{2};1})=0$,
which is (\ref{eq:aim}). \hfill{}$\square$

\vspace{5pt}

Finally we verify the condition (i) of 
Theorem \ref{thm:D}.

\begin{prop}
\label{cla:Hom}
For any point $x\in X$,
 it holds that $\Hom (\sP_{x;1}, \sP_{x;1})\simeq \mC$,
where we recall that $\sP_{x;1}$ is the restriction of $\sP_1$ over $x$.
\end{prop}

\begin{proof}
Since $\sP_{x;1}$ is reflexive of rank 2 by Proposition \ref{cla:tower}, we have
$\sP_{x;1}\simeq \sP^*_{x;1}\otimes \det \sP_{x;1}$ by \cite[Prop.~1.10]{H}.
Therefore, 
the assertion is equivalent to 
\[
\Hom (\sP^*_{x;1}, \sP^*_{x;1})\simeq \mC.
\]
Since $\sP_{x;1}$ is torsion free,
we obtain an injection $\sP^*_{x;1}\hookrightarrow \widetilde{\sQ}|_{\widetilde{Y}}$
by
restricting
(\ref{eq:P*}) to $\widetilde{Y}$.
Therefore we also obtain
an injection
$\Hom (\sP^*_{x;1}, \sP^*_{x;1})\hookrightarrow
\Hom (\sP^*_{x;1}, \widetilde{\sQ}^*|_{\widetilde{Y}})$.
Note that  
$\Hom (\sP^*_{x;1}, \widetilde{\sQ}^*|_{\widetilde{Y}})\simeq
H^0(\widetilde{Y}, \sP_{x;1}\otimes \widetilde{\sQ}^*|_{\widetilde{Y}})$. 
We can show that the r.h.s.~is isomorphic to $\mC$
by the argument to show (\ref{eq:aim})
using the locally free resolution (\ref{eqnarray:onVsx}) and
\cite[Thm.~5.1]{DerSym}.
Indeed, this follows from
the vanishing of 
$H^{\bullet}(\widetilde{\hcoY},\widetilde{\sQ}^*(-t))$,
$H^{\bullet}(\widetilde{\hcoY},\widetilde{\sT}\otimes\widetilde{\sQ}^*(-t))$,
and $H^{\bullet}(\widetilde{\hcoY},\widetilde{\sS}_L^*\otimes 
\widetilde{\sQ}^*(-t))$ for $0\leq t\leq 5$,
and the vanishing of
$H^{\bullet}(\widetilde{\hcoY},\widetilde{\sQ}\otimes 
\widetilde{\sQ}^*(-t))$ for $1\leq t\leq 5$, and
$H^{\bullet}(\widetilde{\hcoY},\widetilde{\sQ}\otimes 
\widetilde{\sQ}^*)\simeq H^{0}(\widetilde{\hcoY},\widetilde{\sQ}\otimes 
\widetilde{\sQ}^*)\simeq \mC$.
\end{proof}

\subsection{Fullness of the collection in Theorem \ref{thm:main}}
\label{section:full}

In this subsection, we show
\[
\sD_{Y}=\langle \Phi_1(\sD^b(X)), \sO_{\widetilde{Y}}(1), \sO_{\widetilde{Y}}(2)\rangle,
\] which completes our proof of Theorem \ref{thm:main}.

\begin{lem}
\label{lem:subset}

\begin{enumerate}[$(1)$]
\item
The collection
\[
\{\sO_{F_i}(-1,-1)\}_{1\leq i\leq 10},
\sO_{\widetilde{Y}}(1),
\sO_{\widetilde{Y}}(2)
\]
is semi-orthogonal, where we omit the symbol ${\iota_i}_*$ for
$\sO_{F_i}(-1,-1)$ to simplify the notation.
\item
We set
\begin{equation}
\label{eq:lr}
\sC_{{Y}}:=
{\empty^{\perp}\langle \{\sO_{F_i}(-1,-1)\}_{1\leq i\leq 10}\rangle\cap
\langle \sO_{\widetilde{Y}}(1),
\sO_{\widetilde{Y}}(2)\rangle}^{\perp}.
\end{equation}
Then
$\Phi_1(\sD^b(X))\subset \sC_{{Y}}$.
\end{enumerate}
\end{lem}

\begin{proof}
(1) First we see that 
\[
\Ext^{\bullet}(\sO_{\widetilde{Y}}(l), \sO_{F_i}(-1,-1))\simeq
H^{\bullet}(F_i, \sO_{F_i}(-1,-1))=0
\]
for any $l\in \mZ$ since $\sO_{\widetilde{Y}}(l)|_{F_i}=\sO_{F_i}$.
Moreover,
\[
\Ext^{\bullet}(\sO_{\widetilde{Y}}(2), \sO_{\widetilde{Y}}(1))\simeq
H^{\bullet}(\widetilde{Y}, \sO_{\widetilde{Y}}(-1))\simeq
H^{3-\bullet}(\widetilde{Y}, \sO_{\widetilde{Y}}(M+K_{\widetilde{Y}})).
\]
The r.h.s.~vanishes
for $\bullet=0,1,2$
by the Kawamata-Viehweg vanishing theorem since $M$ is nef and big.
For $\bullet=3$, since $K_{\widetilde{Y}}=-2M+\sum_{i=1}^{10} F_i$,
we have
\[
H^{0}(\widetilde{Y}, \sO_{\widetilde{Y}}(M+K_{\widetilde{Y}}))\simeq
H^{0}(\widetilde{Y}, \sO_{\widetilde{Y}}(-M+\sum_{i=1}^{10} F_i))=0.
\]
(2)
Since $\sO_x$ ($x\in X$) are spanning classes of $\sD^b(X)$ (see \cite[Prop.~3.17]{Huy} for example),
it suffices to show 
that 
\[
\sP_{x;1}=\Phi_1(\sO_x)\in \sC_{{Y}},
\]
which is 
equivalent to the conditions 
\begin{enumerate}[(a)]
\item
$H^{\bullet}(\widetilde{Y}, \sP_{x;1}(-t))=0$ for $t=1,2$.
\item
$\Ext^{\bullet}(\sP_{x;1}, \sO_{F_i}(-1,-1))=0$.
\end{enumerate}
(see (\ref{eq:lr})).

First we show the claim (a).
By (\ref{eqnarray:onVsx}) and \cite[Thm.~5.1]{DerSym},
we have \[
H^{\bullet}(\Vs_x, \sP_x(-t))=0
\]
for any $\bullet$, $x\in \hchow$, and $1\leq t \leq 7$. 
Then, by the argument to show (\ref{eq:aim}),
we have $H^{\bullet}(\widetilde{Y}, \sP_{x;1}(-t))=0$ for $t=1,2$.

Secondly we show the claim (b).
By the argument to show (\ref{eq:aim}),
it suffices to verify that
\[
\Ext^{\bullet}_{\Vs_x}(\sP_x, \sO_{F_{\widetilde{\hcoY}}}(F_{\widetilde{\hcoY}})(-t)|_{\Vs_x})=0\ (0\leq t\leq 5),
\] 
where we recall $F_{\widetilde{\hcoY}}$ is the exceptional divisor
for $\Lp_{\widetilde{\hcoY}}\colon \widetilde{\hcoY}\to \hcoY$
since $F_{\widetilde{\hcoY}}|_{\widetilde{Y}}=\sum_{i=1}^{10} F_i$ and $\sO_{F_{\widetilde{\hcoY}}}(F_{\widetilde{\hcoY}})
|_{\widetilde{Y}}=\oplus_{i=1}^{10} \sO_{F_i}(-1,-1)$.
By the Grothendieck-Verdier duality (Theorem \ref{cla:duality}), we have
\[
\Ext^{\bullet}_{\Vs_x}(\sP_x, \sO_{F_{\widetilde{\hcoY}}}(F_{\widetilde{\hcoY}})|_{\Vs_x})\simeq
\Ext^{\bullet+1}_{\widetilde{\hcoY}}({\iota_x}_*\sP_x,\sO_{F_{\widetilde{\hcoY}}}(F_{\widetilde{\hcoY}})(-1)).
\]
Therefore the problem is reduced to verify that
\begin{equation}
\label{eq:PF}
\Ext^{\bullet+1}_{\widetilde{\hcoY}}({\iota_x}_*\sP_x,\sO_{F_{\widetilde{\hcoY}}}(F_{\widetilde{\hcoY}})(-1-t))=0\ (0\leq t\leq 5).
\end{equation}
By the exact sequence
\[
0\to \sO_{\widetilde{\hcoY}}\to
\sO_{\widetilde{\hcoY}}(F_{\widetilde{\hcoY}})\to
\sO_{_{F_{\widetilde{\hcoY}}}}(F_{\widetilde{\hcoY}})\to
0,
\] 
the proof of (\ref{eq:PF}) is reduced to show that
\begin{equation}
\label{eq:b1}
\Ext^{\bullet}({\iota_x}_*\sP_x,\sO_{\widetilde{\hcoY}}(-1-t))=0
\end{equation}
and
\begin{equation}
\label{eq:b2}
\Ext^{\bullet}({\iota_x}_*\sP_x,\sO_{\widetilde{\hcoY}}(F_{\widetilde{\hcoY}})(-1-t))=0.
\end{equation}
As for (\ref{eq:b1}),
the vanishing follows by
(\ref{eqnarray:onVsx})
and \cite[Thm.~5.1]{DerSym};
only for $t=5$, we also need to use \cite[Prop.~5.9]{DerSym}.
As for (\ref{eq:b2}),
we have 
\[
\Ext^{\bullet}({\iota_x}_*\sP_x,\sO_{\widetilde{\hcoY}}(F_{\widetilde{\hcoY}})(-1-t))\simeq
\Ext^{9-\bullet}(\sO_{\widetilde{\hcoY}}(F_{\widetilde{\hcoY}})(-1-t),{\iota_x}_*\sP_x(K_{\widetilde{\hcoY}}))^*
\]
by the Serre-Grothendieck duality.
Moreover, by the adjunction formula 
$K_{\widetilde{\hcoY}}=-8M_{\widetilde{\hcoY}}+F_{\widetilde{\hcoY}}$
(see \cite[\S 4.8]{Geom}),
the r.h.s.~is isomorphic to 
\[
H^{9-\bullet}(\widetilde{\hcoY},{\iota_x}_*\sP_x(t-7))^*,
\]
which vanish by 
(\ref{eqnarray:onVsx})
and \cite[Thm.~5.1]{DerSym}.

Hence we have shown the claim (b).
\end{proof}

By Lemma \ref{lem:subset} and
a general result in \cite{Bondal},
$\sD^b(\widetilde{Y})$ admits
the following semi-orthogonal decomposition:
\begin{equation}
\label{eq:decomp}
\sD^b(\widetilde{Y})=\langle \{\sO_{F_i}(-1,-1)\}_{1\leq i\leq 10}, \sC_{{Y}}, \sO_{\widetilde{Y}}(1),
\sO_{\widetilde{Y}}(2)\rangle.
\end{equation}
Then
\[
\sD_{Y}=\langle \sC_{{Y}}, \sO_{\widetilde{Y}}(1),
\sO_{\widetilde{Y}}(2)\rangle
\]
is nothing but
the categorical resolution 
with respect to the dual Lefschetz decomposition of $F_i$
as in the statement of Theorem \ref{thm:main}.
Therefore we have only to show 
\begin{equation}
\label{eq:full}
\sC_{{Y}}=\Phi_1(\sD^b(X))
\end{equation}
in the decomposition (\ref{eq:decomp}).
\vspace{5pt}

Let $A$ be any object of 
\[
{\empty^{\perp}\langle \{\sO_{F_i}(-1,-1)\}_{1\leq i\leq 10}, \Phi_1(\sD^b(X))\rangle\cap
\langle \sO_{\widetilde{Y}}(1),
\sO_{\widetilde{Y}}(2)\rangle}^{\perp}.
\]
The equality (\ref{eq:full}) follows once we show
\begin{equation}
\label{eq:A0}
A=0.
\end{equation}
The rest of this section is devoted to proving it. 
The following argument is inspired by \cite[\S 5]{K-g7} and \cite[\S 6]{HPD1}.

In the sequel, we take several smooth linear sections $S$ of $\hcoY$ 
such that 
\[
S\subset Y,
\]
and
the orthogonal linear section $W$ of $\chow$
to $S$ has only eight $\frac 12 (1,1,1)$-singularities. 
In particular, $S$ and $W$ satisfy the conditions (\ref{eq:P2}).
Note that
\[
W\supset X.
\]
Since
$S$ does not intersect $\Sing \hcoY$,
we may consider $S\subset \widetilde{Y}$,
and we denote by 
\[
\alpha\colon S\hookrightarrow \widetilde{Y}
\]
the natural inclusion.
By a similar reason, we may consider $X\subset \hW$, and
we denote by 
\[
\beta\colon X\hookrightarrow \hW
\]
the natural inclusion.

We define
\[
\sP_2:=\sP|_{S\times \hW}
\]
and
the associated functor
\[
\Phi_2:=\Phi_{\sP_2}\colon \sD^b(\hW)\to \sD^b(S).
\]
We also use $\sC_S$ and $\sC_W$ as in the statement of Theorem \ref{thm:main2}.
 
We give an outline of the proof.
We chase the following diagram:
\[
\xymatrix{\sD^b(\widetilde{Y})\ar[r]^{\Phi_1^*}\ar[d]_{\alpha^*} & \sD^b(X)\ar[d]^{\beta_*}\\
          \sD^b(S)\ar[r]^{\Phi_2^*} & \sD^b(\hW).}
\]
This diagram is not commutative but below
we may evaluate $\Phi_2^*\alpha^*(A)$ in Step 1, 
$\beta_*\Phi_1^*(A)$ in Step 2, and
their difference in Step 3.
From these,
we deduce $\Phi_2^*\alpha^*(A)=0$ in Step 4.
Until Step 4, we just follow the arguments given in 
\cite[\S 5]{K-g7} and \cite[\S 6]{HPD1}.
In Steps 5--7, we study more detailed geometries.
In Step 5, we show $\alpha^*(A)=0$
computing the functor $\Phi_2$ based on the flop equivalence
$\sD^b(\hW)\simeq \sD^b(Z_S)$ (Proposition \ref{prop:flop} (2)).
Then we deduce $A\in \oplus \sD^b(F_i)$ in Step 6 by \cite[Lem.~4.5]{K-g7}.
In Step 7, we finish showing $A=0$
by studying a relation between $\Ima \Phi_1$ and $\oplus \sD^b(F_i)$ in detail.

\vspace{5pt}

{\bf Step 1.} We show
\begin{equation}
\label{eq:CW}
\Phi_2^*\alpha^*(A)\in \sC_{W},
\end{equation}
which immediately follows from the following lemma:
\begin{lem}
\label{lem:CS}
$\Ima \Phi_2^*\subset \sC_{W}$.
\end{lem}
\begin{proof}
Considering the adjunction, it suffices to show
that $\Phi_2(B)=0$ for any $B\in \sC_{W}^{\perp}$.
By the definition of $\sC_{W}$, we may assume that
\[
B=\sO_{F_i}(-2)\, (1\leq i \leq 8),\, \sO_{\hW}(-H),\, \sO_{\hW}(-L),\,
\text{or}\
\sF.
\]
Let $\hB$ be $\sO_{E_f}(E_f)$ if $B=\sO_{E_i}(-2)$,
and the corresponding locally free sheaf on $\hchow$ if otherwise.
By the argument to show (\ref{eq:aim}), it suffices to show that, for any $s\in S$,
$H^{\bullet}(\hW,\sP_{s;2}\otimes B)=\{0\}$,
where $\sP_{s;2}$ is the restriction of $\sP_2$ over $s$,
and this 
follows by showing 
$H^{\bullet}(\hchow,{j_s}_*\sP_s\otimes \hB(-t))=\{0\}$
for $t=0,1,2$.
If $B=\sO_{E_i}(-2)$, then this follows from (\ref{eqnarray:onVsy})
since
the restrictions of 
$\Omega^1_{\hchow/\mathrm{G}(2,V)}(-L)$,
$\sF^*(-H)$, 
$\sO_{\hchow}$, and
$\sO_{{\hchow}}(L-H)$ to the fibers of $E_f\to v_2(\mP(V))$
are direct sums of $\sO_{\mP^2}$ and $\sO_{\mP^2}(1)$.
Otherwise, the assertion follows from (\ref{eqnarray:onVsy})
and \cite[Thm.~3.1]{DerSym}.
\end{proof}
\noindent {\bf Step 2.} Since $A\in {\empty^{\perp}\langle \Phi_1(\sD^b(X))\rangle}$,
we have $\Phi_1^*(A)=0$ by adjunction. In particular, we have 
\begin{equation}
\label{eq:betaA}
\beta_*\Phi_1^*(A)=0.
\end{equation}
{\bf Step 3.} Now we estimate 
the difference between $\Phi_2^*\alpha^*(A)$ and 
$\beta_*\Phi_1^*(A)$.

To formulate the claim precisely,
we closely follow the argument of \cite[\S 6]{HPD1}.
By the commutative diagram
\begin{equation*}
\label{eq:SXSW}
\xymatrix{S\times X\ar@{^{(}->}[r]^{\beta}\ar@{^{(}->}[d]^{\alpha}& S\times \hW
\ar@{^{(}->}[d]^{\alpha}\\
\widetilde{Y}\times X\ar@{^{(}->}[r]^{\beta} & \widetilde{Y}\times \hW,}
\end{equation*}
we have $\alpha^*\sP_1\simeq \beta^*\sP_2$.
Consider the following cartesian product (cf.~\cite[Diagram (19) in \S 6]{HPD1}):
\begin{equation}
\label{eq:YXSW}
\xymatrix{
(\widetilde{Y}\times X)\cup (S\times \hW)\ar@{^{(}->}[r]^{\ \ \ \ \ \ \ \ i}\ar@{^{(}->}[d]_{\widetilde{\xi}} & \widetilde{Y}\times \hW\ar@{^{(}->}[d]^{\xi}\\
\Vs\ar@{^{(}->}[r]^{\iota_{\Vs}} & \widetilde{\hcoY}\times \hchow,
}
\end{equation}
where both horizontal arrows represent divisorial embeddings.
We set 
\[
\widetilde{\sP}:=\widetilde{\xi}^*\sP.
\] 
Then, by the exact sequence
\[
0\to \sO_{S\times \hW}(0,-1)\to \sO_{(\widetilde{Y}\times X)\cup (S\times \hW)}\to
\sO_{\widetilde{Y}\times X}\to 0,
\]
we have an exact triangle on $\widetilde{Y}\times \hW$
\[
\alpha_*\sP_2(0,-1)\to i_*\widetilde{\sP}\to \beta_*\sP_1,
\]
which in turn gives an exact triangle of functors
from $\sD^b(\hW)$ to $\sD^b(\widetilde{Y})$:
\[
\Phi_1\beta^{!}\to \alpha_*\Phi_2\to \Phi_{i_*\widetilde{\sP}(0,1)},
\]
where $\Phi_{i_*\widetilde{\sP}(0,1)}$ is the functor with the kernel $i_*\widetilde{\sP}(0,1)$.
Moreover, taking the left adjoints of these functors, we obtain
an exact triangle of functors from $\sD^b(\widetilde{Y})$ to $\sD^b(\hW)$:
\begin{equation}
\label{eq:ex-tri}
\Phi^*_{i_*\widetilde{\sP}(0,1)}\to \Phi_2^*\alpha^* \to \beta_*\Phi_1^*
\end{equation}
 (cf.~\cite[Lem.~6.14 and Cor.~6.15]{HPD1}). 
Our task is to estimate $\Phi^*_{i_*\widetilde{\sP}(0,1)}(A)$.

\begin{lem}
\label{lem:Phi(0,1)}
$\Phi^*_{i_*\widetilde{\sP}(0,1)}(A)$ belongs to $\sD^3_{W}(-H)$,
where 
\[
\sD^3_{W}:=
\langle \sO_{\hW}, \sO_{\hW}(H-L), \sF(H)\rangle\subset \sD^b(\hW).
\]

\end{lem}

\begin{proof}
More generally, we show the claim for
\[
A\in 
{\empty^{\perp}\langle \{\sO_{F_i}(-1,-1)\}_{1\leq i\leq 10} \rangle\cap
\langle \sO_{\widetilde{Y}}(1),
\sO_{\widetilde{Y}}(2)\rangle}^{\perp}.
\]
The assertion is equivalent to
\[
\Phi^*_{i_*\widetilde{\sP}}(A)\in \sD^3_{W},
\]
which we will prove by following closely the proof of \cite[Lem.~6.18]{HPD1}.
Consider the following diagram:
\begin{equation}
\label{eq:618}
\xymatrix{
\widetilde{Y}\ar@{=}[d] & \widetilde{Y}\times \hW\ar[r]_q\ar[l]\ar@{^{(}->}[d]_j & \hW\ar@{^{(}->}[d]_j
\\
\widetilde{Y}\ar@{^{(}->}[d]_{\xi} & \widetilde{Y}\times \hchow\ar[r]_q\ar[l]\ar@{^{(}->}[d]_{\xi} & 
\hchow\ar@{=}[d]
\\
\widetilde{\hcoY} & \widetilde{\hcoY}\times \hchow\ar[r]_q\ar[l] & 
\hchow,
}
\end{equation} 
where the vertical arrows are closed embeddings and 
the horizontal arrows are natural projections.
Since the diagram (\ref{eq:YXSW}) is an exact cartesian by \cite[Lem.~2.32]{HPD1}, 
we have \[
i_*\widetilde{\sP}\simeq \widetilde{\xi}^*{{\iota_{\Vs}}}_*\sP
\simeq 
j^*{\xi}^*{{\iota_{\Vs}}}_*\sP.
\]
Therefore, since
the upper-right square of the diagram (\ref{eq:618}) is also an exact cartesian, we have 
$\Phi_{i_*\widetilde{\sP}}\simeq \Phi_{\xi^*{\iota_{\Vs}}_*\sP}\circ j_*$
as functors $\sD^b(\hW)\to \sD^b(\widetilde{Y})$
(see an upper part of p.196 in the proof of \cite[Lem.~6.18]{HPD1} for details).
Taking the left adjoint, we have
$\Phi^*_{{i}_*\widetilde{\sP}}\simeq j^*\circ \Phi^*_{\xi^*{\iota_{\Vs}}_*\sP}$.
Thus it suffices to show
\[
\Phi^*_{\xi^*{{\iota_{\Vs}}_*\sP}}(A)\in \sD^3_{\chow},
\]
where 
$\Phi_{\xi^*{{\iota_{\Vs}}_*\sP}}$ is 
a functor from
$\sD^b(\hchow)$ to $\sD^b(\widetilde{Y})$.
Since the lower-left square is an exact cartesian
by \cite[Lem.~2.32]{HPD1},
we have
\begin{equation}
\label{eq:piphi}
\Phi_{\xi^*{\iota_{\Vs}}_*\sP}\simeq \xi^*\circ \Phi_{{\iota_{\Vs}}_*\sP}.
\end{equation}
(see a lower part of p.196 in the proof of \cite[Lem.~6.18]{HPD1} for details).
Since $\widetilde{Y}$ is a complete intersection 
in $\widetilde{\hcoY}$ by six members of $|M|$,
we have 
$\xi^*=\otimes \sO_{\widetilde{Y}}(-6)[6]\circ \xi^{!}$.
Then, by taking the left adjoint of (\ref{eq:piphi}),
we obtain
$\Phi^*_{\xi^*{\iota_{\Vs}}_*\sP}\simeq \Phi^*_{{\iota_{\Vs}}_*\sP}\circ \xi_*\circ \otimes \sO_{\widetilde{Y}}(6)[-6]$.
Therefore 
it remains to show 
that
$\Phi^*_{{\iota_{\Vs}}_*\sP}\circ \xi_*(A(6))\in
\sD^3_{\hchow}$.
To compute the functor $\Phi^*_{{\iota_{\Vs}}_*\sP}$,
we calculate the Fourier-Mukai functors with the kernels 
appearing in the exact sequence (\ref{eqnarray:onVs'}).
Then, by taking account of \cite[Lem.~2.28]{HPD1},
it suffices to show $H^{\bullet}(\widetilde{\hcoY}, \xi_*(A(6))\otimes \omega_{\widetilde{\hcoY}})=0$.
Note that
\begin{eqnarray*}
H^{\bullet}(\widetilde{\hcoY}, \xi_*(A(6))\otimes \omega_{\widetilde{\hcoY}})
\simeq H^{\bullet}(\widetilde{\hcoY}, 
\xi_*(A(-2M+\sum_{i=1}^{10} F_i)))\simeq\\
H^{\bullet}(\widetilde{Y}, 
A(-2M+\sum_{i=1}^{10} F_i))\simeq
\mathrm{RHom}^{\bullet} (\sO_{\widetilde{Y}}(2M-\sum_{i=1}^{10} F_i), A).
\end{eqnarray*} Consider the exact sequence
\begin{equation}
\label{eq:2M}
0\to \sO_{\widetilde{Y}}(2M-\sum_{i=1}^{10} F_i)\to
\sO_{\widetilde{Y}}(2M)\to \oplus_{i=1}^{10} \sO_{F_i}\to 0.
\end{equation}
Since $A\in \langle \sO_{\widetilde{Y}}(1),
\sO_{\widetilde{Y}}(2)\rangle^{\perp}$,
we have 
$\mathrm{RHom}^{\bullet} (\sO_{\widetilde{Y}}(2M), A)=0$.
Moreover,
since 
\[
A\in {\empty^{\perp}\langle \{\sO_{F_i}(-1,-1)\}_{1\leq i\leq 10} \rangle},
\]
we have
$\mathrm{RHom}^{\bullet} (\sO_{F_i}, A)\simeq 
\mathrm{RHom}^{3-\bullet}(A, \sO_{F_i}(-1-1))=0$,
where the first isomorphism is given by the Serre duality.
Therefore, by the exact sequence (\ref{eq:2M}),
we finally obtain 
$\mathrm{RHom}^{\bullet} (\sO_{\widetilde{Y}}(2M-\sum_{i=1}^{10} F_i), A)=0$.
\end{proof}
\noindent{\bf Step 4.} We derive \[
\Phi_2^*\alpha^*(A)=0.
\]

Indeed, by (\ref{eq:betaA}), (\ref{eq:ex-tri}), and Lemma \ref{lem:Phi(0,1)},
we have
$\Phi_2^*\alpha^*(A)\in \sD^3_{W}(-H)$.
Combining this with (\ref{eq:CW}),
we obtain $\Phi_2^*\alpha^*(A)\in \sC_{W}\cap \sD^3_{W}(-H)$,
which implies that $\Phi_2^*\alpha^*(A)=0$
since $\sD_W=\langle \sD_W^3(-H),\sC_W\rangle$.
\vspace{5pt}

\noindent {\bf Step 5.}
We show $\alpha^*(A)=0$.

Since $A\in \langle \sO_{\widetilde{Y}}(1),
\sO_{\widetilde{Y}}(2)\rangle^{\perp}$,
we have $\alpha^*(A)\in \langle \sO_{S}(1)\rangle^{\perp}=\sC_S$.
Noting this, we show a more general claim;
\begin{equation}
\label{Phi_2}
\text{For $B\in \sC_S$, if $\Phi_2^*(B)=0$, then $B=0$.}
\end{equation}
Note that for $C\in \sD^b(\hW)$,
it holds that $\Hom (B,\Phi_2(C))=\Hom(\Phi_2^*(B),C)=0$ if $\Phi_2^*(B)=0$.
Therefore, if we show 
\begin{equation}
\label{eq:generate}
\text{$\Ima \Phi_2$ generates $\sC_S$},
\end{equation}
then we have $B=0$ as desired in (\ref{Phi_2}).
We show the claim (\ref{eq:generate}).

First we see that $\Phi_2(\sO_{\hW})=\widetilde{\sQ}|_S$,
and in particular, 
\[
\widetilde{\sQ}|_S\in \Ima \Phi_2.
\]
Let $p_1\colon S\times \hW\to S$ be the first projection.
Then the assertion is equivalent to
${p_1}_*\sP_2\simeq \widetilde{\sQ}|_S$ and
$R^{\bullet}{p_1}_*\sP_2=0$ for $\bullet>0$.
We denote by $\sP_{s;2}$ the restriction of $\sP_2$ over a point 
$s\in S$. Since $\sP_2$ is flat over $S$ by Proposition \ref{prop:flatP1}, the problem is reduced to compute
$H^{\bullet}(\hW,\sP_{s;2})$.
By \cite[Thm.~3.1]{DerSym}, we have  
the vanishings of the cohomology groups
\[
H^{\bullet}(\hchow, \Omega^1_{\hchow/\mathrm{G}(2,V)}(-L-tH)),\,
H^{\bullet}(\hchow, \sF^*(-(t+1)H)),\,
H^{\bullet}(\hchow, \sO_{{\hchow}}(L-(t+1)H))
\] for $t=0,1,2$, and
$H^{\bullet}(\hchow, \sO_{{\hchow}}(-tH))$ 
for $t=1,2$.
Thus by (\ref{eqnarray:onVsy}), 
we obtain 
$H^0(\hW,\sP_{s;2})\simeq \mC^3$,
and $H^{\bullet}(\hW,\sP_{s;2})=0$ for $\bullet>0$.
This implies the assertion by taking account of (\ref{eqnarray:onVs'}).

To proceed, we describe the properties of 
\[
\Delta_2:=\Delta|_{S\times \hW}.
\]
We denote by $p\colon \Delta_2\to S$, $q\colon \Delta_2\to \hW$,
$\rho_1\colon Z_S\times \hW\to Z_S$, and
$\rho_2\colon Z_S\times \hW\to \hW$
the natural projections.
We also denote by $\rho\colon \Delta_2\to Z_S$ and
$\pi\colon Z_S\to S$ the natural morphisms factoring $p$,
and by $j\colon \Delta_2 \hookrightarrow Z_S\times \hW$ 
the natural closed embedding.
\[
\xymatrix{& & Z_S \times \hW\ar[ddl]_{\rho_1}\ar[ddr]^{\rho_2} & \\ & & \Delta_2\ar@{^{(}->}[u]_{j}\ar[dl]_{\rho}\ar[dr]^{q}\ar@/^1.5pc/[ddll]^p &  \\&  Z_S\ar[dl]_{\pi} & & \hW & \\ S & & & }\] 

\begin{lem}
\label{prop:Delta2descrip}
$q\colon \Delta_2\to \hW$ is the blow-up along
the flopping curves of $\hW\to \overline{W}$,
and
$\rho\colon \Delta_2\to Z_S$ is 
the blow-up along
the flopped curves of $Z_S\to \overline{W}$,
where we recall that $\overline{W}\subset \mathrm{G}(2,V)$ is the image of 
$Z_S$ and $\hW$.
\end{lem}

\begin{proof}
We recall that $\widetilde{\Zpq}\to \widetilde{\hcoY}$ and $\Zpq\to \hcoY$ coincide over the locus of rank $\geq 3$ points by \cite[Prop.~4.20]{Geom}.
Therefore the restriction of $\widetilde{\Zpq}$ over $S$ coincides with $Z_S$.
Then, by (\ref{eq:defDelta}), we may consider
\[
\Delta_2=(\widetilde{\Zpq}\times_{\mathrm{G}(2,V)} \hchow)|_{Z_S\times \hW}
=Z_S\times_{\overline{W}}\hW.
\]
Since $\hW\dashrightarrow Z_S$ is an Atiyah's flop by Proposition \ref{prop:flop} (2), it is well-known that
$Z_S\times_{\overline{W}}\hW$ has properties as described in the statement.
\end{proof}

Now we show that, for any object $C\in
\sC_S$, there exists the following exact triangle:
\begin{equation}
\label{eq:finaleq}
H^{\bullet}(S,C)\otimes {\widetilde{\sQ}}|_S\to
\Phi_2(q_*p^*C)\to C[-1]
\to H^{\bullet+1}(S,C)\otimes {\widetilde{\sQ}}|_S.
\end{equation}
This immediately implies 
(\ref{eq:generate}) since $\widetilde{\sQ}|_S\in \Ima \Phi_2$.

The key ingredient to show (\ref{eq:finaleq}) is 
the following exact sequence relating $\sP_2$ and $\Delta_2$,
which is derived in the appendix \ref{app:B} (the subsection \ref{cutting}):
\begin{equation}
\label{eq:ShW}
0\to \sO_{S}\boxtimes \sO_{\hW}(-L)\to
\widetilde{\sQ}|_S\boxtimes \sO_{\hW}\to \sP_2\to 
\omega_{\Delta_2/S}\to 0.
\end{equation}

We derive the exact triangle (\ref{eq:finaleq})
by computing the Fourier-Mukai functors 
from $\sD^b(\hW)$ to $\sD^b(S)$
whose kernels are the terms of
(\ref{eq:ShW}).
\begin{itemize}
\item
$\Phi_{\sO_{S}\boxtimes \sO_{\hW}(-L)}(q_*p^*C)=0$.

Indeed,
we have
$H^{\bullet}(\hW,q_*p^* C\otimes \sO_{\hW}(-L))
\simeq H^{\bullet}(Z_S,\pi^*C\otimes \sO_{Z_S}(-L))$
since $q_*\rho^*$ is an equivalence and $q_*\rho^*\sO_{Z_S}(L)=\sO_{\hW}(L)$.
Then the assertion follows from the following 
chain of isomorphisms:
\begin{eqnarray*}
H^{\bullet}(Z_S,\pi^*C\otimes \sO_{Z_S}(-L))\simeq
H^{\bullet}(Z_S,\pi^*C\otimes \omega_{Z_S})\simeq\\
\mathrm{RHom}^{3-\bullet}(\pi^*C, \sO_{Z_S})^*\simeq 
\mathrm{RHom}^{3-\bullet}(C, \sO_{S})^*\simeq 
\mathrm{RHom}^{\bullet-1}(\omega_S^{-1}, C)=0,
\end{eqnarray*}
where the second isomorphism follows from the duality on $Z_S$,
the third isomorphism follows from fully faithfulness of $\pi^*$ \cite{Or} (note that $\pi$ is a $\mP^1$-bundle), the fourth isomorphism follows from
the duality on $S$, and 
the last equality follows from $C\in \sC_S=\langle\omega_S^{-1}\rangle^{\perp}$.
\item
$\Phi_{\widetilde{\sQ}|_S\boxtimes \sO_{\hW}}(q_*p^*C)=
H^{\bullet}(S,C)\otimes {\widetilde{\sQ}}|_S$.

Indeed, we have
$H^{\bullet}(\hW,q_*p^* C)\simeq H^{\bullet}(Z_S,\pi^*C)$
since $q_*\rho^*$ is an equivalence and $q_*\rho^*\sO_{Z_S}=\sO_{\hW}$.
Moreover,
$H^{\bullet}(Z_S,\pi^* C)\simeq H^{\bullet}(S,C)$
since $\pi\colon Z_S\to S$ is a $\mP^1$-bundle, and then $\pi^*$ is 
fully faithful by \cite{Or}.
\item
$\Phi_{\omega_{\Delta_2/S}}(q_*p^*C)=C[-1]$.

Note that the l.h.s. precisely stands for
$\Phi_{j_*\omega_{\Delta_2/S}}(q_*p^*C)$.
Since $Z_S\dashrightarrow \hW$ is an Atiyah's flop by Proposition \ref{prop:flop},
the functor $q_*\rho^*$ is the Fourier-Mukai functor
with the kernel $j_*\sO_{\Delta_2}$ by \cite{BO}.
By \cite[Lem.~2.2.8]{HPD1},
the left adjoint of $\Phi_{j_*\sO_{\Delta_2}}\colon \sD^b(Z_S)\to \sD^b(\hW)$ is 
the Fourier-Mukai functor with the kernel
\begin{eqnarray*}
(j_*\sO_{\Delta_2})^{\#}:=
R\sH om(j_*\sO_{\Delta_2},\omega_{\hW\times Z_S/Z_S}[3])\simeq\\
j_*R\sH om(\sO_{\Delta_2},j^{!}\omega_{\hW\times Z_S/Z_S}[3])\simeq
j_*\omega_{\Delta_2/Z_S}.
\end{eqnarray*}
Since $\Phi_{j_*\sO_{\Delta_2}}$ is an equivalence,
we have
\[
\Phi_{(j_*\sO_{\Delta_2})^{\#}}\circ \Phi_{j_*\sO_{\Delta_2}}(\pi^*C)\simeq \pi^*C.
\]
Since $\omega_{\Delta_2/S}=\omega_{\Delta_2/Z_S}\otimes \rho^*\omega_{Z_S/S}$,
we have
\begin{align*}
\Phi_{j_*\omega_{\Delta_2/S}}(q_*p^*C)\simeq
\Phi_{j_*\omega_{\Delta_2/S}}\circ \Phi_{j_*\sO_{\Delta_2}}(\pi^*C)\simeq\\
\Phi_{j_*\omega_{\Delta_2/Z_S}\otimes {\rho_1}^*\omega_{Z_S/S}}\circ \Phi_{j_*\sO_{\Delta_2}}(\pi^*C)
\simeq 
\pi_*(\pi^*C\otimes \omega_{Z_S/S})
\simeq C[-1],
\end{align*}
where the last isomorphism follows 
by Theorem \ref{cla:duality} since $Z_S\to S$ is a $\mP^1$-bundle.
\end{itemize}
Therefore we obtain (\ref{eq:finaleq}).

\begin{rem}
Step 5 indicates
that the functor $\Phi_2$ induces
an equivalence between $\sC_{W}$ and $\sC_S$
in Theorem \ref{thm:main2}.
\end{rem}

\vspace{5pt}

\noindent {\bf Step 6.} We show
$A\in \oplus \sD^b(F_i)$.

By Step 5, we have $\alpha^*(A)=0$ for 
any smooth $S\subset \widetilde{Y}$
such that its orthogonal linear section $W$ has only
$\frac 12(1,1,1)$-singularities. 
Then, by \cite[Lem.~4.5]{K-g7} and the following lemma, we see that
the supports of cohomologies of $A$ are contained in $\cup_{i=1}^{10} F_i$.
\begin{lem}
\label{lem:SW}
For any point $y\in \widetilde{Y}\setminus \cup_{i=1}^{10} F_i$,
we may choose a smooth $S$ through $y$
such that its orthogonal linear section $W$ has only
eight $\frac 12(1,1,1)$-singularities.
\end{lem}

\begin{proof}
By the Bertini theorem, a general $S$ through $y$ is smooth.
Let $H_y\subset \mP({\ft S}^2 V)$ be the hyperplane corresponding to $y$,
and $Q_y\subset \mP(V)$ the quadric corresponding to the image of $y$ in
$\mP({\ft S}^2 V^*)$. Note that $\rank Q_y=3$ or $4$.
Then $H_y$ cut out $Q_y$ from $v_2(\mP(V))$
by the projective duality of $v_2(\mP(V))$.
By the orthogonality between $X$ and $Y$,
$X$ is contained in $H_y$.
Let $\Gamma$ be the linear system of hyperplane sections
of $\chow\cap H_y$ containing $X$.
Then a codimension three linear section $W$ of $\chow$ 
such that $X\subset W$ and its orthogonal linear section of $\widetilde{\hcoY}$
contains $y$ is of the form $W=\chow\cap H_y \cap H_1\cap H_2$ with
$H_1,H_2\in \Gamma$. We have only to show 
such a general $W$ has only
eight $\frac 12(1,1,1)$-singularities, equivalently,
$v_2(\mP(V))\cap H_y\cap H_1\cap H_2=Q_y\cap H_1\cap H_2$ consists of 
eight points for general $H_1$ and $H_2$.
This follows since 
the base locus of $\Gamma$ is $X$ and $X\cap v_2(\mP(V))=\emptyset$.
\end{proof}

Therefore, we have shown $A\in \oplus \sD^b(F_i)$.
\vspace{5pt}

\noindent {\bf Step 7.}
We finish the proof of (\ref{eq:A0}).

For this,
we need to investigate the relationship between $\sD^b(F_i)$ and $\Ima \Phi_1$
(the arguments become related to the result of \cite{IK}.
We continue arguments in the next subsection). 

First we review some classical geometries of $X$.
It is well-known that $X$ has ten elliptic fibrations $\pi_i\colon X\to \mP^1$
($1\leq i\leq 10$), and each fibration has two multiple fibers.
Let $\delta_i$, $\delta'_i$ be the supports of the multiple fibers of $\pi_i$.
In \cite[Lem.~5.13]{IK}, these are described by geometries of $Y$ as follows.
We denote by $\mP(V_3^a)\cup \mP(V_3^b)$ the rank two quadric corresponding to
the singular point $y_i\in Y$.
Let $\mP_k= 
\{\mC^2\mid \mC^2\subset V_3^k\}$ ($k=a,b$),
which is a $\sigma$-plane.
By \cite[Prop.~3.7]{Geom},
the fiber of $Z\to Y$ over $y_i$ is $\mP_a\cup_{1pt}\mP_b$.
Then $\mP_a\cap X$ and $\mP_b\cap X$ are the supports of 
two multiple fibers of an elliptic fibration.
From now on, we denote by $\pi_i$ this elliptic fibration and we set
\[
\delta_i:=\mP_a\cap X,\, \delta'_i:=\mP_b\cap X.
\] 
By this choice, we may consider
\[
\text{$\delta_i$ and $\delta'_i$ correspond to $\sO_{F_i}(1,0)$
and $\sO_{F_i}(0,1)$, respectively.}
\]

\begin{lem}
\label{lem:IK}
\begin{equation}
\label{eq:Ei}
\Phi_1(\sO_X(-\delta_i)), \Phi_1(\sO_X(-\delta'_i))\in \sD^b(F_i),
\end{equation}
and
$\sD^b(F_i)$ admits the following semi-orthogonal decomposition\,$:$
\begin{equation}
\label{eq:decEi}
\sD^b(F_i)=\langle \sO_{F_i}(-1,-1),
\Phi_1(\sO_X(-\delta_i)), \Phi_1(\sO_X(-\delta'_i)),
\sO_{F_i}\rangle.
\end{equation}
\end{lem}

\begin{proof}
We only treat $\delta_i$ for the claim (\ref{eq:Ei})
by symmetry of $\delta_i$ and $\delta'_i$.
For this claim, it suffices to show that
\[
H^{\bullet}(X, \sP_{y;1}\otimes \sO_X(-\delta_i))=0
\ \text{for any $y\in \widetilde{\hcoY}\setminus F_i$},
\]
where we recall that $\sP_{y;1}$ is the restriction of $\sP_1$ over $y$.
Since $\mP_a$ is a $\sigma$-plane as above,
its ideal sheaf $\sI_{\mP_a}$ has the following locally free resolution:
\begin{equation}
\label{eq:sigma}
0\to \sO_{\mathrm{G}(2,V)}(-1)\to \sF\to \sI_{\mP_a}\to 0.
\end{equation}
By considering $X$ to be contained in $\hchow$,
$\delta_i$ is the intersection between $X$ and $g^{-1}(\mP_a)$.
By (\ref{eq:sigma}), 
the ideal sheaf $\sI_{g^{-1}(\mP_a)}$
of ${g^{-1}(\mP_a)}$ on $\hchow$ has the following
locally free resolution:
\begin{equation}
\label{eq:sigma'}
0\to \sO_{\hchow}(-L)\to \sF\to \sI_{g^{-1}(\mP_a)}\to 0,
\end{equation}
where we have suppressed $g^*$ in $g^*\sF$ for simplicity.
Now we represent $X$ as the complete intersection of four members
$H_1,\dots, H_4$ of $|\sO_{\hchow}(1)|$ such that, by projective duality, 
$H_1$ corresponds to the image on $\hcoY$ of $s$ and 
$H_2$ corresponds to $y_i$, where we recall that $y_i$ is the image of $F_i$
on $Y$.  
This choice of $H_1$ and $H_2$ is crucial for our argument;
$H_1\not= H_2$ since the image on $\hcoY$ of $s$ is different from $y_i$.
Note that $H_2$ contains $g^{-1}(\mP_a)$.
Then, by a similar argument to the proof of (\ref{eq:aim}),
we obtain the desired vanishing of
$H^{\bullet}(X, \sP_{y;1}\otimes \sO_X(-\delta_i))$
by (\ref{eqnarray:onVsy}), (\ref{eq:sigma'})
and \cite[Thm.~5.1]{DerSym}.

Now we derive the decomposition (\ref{eq:decEi}).
By Lemma \ref{lem:subset} (2),
\[
\Phi_1(\sO_X(-\delta_i)), \Phi_1(\sO_X(-\delta'_i))
\in {\empty^{\perp}\langle \sO_{F_i}(-1,-1)\rangle}
\cap \langle \sO_{F_i}\rangle^{\perp}
\] in $\sD^b(F_i)$.
Since we have shown that $\Phi_1$ is fully faithful in the subsection \ref{subsection:ff},
$\Phi_1(\sO_X(-\delta_i))$, and $\Phi_1(\sO_X(-\delta'_i))$ are exceptional objects on
$\sD^b(F_i)$.
Thus we obtain an exceptional collection; 
\begin{equation*}
\sO_{F_i}(-1,-1),
\Phi_1(\sO_X(-\delta_i)), \Phi_1(\sO_X(-\delta'_i)),
\sO_{F_i}.
\end{equation*}
We conclude that this collection is full by Theorem \ref{thm:delPezzo}.
\end{proof}
Now the claim (\ref{eq:A0}) follows from Step 6 and Lemma \ref{lem:IK}, and then we have finished
our proof of the
fullness of the collection in Theorem \ref{thm:main}.
$\hfill\square$
\subsection{Relation with the result of Ingalls and Kuznetsov}
\label{sub:IK}
In this subsection, we give a new proof of \cite[Theorem 4.3]{IK} 
using the functor $\Phi_1$.
For this, we refine Lemma \ref{lem:IK}.

We describe 
\[
\Delta_1:=\Delta|_{\widetilde{Y}\times X}
\]
with the natural projections
\[
\text{$p\colon \Delta_1\to \widetilde{Y}$ and $q\colon \Delta_1\to X$}.
\] 

\begin{lem}
\label{Delta1fiber}
\begin{enumerate}[$(1)$]
\item
Any fiber of $q\colon \Delta_1\to X$ is a tree of $\mP^1$'s,
and
a general fiber is the strict transform
of a copy of $\mP^1$ in $Y$ of degree $1$ with respect to $M$.
In particular, $R^{\bullet} q_*\sO_{\Delta_1}=0$ for $\bullet>0$.
\item
$p\colon \Delta_1\to \widetilde{Y}$ 
is a finite morphism outside a finite set 
of points on $\widetilde{Y}$.
A $0$-dimensional fiber of $p$ is a length six subscheme of $X$.
A positive dimensional irreducible component
 of a fiber is a line or a conic on $X$ with respect to
$\sO_X(L)$.
Moreover, for any point $y\in F_i$,
the fiber of $p$ over $y$ is a $0$-dimensional subscheme 
\[
\eta_y=\eta+\eta',
\]
where $\eta\sim L|_{\delta_i}$ and $\eta'\sim L|_{\delta'_i}$
as divisors on $\delta_i$ and $\delta'_i$, respectively.

\end{enumerate}

\end{lem}

\begin{proof}
(1) 
Let $x$ be a point of $X$.
By Proposition \ref{prop:Deltasmooth} (1),
$q^{-1}(x)$ parameterizes conics in $\mathrm{G}(2,V)$
containing $[l_x]$ and corresponding to quadrics in $P_3$
(we recall that $\widetilde{Y}\cap \Prt_{\sigma}=\emptyset$).
By Proposition \ref{prop:flop} (1),
quadrics in $P_3$ containing $l_x$ form a line in $P_3$.
If this line does not pass through singular points of $H$,
then $q^{-1}(x)$ is isomorphic to this line.
This gives the description of a general fiber of $q$ as in the statement.

Assume that this line passes through at least one singular point
of $H$. 
Let $Q$ be the rank 2 quadric corresponding to this singular point.

We show that $l_x$ is not contained in the singular locus of $Q$. 
Actually, this is a classically well-known result
(cf.~\cite[Ex.~VIII.19 ($H_2$)]{Be}).
We give a proof here for readers' convenience.
By the projective duality between $\chow={\ft S}^2 \mP(V)$ and
$\Sing \Hes={\ft S}^2 \mP(V^*)$,
$Q$ corresponds to a hyperplane $H_Q\subset \mP({\ft S}^2 V)$
tangent to $\chow$ and, moreover, $\chow\cap H_Q$ is singular along ${\ft S}^2 \mP(V_2)$.
Since $f(x)\in \chow$ is contained in ${\ft S}^2 \mP(V_2)$,
$\chow\cap H_Q$ is singular at $f(x)$.
This is a contradiction since
$X$ can be written as $X=\chow\cap H_Q\cap H_1\cap H_2\cap H_3$
with three hyperplanes $H_1,H_2,H_3$, and then would be singular at $f(x)$.

We write $Q=\mP(V_3^a)\cup \mP(V_3^b)$.
We may assume that $l_x\subset \mP(V_3^a)$ and
$l_x\not \subset \mP(V_3^b)$ by the previous paragraph.
Let $l_a\cup l_b$ be a rank two conic corresonding to 
$Q$, 
where $l_k:=\{\mC^2\mid V_1^k\subset \mC^2\subset V_3^k\}\, (k=a,b)$
with 
$1$-dimensional subspaces $V_1^k\subset V_3^a\cap
V_3^b$.
Then $l_a\cup l_b$ contains $[l_x]$ if and only if
$[V_1^a]\in l_x$.
Under this condition, $[V_1^a]$ is determined as 
$[V_1^a]=l_x\cap \mP(V_3^a\cap V_3^b)$.
Therefore the conics  
$l_a\cup l_b$ containing $[l_x]$ form a copy of $\mP^1$
as $l_b$ varies.
Consequently, any component of $q^{-1}(x)$ is a $\mP^1$.

The second assertion in (1) follows from the first.

\vspace{5pt}

\noindent (2)
Let $y$ be a point of $\widetilde{Y}$.
Since $\widetilde{Y}\cap \Prt_{\sigma}=\emptyset$,
the fiber of $\widetilde{\Zpq}\to \widetilde{\hcoY}$ over $y$
is a conic, which we denote by $q_y$ (Proposition \ref{prop:cobu}).
Then, by Proposition \ref{prop:Deltasmooth} (2),
$p^{-1}(y)=g^{-1}(q_y)\cap X$,
where we consider $X\subset \hchow$.

Let $y$ be a point of $\widetilde{Y}$ such that $\dim p^{-1}(y)=0$.
We write $X=H_1\cap H_2\cap H_3\cap H_4$,
where $H_i\in |\sO_{\hchow}(1)|$ $(1\leq i\leq 4)$ and
$H_1$ corresponds to the image of $y$ on $P_3$ by the projective duality.
Then the $\mP^2$-bundle $g^{-1}(q_y)$ is contained in $H_1$, and 
hence $p^{-1}(y)=g^{-1}(q_y)\cap H_2\cap H_3\cap H_4$.
We see the degree of the r.h.s. is six since
$\deg {\ft S}^2 \sF|_{q_y}=6$.

From now on we consider $X\subset \mathrm{G}(2,V)$.
Then $p^{-1}(y)=q_y\cap X$. 

If a positive dimensional subvariety of $\widetilde{Y}$ is contained in 
the subset of $y$'s such that $\dim p^{-1}(y)>0$,
then $X$ is covered by the images of components of $q_y$ for such $y$'s, a contradiction since $X$ is not ruled.
Therefore, there exists at most finite number of such $y$'s. 

Now let $y$ be a point of $F_i$.
Let
$l_a\cup l_b$ be the rank two conic corresonding to 
$y$, where
$l_k:=\{\mC^2\mid V_1^k\subset \mC^2\subset V_3^k\}\, (k=a,b)$
with $3$-dimensional subspaces $V_3^k$ as in Step 7 
of the proof of (\ref{eq:A0}) and 
$1$-dimensional subspaces $V_1^k\subset V_3^a\cap
V_3^b$. 
Then the fiber of $p$ over $y$ is $(l_a\cup l_b)\cap X$.
It is equal to
$(l_a\cap \delta_i)\cup (l_b\cap \delta'_i)$
since $\mP_a\cap X=\delta_i$ and $\mP_b\cap X=\delta'_i$. 
If $l_a\subset \delta_i$ or $l_b\subset \delta'_i$,
then $\mP_a\cap \mP_b\in X$.
This contradicts  
the third paragraph in the proof of Lemma \ref{Delta1fiber} (1)
since $\mP_a\cap \mP_b$ corresponds to
the line contained in $\mP(V_3^a)\cap \mP(V_3^b)$.
Thus $(l_a\cup l_b)\cap X$ 
is a $0$-dimensional subscheme of length six.
Setting $\eta=l_a\cap \delta_i$ and
$\eta'=l_b\cap \delta'_i$,
we obtain the final assertion of the lemma
since $\sO_{\mP_a}(1)|_{\delta_i}=L|_{\delta_i}$ and
$\sO_{\mP_b}(1)|_{\delta'_i}=L|_{\delta'_i}$.
\end{proof}

\begin{rem}
It is possible to construct an example such that $p$ has a positive dimensional fiber.
\end{rem}

In \cite{IK}, Kuznetsov and Ingalls obtained
the following result:
\begin{thm}
\label{thm:IK}
We define the following triangulated subcategories $\sA_X$ and $\sA_Y$ in
$\sD^b(X)$ and $\sD^b(\widetilde{Y})$, respectively\,$:$
\begin{eqnarray*}
\sD^b(X)&=&\langle \{ \sO_X(-\delta_i) \}_{i=1}^{10},\sA_X\rangle,\\
\sD^b(\widetilde{Y})&=&\langle \{\sO_{F_i}(-1,-1)\}_{i=1}^{10},
\{\sO_{F_i}(0,-1)\}_{i=1}^{10}, \sA_Y,\sO_{\widetilde{Y}}(1), \sO_{\widetilde{Y}}(2)\rangle.
\end{eqnarray*}
Then there exists an equivalence $\sA_X\simeq \sA_Y$.
\end{thm}

Refining Lemma \ref{lem:IK} as in the following proposition,
we deduce from Theorem \ref{thm:main} that $\Phi_1$ induces an equivalence $\sA_X\to \sA_Y$,
which immediately gives another proof of Theorem \ref{thm:IK}.
\begin{prop}
\label{prop:refine}
\[
\Phi_1(\sO_X(-\delta_i))=\sO_{F_i}(0,-1)[-1],\,
\Phi_1(\sO_X(-\delta'_i))=\sO_{F_i}(-1,0)[-1].
\]
\end{prop}

\begin{proof}
By symmetry, we have only to show the claim for $\delta_i$.
We denote by 
\[ 
\text{$p_1\colon \widetilde{Y}\times X\to \widetilde{Y}$
and $p_2 \colon \widetilde{Y}\times X\to X$},
\]
the natural projections. 

We compute $\Phi_1(\sO_X(-\delta_i))$ explicitly
by using the descriptions of $\Delta_1$
as in Lemma \ref{Delta1fiber}, and
the following exact sequence relating $\sP_1$ and $\Delta_1$,
which is derived in the appendix \ref{app:B} (the subsection \ref{cutting}):
\begin{equation}
\label{eq:PYX}
0\to \sO_{\widetilde{Y}}\boxtimes \sO_{X}(-L)\to
\widetilde{\sQ}|_{\widetilde{Y}}\boxtimes \sO_{X}\to \sP_1\to 
\omega_{\Delta_1/\widetilde{Y}}\otimes \omega^{-1}_{X}\otimes \sO_{X}(-L)
\to 0.
\end{equation}
We split (\ref{eq:PYX}) $\otimes p_2^*\sO_X(-\delta_i)$ as follows:
\begin{eqnarray}
0\to \sO_{\widetilde{Y}}\boxtimes \sO_{X}(-L-\delta_i)\to
\widetilde{\sQ}|_{\widetilde{Y}}\boxtimes \sO_{X}(-\delta_i)
\to \sC\to 0,\label{eqn:LP1}\\
0\to \sC\to \sP_1\otimes \sO_X(-\delta_i)\to 
\omega_{\Delta_1/\widetilde{Y}}\otimes \sO_{X}(-L-\delta'_i)
\to 0,\label{eqn:LP2}
\end{eqnarray}
where
we use $K_X=\delta'_i-\delta_i$ in (\ref{eqn:LP2}).

\vspace{5pt}

\noindent{\bf Step 1.} We will derive 
the following short exact sequence 
by computing ${p_1}_*$ of (\ref{eqn:LP1}).
\begin{align}
\label{eqn:LP1'}
0\to
R^1 {p_1}_* \sC\to 
H^0(X,\sO_X(L+\delta'_i))^*\otimes \sO_{\widetilde{Y}}\to\\
H^0(X,\sO_X(\delta'_i))^*\otimes 
\widetilde{\sQ}|_{\widetilde{Y}}\to 0.\nonumber 
\end{align}

Indeed, by (\ref{eqn:LP1}), we obtain ${p_1}_*\sC=0$
since
$H^{\bullet}(X,\sO_X(-L-\delta_i))\simeq 
H^{\bullet}(X,\sO_X(-\delta_i))=0$ for $\bullet=0,1$.
By the Serre duality,
we have
\[
H^2(X,\sO_X(-L-\delta_i))\simeq H^0(X,\sO_X(L+\delta'_i))^*,\,
H^2(X,\sO_X(-\delta_i))\simeq H^0(X,\sO_X(\delta'_i))^*
\]
since $K_X=\delta'_i-\delta_i$.
Consider the map 
\begin{equation}
\label{eq:Phimap}
H^2(X,\sO_X(-L-\delta_i))\otimes \sO_{\widetilde{Y}}\to
H^2(X,\sO_X(-\delta_i))\otimes \widetilde{\sQ}|_{\widetilde{Y}}
\end{equation}
obtained by taking ${p_1}_*$ of (\ref{eqn:LP1}).
It is easy to see by the Serre duality 
that this map is dual to the map
\begin{equation}
\label{eq:H0}
H^0(X,\sO_X(\delta'_i))\otimes 
\widetilde{\sQ}^*|_{\widetilde{Y}}\to
H^0(X,\sO_X(L+\delta'_i))\otimes \sO_{\widetilde{Y}}
\end{equation}
induced from
the map 
\begin{equation}
\label{eq:QL}
\widetilde{\sQ}^*|_{\widetilde{Y}}\boxtimes 
\sO_X\to
\sO_X(L)\boxtimes \sO_{\widetilde{Y}},
\end{equation}
which is obtained by taking the dual of (\ref{eq:PYX}).
We see that the cokernel of (\ref{eq:H0}) is a locally free sheaf
since the map (\ref{eq:H0}) at the fiber of any point $y\in \widetilde{Y}$
gives three linearly independent members of $|L+\delta'_i|$.
Therefore the map (\ref{eq:Phimap}) is surjective, and then we have $R^2 {p_1}_*\sC=0$ from (\ref{eqn:LP1}).
Now we have obtained (\ref{eqn:LP1'}).

\vspace{5pt}

\noindent{\bf Step 2.} We will derive 
the following exact sequence
by computing ${p_1}_*$ of (\ref{eqn:LP2}):

\begin{align}
\label{eqn:LP2'}
&0\to (p_*q^*\sO_{X}(L+\delta'_i))^*\to R^1 {p_1}_* \sC\to\\
&R^1 {p_1}_* (\sP_1\otimes {p_2}^*\sO_X(-\delta_i))\to 
R^1 p_*
\{
\omega_{\Delta_1/\widetilde{Y}}\otimes  \sO_{X}(-L-\delta'_i)\}
\to 0.\nonumber
\end{align}
Moreover we will obtain
\begin{equation}
\label{eq:Phi1}
\Phi_1(\sO_X(-\delta_i))=R^1 {p_1}_* (\sP_1\otimes {p_2}^*\sO_X(-\delta_i))[-1].
\end{equation}

Indeed, by Lemma \ref{Delta1fiber} (2), we can describe 
$R^{\bullet} p_*
\{
\omega_{\Delta_1/\widetilde{Y}}\otimes  \sO_{X}(-L-\delta'_i)\}$
as follows:
\begin{itemize}
\item $R^2 p_*
\{
\omega_{\Delta_1/\widetilde{Y}}\otimes  \sO_{X}(-L-\delta'_i)\}=0$
since any fiber of $p$ has dimension $\leq 1$. 
\item
\begin{equation}
\label{eq:odim}
\dim \Supp 
R^1 p_*
\{
\omega_{\Delta_1/\widetilde{Y}}\otimes  \sO_{X}(-L-\delta'_i)\}=0
\end{equation}
since
the support is contained in the union of the images of
positive dimensional fibers of $p$. 
\item
$p_*\{
\omega_{\Delta_1/\widetilde{Y}}\otimes  \sO_{X}(-L-\delta'_i)\}
\simeq 
(p_*q^*\sO_{X}(L+\delta'_i))^*$.

Indeed,
this isomorphism holds outside 
the union of the images of
positive dimensional fibers of $p$
by the relative duality (Theorem \ref{cla:duality}).
Then, actually this isomorphism holds all over $\widetilde{Y}$
since the sheaves on the both sides are reflexive 
by the proof of \cite[Cor.~1.7]{H}.
\end{itemize}
Then, by taking ${p_1}_*$ of (\ref{eqn:LP2}),
we have
$R^2 {p_1}_*(\sP_1\otimes {p_2}^*\sO_X(-\delta_i))=0$
(note that
we have already shown $R^2 {p_1}_*\sC=0$).
We also have
${p_1}_*(\sP_1\otimes {p_2}^*\sO_X(-\delta_i))=0$ since
it is at most a torsion sheaf by Lemma \ref{lem:IK} and
$(p_*q^*\sO_{X}(L+\delta'_i))^*$ is torsion free (note also that we have already
shown ${p_1}_*\sC=0$).
Therefore we obtain (\ref{eqn:LP2'}) and (\ref{eq:Phi1}).

\vspace{10pt}

We set 
\begin{equation}
\label{sA}
\sA:=R^1 {p_1}_* (\sP_1\otimes {p_2}^*\sO_X(-\delta_i)).
\end{equation}
Now the problem is reduced to compute $\sA$ explicitly, which will be done in Step 5 below.

\vspace{5pt}

\noindent{\bf Step 3.} 
We compute the duals of (\ref{eqn:LP1'}) and (\ref{eqn:LP2'}).

As for (\ref{eqn:LP1'}), we immediately obtain
\begin{eqnarray}
\label{eqn:LP1''}
0\to
H^0(X,\sO_X(\delta'_i))\otimes 
\widetilde{\sQ}^*|_{\widetilde{Y}}
\to 
H^0(X,\sO_X(L+\delta'_i))\otimes \sO_{\widetilde{\hcoY}}\to\\
(R^1 {p_1}_* \sC)^*\to 0.\nonumber
\end{eqnarray}

As for (\ref{eqn:LP2'}), we have
\begin{eqnarray}
\label{eqn:LP2''}
0\to (R^1 {p_1}_* \sC)^*\to
p_*q^*\sO_{X}(L+\delta'_i)\to
\sE xt^1_{\widetilde{Y}}(\sA,\sO_{\widetilde{Y}})\to 0.
\end{eqnarray}
Indeed, this follows by noting
\begin{itemize}
\item
$p_*q^*\sO_{X}(L+\delta'_i)$ is a reflexive sheaf on $\widetilde{Y}$
by Lemma \ref{Delta1fiber} (2) and the proof of \cite[Cor.~1.7]{H}.

\item $R^1 {p_1}_* \sC$ is a locally free sheaf on $\widetilde{Y}$
by (\ref{eqn:LP1'}), and then
$\sE xt^1 (R^1 {p_1}_* \sC,\sO_{\widetilde{Y}})=0$.
\item
$\sE xt^{\bullet}(R^1 p_*
\{
\omega_{\Delta_1/\widetilde{Y}}\otimes  \sO_{X}(-L-\delta'_i)\},
\sO_{\widetilde{Y}})=0$ for $\bullet <3$ by (\ref{eq:odim}) and \cite[Cor.~3.5.11]{BH}.
\end{itemize}

\vspace{5pt}

\noindent{\bf Step 4.}
We will prove
the composite
\[
H^0(X,\sO_X(L+\delta'_i))\otimes \sO_{\widetilde{\hcoY}}\to\\
(R^1 {p_1}_* \sC)^*\to
p_*q^*\sO_{X}(L+\delta'_i)
\]
induced from (\ref{eqn:LP1''}) and
(\ref{eqn:LP2''}) coincides with
the natural map 
\begin{equation}
\label{eq:natkey}
H^0(\widetilde{Y}, p_*q^*\sO_{X}(L+\delta'_i))
\otimes
\sO_{\widetilde{\hcoY}}
\to
p_*q^*\sO_{X}(L+\delta'_i)
\end{equation}
up to
a linear isomorphism of 
$H^0(\widetilde{Y}, p_*q^*\sO_{X}(L+\delta'_i))$ onto itself.

Indeed, in (\ref{eqn:LP1''}), we have
$H^{\bullet}(\widetilde{Y},\widetilde{\sQ}^*|_{\widetilde{Y}})=0$ for
any $\bullet$ since
$H^{\bullet}(\widetilde{\hcoY},\widetilde{\sQ}^*(-t))=0$ for
$0\leq t\leq 6$ by \cite[Thm.5.1, Prop.~5.9]{DerSym}.
Therefore the map 
\begin{equation}
\label{eq:isom}
H^0(X,\sO_X(L+\delta'_i))\to H^0(\widetilde{Y},(R^1 {p_1}_* \sC)^*)
\end{equation}
induced from (\ref{eqn:LP1''})
is an isomorphism.
Note that
\[
H^0(\widetilde{Y}, p_*q^*\sO_{X}(L+\delta'_i))
\simeq
H^0(\Delta_1, q^*\sO_{X}(L+\delta'_i))
\simeq 
H^0(X, \sO_{X}(L+\delta'_i)),
\]
where the second isomorphism follows from 
$R^{\bullet} q_*\sO_{\Delta_1}=0$ for $\bullet>0$ 
(Lemma \ref{Delta1fiber} (1)). 
Since 
the map 
\begin{equation}
\label{eq:inj}
H^0(\widetilde{Y},(R^1 {p_1}_* \sC)^*)\to
H^0(\widetilde{Y},p_*q^*\sO_{X}(L+\delta'_i))\simeq
H^0(X, \sO_{X}(L+\delta'_i))
\end{equation}
induced from (\ref{eqn:LP2''}) is injective,
the composite of (\ref{eq:isom}) and (\ref{eq:inj}) is an 
isomorphism. This implies the assertion.

\vspace{5pt}

\noindent {\bf Step 5.} Now we compute $\sA$ as in (\ref{sA}) and finish the proof of the proposition.

Note that $\sA$ is a locally free sheaf on $F_i$ by Lemma \ref{lem:IK} and 
Theorem \ref{thm:delPezzo}.
Therefore, by duality, we have
\[
\sE xt^1_{\widetilde{Y}}(\sA,\sO_{\widetilde{Y}})\simeq
\sH om_{F_i}(\sA,\sO_{F_i}(F_i))\simeq \sA^*(-1,-1),
\]
where $\sA^*$ means the dual of $\sA$ as an $\sO_{F_i}$-module.
By (\ref{eqn:LP2''}) and Step 4,
$\sA^*(-1,-1)$ is the cokernel
of (\ref{eq:natkey}).
Let $y \in F_i$ be a point. Then the fiber 
$p\colon \Delta_1\to \widetilde{\hcoY}$ over $y$
is the $0$-dimensional subscheme $\eta_y=\eta+\eta'$ of degree six
described in Lemma \ref{Delta1fiber} (2).
We will show that the natural map
$H^0(X,\sO_X(L+\delta'_i))\to H^0(\eta',\sO_{\eta'}(L+\delta'_i))$ is surjective, and
the natural map
$H^0(X,\sO_X(L+\delta'_i))\to H^0(\eta,\sO_{\eta}(L+\delta'_i))$ has
one dimensional cokernel.

Since $H-L=\delta'_i-\delta_i$,
we have 
\begin{equation}
\label{eq:HL}
H+\delta_i=L+\delta'_i.
\end{equation}

We show the assertion for $\eta'$.
By (\ref{eq:HL}),
we have the exact sequence:
\[
0\to \sO_{X}(L)\to \sO_X(L+\delta'_i)\to \sO_{\delta'_i}(H)\to 0
\]
since $\delta_i\cap \delta'_i=\emptyset$.
This induce a surjection
$H^0(X,\sO_X(L+\delta'_i))\to H^0(\delta'_i,\sO_{\delta'_i}(H))$
since $H^1(X,\sO_X(L))=0$ by the Kodaira vanishing theorem.
We consider the exact sequence
\[
0\to \sO_{\delta'_i}(H-L)\to\sO_{\delta'_i}(H)\to \sO_{\eta'}(H)\to 0,
\] 
where we note that $L|_{\delta'_i}\sim \eta'$.
Since $(H-L)|_{\delta'_i}$ is a torsion divisor,
we have 
$H^{\bullet}(\delta'_i, (H-L)|_{\delta'_i})=0$ for $\bullet=0,1$.
Thus the induced map $H^0(\delta'_i,\sO_{\delta'_i}(H))
\to H^0(\eta',\sO_{\eta'}(H))$ is an isomorphism, and then
the induced map
$H^0(X,\sO_X(L+\delta'_i))\to H^0(\delta'_i,\sO_{\delta'_i}(H))
\to H^0(\eta',\sO_{\eta'}(H))$ is surjective. 

We show the assertion for $\eta$.
By (\ref{eq:HL}),
we have the exact sequence:
\[
0\to \sO_{X}(H)\to \sO_X(L+\delta'_i)\to \sO_{\delta_i}(L)\to 0.
\]
This induce a surjection
$H^0(X,\sO_X(L+\delta'_i))\to H^0(\delta_i,\sO_{\delta_i}(L))$
since $H^1(X,\sO_X(H))=0$ by the Kodaira vanishing theorem.
We consider the exact sequence
\[
0\to \sO_{\delta_i}\to\sO_{\delta_i}(L)\to \sO_{\eta}(L)\to 0,
\] 
where we note that $L|_{\delta_i}\sim \eta$.
Thus the induced map $H^0(\delta_i,\sO_{\delta_i}(L))
\to H^0(\eta,\sO_{\eta}(L))$ has one dimensional cokernel,
and then so does 
the induced map
$H^0(X,\sO_X(L+\delta'_i))\to H^0(\delta_i,\sO_{\delta_i}(L))
\to H^0(\eta,\sO_{\eta}(L))$. 

In particular, $\sA$ is an invertible sheaf on $F_i$.
Moreover, if $y$ moves on a fiber corresponding to 
$\delta'_i$, then the fiber of $\sA$ at $y$ does not change.
Therefore $\sA^*(-1,-1)=\sO_{F_i}(a,0)$ with some $a\in \mZ$.
As we have seen above,
(\ref{eqn:LP2''}) induces an isomorphism
$H^0(\widetilde{Y},(R^1 {p_1}_* \sC)^*)\simeq
H^0(\widetilde{Y}, p_*q^*\sO_{X}(L+\delta'_i))$,
and, by (\ref{eqn:LP1''}),
we have
\[
H^{\bullet}(\widetilde{Y},(R^1 {p_1}_* \sC)^*)=0\, (\bullet=1,2,3).
\]
Moreover, by the Leray spectral sequence for $p$,
$H^1(\widetilde{Y}, p_*q^*\sO_{X}(L+\delta'_i))$
is contained in 
$H^1(\Delta_1, q^*\sO_{X}(L+\delta'_i))$, and,
by $R^{\bullet} q_*\sO_{\Delta_1}=0$ for $\bullet>0$ 
(Lemma \ref{Delta1fiber} (1)), 
the latter is isomorphic to
$H^1(X, \sO_{X}(L+\delta'_i))$,
which is zero by the Kodaira vanishing theorem.
Thus we have $H^1(\widetilde{Y}, p_*q^*\sO_{X}(L+\delta'_i))=0$.
Therefore, by (\ref{eqn:LP2''}),
we obtain
$H^{\bullet}(F_i,\sA^*(-1,-1))=0$ for $\bullet=0,1$.
Thus we have $a=-1$, which in turn shows 
$\sA=\sO_{F_i}(0,-1)$. 

\end{proof}

\appendix

\section{{\bf Locally free resolution of ${\iota_{\Vs}}_*\sP$}}
\label{app:B}

\subsection{Locally free resolutions of the ideal sheaves of $\Delta$}
\label{sub:resol}

The aim of this subsection is 
to construct locally free resolutions of 
the ideal sheaves of $\Delta$ in $\widetilde{\hcoY}\times
\hchow$ and $\Vs$.

\begin{thm}
\label{thm:resolY}
\begin{enumerate}[$(1)$]
\item
The ideal sheaf $\sI$ of $\Delta$ in $\widetilde{\hcoY}\times \hchow$ has 
the following $\mathrm{SL}(V)$-equivariant
locally free resolution\,$:$
\begin{eqnarray}
\label{eqnarray:fin5}
0\to
\widetilde{\sS}_L\boxtimes
\sO_{{\hchow}}(H-L)\to
\widetilde{\sT}^*\boxtimes \sF(H)
\to \\
\sO_{\widetilde{\hcoY}} 
\boxtimes \ft{ S}^2 \sF(H+L) \oplus
\widetilde{\sQ}^*(M) \boxtimes \sO_{\hchow}(H)
\to\nonumber \\
\sI(M+H+L)\to 0,
\nonumber
\end{eqnarray}
where the symbol $g^*$ for $\sF^*$ and ${\ft S}^2 \sF^*$ is omitted,
$M+H+L$ means the twist by $\sO_{\widetilde{\hcoY}}(M)\boxtimes 
\sO_{\hchow}(H+L)$, 
and we follow Convention $\ref{conv:chow}$.
\item
Set $\sI_{\Delta/\Vs}:=\sI/\sI_{\Vs}$, the ideal sheaf of $\Delta$
in $\Vs$.
Then ${\iota_{\Vs}}_*\sI_{\Delta/\Vs}$ has the following $\SL(V)$-equivariant
locally free resolution on $\Vs:$
\begin{eqnarray}
\label{eqnarray:onVs}
0\to
\widetilde{\sS}_L\boxtimes
\sO_{{\hchow}}(H-L)\to
\widetilde{\sT}^*\boxtimes \sF(H)
\to \\
\sO_{\widetilde{\hcoY}} 
\boxtimes T_{\hchow/\mathrm{G}(2,V)}(L) 
\oplus
\widetilde{\sQ}^*(M) 
\boxtimes \sO_{\hchow}(H)
\to\nonumber \\
{\iota_{\Vs}}_*\sI_{\Delta/\Vs}(M+H+L)\to 0,
\nonumber
\end{eqnarray}
where 
$T_{\hchow/\mathrm{G}(2,V)}$ is the relative tangent bundle
for the morphism $\hchow \to \mathrm{G}(2,V)$.
\end{enumerate}
\end{thm}

\begin{rem}
The twist by
$\sO_{\widetilde{\hcoY}}(M)\boxtimes 
\sO_{\hchow}(H+L)$ turns out to be convenient in the 
proof of Theorem \ref{thm:newker} below.
\end{rem} 

The proof of Theorem \ref{thm:resolY} is almost identical with that of
\cite[Thm.~5.1.3]{HoTa4}, so we only give its outline below.

We recall the diagram (\ref{eq:constDelta}).
The starting point is the following locally free resolution of the 
ideal sheaf $\sI_0$
of $\Delta_0$ in $\mathrm{G}(2,V)\times \mathrm{G}(2,V)$ (cf.~\cite[Prop.~5.1.1]{HoTa4}).
\begin{prop}
\label{cla:Delta}
The ideal sheaf $\sI_0$
has the following Koszul resolution\,$:$
\begin{equation}
\label{eq:delta0}
0\to \wedge^4 ({\sG}^*\boxtimes {\sF})\to \wedge^3 ({\sG}^*\boxtimes {\sF})
\to \wedge^2 ({\sG}^*\boxtimes {\sF}) \to {\sG}^*\boxtimes {\sF}\to 
\sI_0\to 0.
\end{equation}
\end{prop}

Let $\sI_{\widetilde{\zpq}}$ be the ideal sheaf of
${\Delta}_{\widetilde{\zpq}}$ on $\widetilde{\Zpq}\times \hchow$.
By pulling back the locally free resolution (\ref{eq:delta0}) to $\widetilde{\Zpq}\times \hchow$,  
we see that $\sI_{\widetilde{\zpq}}$ has 
the following locally free resolution, where we omit the symbols of the pull-backs:
\begin{equation}
\label{eqnarray:I''}
0\to 
\wedge^4 ({\sG}^*\boxtimes \sF)\to 
\wedge^3 ({\sG}^*\boxtimes \sF)\to 
\wedge^2 ({\sG}^*\boxtimes \sF)\to {\sG}^*\boxtimes \sF \to  
\sI_{\widetilde{\zpq}}\to 0.
\end{equation}

We recall that 
we denote the transform of $\Prt_{\sigma}$ on
$\widetilde{\hcoY}$ also by
$\Prt_{\sigma}$.
We set \[
\widetilde{\Zpq}^o:=\widetilde{\Zpq}\setminus \Lpi_{\widetilde{\Zpq}}^{-1}(\Prt_{\sigma}),\,
\widetilde{\hcoY}^o:=\widetilde{\hcoY}\setminus \Prt_{\sigma}.
\]
By Proposition \ref{prop:cobu},
$\widetilde{\Zpq}^o\to \widetilde{\hcoY}^o$
is a conic bundle and a fiber of $\check{\Lpi}_{\widetilde{\Zpq}}$
is a non $\sigma$-conic on $\mathrm{G}(2,V)$.
Now
we calculate the pushforward of (\ref{eqnarray:I''})
by
$\check{\Lpi}_{\widetilde{\Zpq}}:=\Lpi_{\widetilde{\Zpq}}\times \id_{\hchow}\colon
\widetilde{\Zpq}\times \hchow \to \widetilde{\hcoY}\times \hchow$ over 
its flat locus $\widetilde{\hcoY}^o$.

{\it Until the end of this subsection,
 we consider only on $\widetilde{\hcoY}^o\times {\hchow}$
to calculate the higher direct images 
for $\check{\Lpi}_{\widetilde{\Zpq}}$.
To simplify the notation, we abbreviate the symbols for the restriction.}

Then we obtain the following exact sequence on
the locus $\widetilde{\hcoY}^o\times {\hchow}$:
\begin{eqnarray}
\label{eqnarray:I'}
0\to R^1 \check{\Lpi}_{\widetilde{\Zpq} *}
\wedge^4 ({\sG}^*\boxtimes \sF)
\to  
R^1 \check{\Lpi}_{\widetilde{\Zpq} *}
\wedge^3 ({\sG}^*\boxtimes \sF)\to \\
R^1 \check{\Lpi}_{\widetilde{\Zpq} *}
\wedge^2 ({\sG}^*\boxtimes \sF)
\to \sI^o\to 0,\nonumber
\end{eqnarray}
where $\sI^o$ is the ideal sheaf of $\Delta^o$ and is equal to 
$\check{\Lpi}_{\widetilde{\Zpq} *} \sI_{\widetilde{\zpq}}$.

By Proposition \ref{cla:duality} for 
the morphism
$\widetilde{\Zpq}^o\times {\hchow}\to \widetilde{\hcoY}^o\times {\hchow}$,
we have
\[
R^1 \check{\Lpi}_{\widetilde{\Zpq} *}
\wedge^i ({\sG}^*\boxtimes \sF)
\simeq 
\big(\check{\Lpi}_{\widetilde{\Zpq} *}\{
\wedge^i ({\sG}\boxtimes \sF^*)\otimes
\omega_{\widetilde{\Zpq}\times {\hchow}/
\widetilde{\hcoY}\times {\hchow}}\}\big)^*.
\]
Note that 
\[
\omega_{\widetilde{\Zpq}\times {\hchow}/
\widetilde{\hcoY}\times {\hchow}}=
\mathrm{pr}_2^*\omega_{\widetilde{\Zpq}/
\widetilde{\hcoY}}
=\omega_{\widetilde{\Zpq}/\widetilde{\hcoY}}\boxtimes
\sO_{{\hchow}}\simeq
\sO_{\widetilde{\Zpq}}(M-L)\boxtimes \sO_{\hchow},
\]
where the second isomorphism follows from 
the formula of 
the relative canonical divisor $K_{\widetilde{\Zpq}/\widetilde{\hcoY}}$:
\[
K_{\widetilde{\Zpq}/\widetilde{\hcoY}}=
M-L
\]
(cf.~\cite[Prop.~4.5.1 (3)]{HoTa4}).
Thus we have
\begin{align*}
R^1 \check{\Lpi}_{\widetilde{\Zpq} *}
\wedge^i ({\sG}^*\boxtimes \sF)
&\simeq \label{duality} \\
\big(\check{\Lpi}_{\widetilde{\Zpq} *}\{
\wedge^i ({\sG}\boxtimes \sF^*)&\otimes  
(\sO_{\widetilde{\Zpq}}(-L) \boxtimes \sO_{{\hchow}})\}\otimes
(\sO_{\widetilde{\hcoY}}(M)\boxtimes \sO_{{\hchow}})\big)^*. \nonumber
\end{align*}
We write down this more explicitly.
Note that, by \cite[Exercise 6.11]{FH},
it holds that
\[
\wedge^i ({\sG}\boxtimes \sF^*)
\simeq \bigoplus_{\lambda} \ft{\Sigma}^{\lambda}
{\sG}\boxtimes \ft{\Sigma}^{\lambda'} \sF^*,
\]
where $\lambda$ are partitions of $i$
with at most $2$ rows and column,
and $\lambda'$ is the partitions dual to $\lambda$. 

Now the exact sequence (\ref{eqnarray:I'})$\otimes 
\sO_{\widetilde{\hcoY}}(M)\boxtimes 
\sO_{\hchow}(H+L)$ 
on
the locus $\widetilde{\hcoY}^o\times {\hchow}$
is presented as follows:
\begin{eqnarray*}
\label{eqnarray:fin}
0\to(\Lpi_{\widetilde{\Zpq} *}\sO_{\widetilde{\Zpq}}(L))^*\boxtimes
\sO_{{\hchow}}(H-L)
\to(\Lpi_{\widetilde{\Zpq} *}\sG)^*
\boxtimes
\sF(H)
\to\\
\sO_{\widetilde{\hcoY}}\boxtimes \ft{ S}^2 \sF(H+L) 
  \oplus
  (\Lpi_{\widetilde{\Zpq} *}(\ft{ S}^2 {\sG}(-1)))^*\boxtimes
\sO_{{\hchow}}(H)\nonumber \\
\to
\sI^o(M+H+L)\to 0.
\nonumber
\end{eqnarray*}

We would like to compute the following sheaves explicitly:
\begin{equation}
\label{eq:sheaves}
\Lpi_{\widetilde{\Zpq} *}\sO_{\widetilde{\Zpq}}(L),\,
\Lpi_{\widetilde{\Zpq} *}{\sG},\,
\Lpi_{\widetilde{\Zpq} *}(\ft{ S}^2 \sG(-1)).
\end{equation}
For this, we estimate these sheaves using
$\widetilde{\Zpq}^t$ constructed in the subsection \ref{Z2Y2}.
Let 
$\Lpi_{\widetilde{\Zpq}^t}\colon \widetilde{\Zpq}^t\to \widetilde{\hcoY}$ and
$\widetilde{\rho}^t\colon \widetilde{\Zpq}^t\to \mathrm{G}(2,V)$
be the natural morphisms.
We set $\widetilde{G}:=\mathrm{G}(2,V)\times \widetilde{\hcoY}$.
$\widetilde{\Zpq}^t$ has a better description in $\widetilde{G}$
than $\widetilde{\Zpq}$.
Namely,
$\widetilde{\Zpq}^t$
is the complete intersection 
in $\widetilde{G}$
with respect to a section of
$\widetilde{\sQ}\boxtimes \sO_{\mathrm{G}(2,V)}(1)$
by Proposition \ref{prop:ci}, and then
the sheaf $\sO_{\widetilde{\Zpq}^t}$
has the following Koszul resolution as a $\sO_{\widetilde{G}}$-module:
\begin{equation*}
\mylabel{eqnarray:KoszulZ2}
0\to
\sE_3\to \sE_2\to \sE_1\to
\sO_{\widetilde{G}}\to
\sO_{\widetilde{\Zpq}^t}
\to 0,
\end{equation*}
where we set
\[
\sE_i:=
\wedge^i \widetilde{\sQ}^* 
\boxtimes
\sO_{\mathrm{G}(2,V)}(-i)  
\
\text{for}\ i=0,1,2,3.
\]
Using this Koszul resolution, we show
(cf.~\cite[Lem.~5.6.2]{HoTa4})
\begin{lem}
\label{cla:est}
\begin{enumerate}[$(i)$]
\item
$\Lpi_{{\widetilde{\Zpq}^t} *}\sO_{\widetilde{\Zpq}^t}(L)\simeq 
\widetilde{\sS}^*_L$,
\item
$\Lpi_{\widetilde{\Zpq}^t *}{\sG}\simeq
V\otimes \sO_{\widetilde{\hcoY}}$, and
\item
$\Lpi_{\widetilde{\Zpq}^t *}(\ft{ S}^2 {\sG}(-1))\simeq 
\widetilde{\sQ}(-M-F_{\rho}),$ where we omit the symbol of the pull-back
$(\widetilde{\rho}^t)^*$ in the l.h.s.
\end{enumerate}
\end{lem}
These are good estimates of the sheaves in (\ref{eq:sheaves}). 
Indeed,
for any sheaf $\sB$ on $\mP(\Lrho_{\widetilde{\hcoY}}^*\sS)$, we have a natural map
$\Lpi_{{\widetilde{\Zpq}^t} *}(\sB|_{\widetilde{\Zpq}^t})\to \Lpi_{\widetilde{\Zpq} *}(\sB|_{\widetilde{\Zpq}})$ on $\widetilde{\hcoY}$,
which is isomorphic outside $F_{\rho}$.
Moreover, if $\Lpi_{\widetilde{\Zpq}^t *}(\sB|_{\widetilde{\Zpq}^t})$ is locally free,
then the map is injective. Note that
this is the case for each sheaf as in (\ref{eq:sheaves})
by Lemma \ref{cla:est} (i)--(iii).

Finally we obtain (cf.~\cite[Prop.~5.6.4]{HoTa4})
\begin{prop}
\mylabel{prop:est}
\[
\Lpi_{\widetilde{\Zpq} *}\sO_{\widetilde{\Zpq}}(L)
\simeq \widetilde{\sS}^*_L,\
\Lpi_{\widetilde{\Zpq} *}{\sG}\simeq
\widetilde{\sT},\
\Lpi_{\widetilde{\Zpq} *}(\ft{ S}^2 \sG(-1))
\simeq
\widetilde{\sQ}(-M).
\]
\end{prop}
Now we have obtained the following locally free resolution of
$\sI^o(M+H+L)$:
\begin{eqnarray}
\label{eqnarray:fin2}
0\to
\widetilde{\sS}_L \boxtimes
\sO_{{\hchow}}(H-L)\to
\widetilde{\sT}^*\boxtimes \sF(H)
\to \\
\sO_{\widetilde{\hcoY}} 
\boxtimes \ft{ S}^2 \sF(H+L) \oplus
\widetilde{\sQ}^*(M) \boxtimes \sO_{\hchow}(H)
\to\nonumber \\
\sI^o(M+H+L)\to 0.
\nonumber
\end{eqnarray}

Let $\iota_{\widetilde{\hcoY}^o}$ be the open immersion 
${\widetilde{\hcoY}^o}\times {\hchow}
\hookrightarrow \widetilde{\hcoY}\times {\hchow}$.
As in \cite[\S 5.8]{HoTa4}, we see that
$\iota_{\widetilde{\hcoY}^o *}\sI^o=\sI$ and 
the locally free resolution (\ref{eqnarray:fin2}) extends
to (\ref{eqnarray:fin5}).

Now we complete an outline of our proof of Theorem \ref{thm:resolY} (1).

\vspace{5pt}

Next we consider Theorem \ref{thm:resolY} (2).
The ideal sheaf $\sI_{\Vs}$
of $\Vs$ on $\widetilde{\hcoY}\times{\hchow}$ is isomorphic to $\sO_{\widetilde{\hcoY}}(-M)\boxtimes\sO_{{\hchow}}(-H)$.
The injection $\sO_{\widetilde{\hcoY}}(-M)\boxtimes\sO_{{\hchow}}(-H)\to\sO_{\widetilde{\hcoY}\times{\hchow}}$
is $\mathrm{SL}(V)$-equivariant since
$\Vs$ has a natural $\mathrm{SL}(V)$-action. 
We note that
\begin{align*}
\Hom (\sO_{\widetilde{\hcoY}}(-M)\boxtimes\sO_{{\hchow}}(-H),\sO_{\widetilde{\hcoY}\times{\hchow}})&\simeq\\
\Hom(\sO_{\widetilde{\hcoY}}(-M),\sO_{\widetilde{\hcoY}})
\otimes \Hom(\sO_{{\hchow}}(-H),\sO_{{\hchow}})&\simeq
\Hom (\ft{S}^{2}V, \ft{S}^{2}V),
\end{align*}
and
$\Hom (\ft{S}^{2}V, \ft{S}^{2}V)\simeq \ft{S}^{2}V\otimes\ft{S}^{2}V^{*}$
contains a unique one-dimensional representation, which is generated
by the identity element.
Thus the above injection is induced
from the identity element 
of $\Hom (\ft{S}^{2}V, \ft{S}^{2}V)$
up to constant.

We have an $\mathrm{SL}(V)$-equivariant map \[
\sO_{\widetilde{\hcoY}}(-M)\boxtimes\sO_{{\hchow}}(-H)\to\sO_{\widetilde{\hcoY}}(-M)\boxtimes\ft{S}^{2}\sF,\]
 which is induced from the inclusion $\sO_{\mP(\ft{S}^{2}\sF)}(-1)\to \ft{S}^{2}\sF$, where we omit the symbol $g^*$ for $\ft{S}^{2}\sF$.
Therefore we have an $\mathrm{SL}(V)$-equivariant map \[
\sO_{\widetilde{\hcoY}}(-M)\boxtimes\sO_{{\hchow}}(-H)\to\sO_{\widetilde{\hcoY}}(-M)\boxtimes\ft{S}^{2}\sF\oplus\widetilde{\sQ}^{*}\boxtimes\sO_{\hchow}(-L)\to\sI\hookrightarrow\sO_{\widetilde{\hcoY}\times{\hchow}}.\]
It is easy to verify this is nonzero.
Therefore, by the uniqueness of such a map, its image coincides with $\sI_{\Vs}$.
Then it is easy to 
obtain a locally free sheaf of ${\iota_{\Vs}}_*\sI_{\Delta/\Vs}$
from (\ref{eqnarray:fin5}) by replacing
$\ft{S}^2\sF(H+L)$ with
$\ft{S}^2\sF(H+L)/\sO_{\hchow}(L)$. 
Now we consider the relative Euler sequence
associated to the projective bundle $\hchow=\mP({\ft S}^2 \sF)$: 
\begin{equation*}
\label{eq:relEu}
0\to \sO_{\hchow}(-H) \to \ft{S}^2\sF  
\to T_{\hchow/\mathrm{G}(2,V)}(-H) \to 0.
\end{equation*}
Then
\begin{equation}
\label{eq:TL}
\ft{S}^2\sF(H+L)/\sO_{\hchow}(L)\simeq 
T_{\hchow/\mathrm{G}(2,V)}(L),
\end{equation}
hence we obtain (\ref{eqnarray:onVs}). 
$\hfill\square$
\subsection{Locally free resolution of ${\iota_{\Vs}}_*\sP$}


To show Theorem \ref{thm:newker}, 
we start from some preliminary constructions.
We set
\begin{equation}
\label{eq:sK}
\sK:=\mathrm{Coker}\, \big(\sJ(M+H+L)\hookrightarrow \sI_{\Delta/\Vs}(M+H+L)\big).
\end{equation}
Then we have the following commutative diagram
with exact rows and column:

\vspace{5pt}

\begin{equation}
\label{eq:K}
\xymatrix{
& & & 0\ar[d] & 0\ar[d] & \\
&  &  & \sC_1(M+H)\ar[r]\ar[d] & {\iota_{\Vs}}_*\sJ(M+H+L)\ar[r]\ar[d] & 0\\
0\ar[r] & \sC_2\ar[r]& 
\sC_3\ar[r]\ar[d]
 & \sC_4 \oplus \sC_1(M+H)\ar[r]\ar[d] & 
{\iota_{\Vs}}_*\sI_{\Delta/\Vs}(M+H+L)\ar[r]\ar[d] & 0\\
& & 
\sC_3\ar[r]
 & \sC_4 \ar[r]\ar[d] &
{\iota_{\Vs}}_*\mathcal{K}\ar[r]\ar[d] & 0\\
& & & 0 & 0,}
\end{equation}
where we set 
\begin{eqnarray*}
\sC_1:=\widetilde{\sQ}^*\boxtimes 
\sO_{\hchow},\,
\sC_2:=\widetilde{\sS}_L\boxtimes
\sO_{\hchow}(H-L),\\
\sC_3:=\widetilde{\sT}^*\boxtimes \sF(H),\,
\sC_4:=\sO_{\widetilde{\hcoY}} 
\boxtimes 
T_{\hchow/\mathrm{G}(2,V)}(L)
\end{eqnarray*}
for simplicity of notation,
and the first row comes from the definition of $\sJ$,
the second row is exactly (\ref{eqnarray:onVs}),
and the third row are derived from
a simple diagram chasing.

Now we will extend the third row of this diagram to a certain complex.
By the second row, we have the map $\sC_2\to \sC_3$.
Moreover, by \cite[Rem.~3.4 (3), Rem.~5.12 (2)]{DerSym}, 
we obtain a nonzero unique $\mathrm{SL}(V)$-equivariant map
$\sC_1\to \sC_3$ up to constant.
Therefore we obtain a map
$\sC_1\oplus \sC_2\to \sC_3$.

\begin{lem}
\label{lem:sK}
The third row of the diagram $(\ref{eq:K})$ and
the map $\sC_1\oplus \sC_2\to \sC_3$ obtained above
induce a complex\,$;$
\begin{eqnarray}
\label{eqn:complex}
0\to
\sC_1\oplus \sC_2
\to
\sC_3 \to \sC_4 \to 
{\iota_{\Vs}}_*\sK \to 0,
\end{eqnarray}
which is exact except that
the kernel of the map $\sC_3\to \sC_4$
does not coincide with
$\sC_1\oplus \sC_2$.
\end{lem}

\begin{rem}
It is useful to split (\ref{eqn:complex}) into 
the following three short exact sequences:
\begin{equation}
\label{eq:rem}
\begin{cases}
0\to \sC_1\oplus \sC_2\to \sK_1\to \sK_2\to 0,\\
0\to \sK_1 \to \sC_3\to\sK_3\to 0,\\
0\to \sK_3 \to 
\sC_4\to
{\iota_{\Vs}}_*\sK\to 0,
\end{cases}
\end{equation}
where we define 
$\sK_1$ to be the kernel of the map $\sC_3\to \sC_4$,
$\sK_2$ to be the cokernel of the inclusion $\sC_1\oplus \sC_2\hookrightarrow \sK_1$,
and $\sK_3$ to be
the kernel of the map
$\sC_4\to {\iota_{\Vs}}_*\sK$. 
\end{rem}

\begin{proof}
By the diagram (\ref{eq:K}), we have only to show
the following claims (a) and (b):
\begin{enumerate}[(a)]
\item 
\begin{eqnarray}
\label{eqnarray:fin6}
\sC_1
\to
\sC_3\to \sC_4
\end{eqnarray}
is a complex (note that
$\sC_2
\to
\sC_3\to \sC_4$ is a complex by Theorem \ref{thm:resolY} (2)).
\item
\begin{equation}
\label{eq:map2}
\sC_1\oplus \sC_2
\to
\sC_3
\end{equation}
is injective.
\end{enumerate}
{\bf Proof of the claim (a).}

It suffices to show the composite (\ref{eqnarray:fin6}) is a $0$-map
at the generic point of $\widetilde{\hcoY}\times \hchow$
since the target 
$\sC_4$ of (\ref{eqnarray:fin6})
is locally free, hence is torsion free.
Therefore we only consider points $(y,x)$ of 
$\widetilde{\hcoY}\times \hchow$
such that 
$\widetilde{\hcoY}\to \widehat{\hcoY}$ is isomorphic at $y$, namely, $y\not \in F_{\rho}$.
Then the fiber of the sheaf $\widetilde{\sT}^*$ at $y$ is isomorphic to $V^*$.
Note that a point $x\in \hchow$ corresponds to 
a pair $(V_2, U_1)$ of 
$[V_2]\in G(2,V)$ and one-dimensional subspace
$U_1\simeq \mC\subset {\ft S}^2 V_2$.

\vspace{5pt}

\noindent {\bf Step 1.}
We calculate the map $\sC_3\to \sC_4$
at $(y,x)$.

For this, we treat
the twisted map $\sC_3(-H+L)\to \sC_4(-H+L)$ instead.
Note that the fiber of $\sF^*$ is $V_2^*$, and
the fiber of 
$T_{\hchow/\mathrm{G}(2,V)}(-H+2L)$
is 
${\ft S}^2 V_2^*/(U_1\otimes (\wedge^2 V_2^*)^{\otimes 2})$
since
$T_{\hchow/\mathrm{G}(2,V)}(-H+2L)\simeq \ft{ S}^2 \sF^*/\sO_{\hchow}(-H+2L)$
by (\ref{eq:TL})．
Now we take a basis $e_1$, $e_2$ of $V_2$
and extend it to a basis $e_1,\dots, e_4$ of $V$.
Note that
$\Hom(\widetilde{\sT}^*, \sO_{\widetilde{\hcoY}})\simeq V$
and $\Hom(\sF^*, \ft{ S}^2 \sF^*/\sO_{\hchow}(-H+2L))\simeq V^*$ 
by \cite[Rem.~3.4 (3), Rem.~5.12 (2)]{DerSym}.
Then
the map $\sC_3\to \sC_4$ is the unique nonzero $\SL(V)$-equivariant map
corresponding to the identity
of $\Hom (V, V)\simeq V^*\otimes V$ up to constant.
Note that, at each fiber,
the natural map $\Hom(\widetilde{\sT}^*, \sO_{\widetilde{\hcoY}})\otimes \widetilde{\sT}^*\to \sO_{\widetilde{\hcoY}}$ is 
the canonical projection $V\otimes V^*\to \mC$,
and, by $\Hom(\sF^*, \ft{ S}^2 \sF^*/\sO_{\hchow}(-H+2L))
\simeq \Hom(\sF^*, \ft{ S}^2 \sF^*)$, the map
\[
\Hom(\sF^*, \ft{ S}^2 \sF^*/\sO_{\hchow}(-H+2L))\otimes \sF^*\to \ft{ S}^2 \sF^*/\sO_{\hchow}(-H+2L)
\]
is the composite
\begin{equation}
\label{eq:fiber1}
V^*\otimes V_2^*\to V_2^*\otimes V_2^*\to {\ft S}^2 V_2^*\to
{\ft S}^2 V_2^*/(U_1\otimes (\wedge^2 V_2^*)^{\otimes 2}),
\end{equation}
where the first map in (\ref{eq:fiber1}) is induced from the natural surjection
$V^*\to V_2^*$, and 
the second map is the canonical projection.
Therefore, at each fiber, 
the map $\sC_3(-H+L)\to \sC_4(-H+L)$ is defined by 
\begin{align}
\label{eqnarray:V2}
& V^*\otimes V_2^*\ni e_i^*\otimes e_j^*\mapsto\\
& (\sum_{k=1}^{4} e_k\otimes e_k^*)\otimes e_i^*\otimes e_j^*\in
(V\otimes V^*)\otimes (V^*\otimes V_2^*)
\mapsto \nonumber \\
& \begin{cases}
0 : i=3,4,\, j=1,2\\
1\otimes e_i^*e_j^* : i=1,2,\, j=1,2
\end{cases}
\in \mC\otimes {\ft S}^2 V^*_2. 
\nonumber
\end{align}
In other words,
$V^*\otimes V_2^*\to \mC\otimes {\ft S}^2 V_2^*\simeq {\ft S}^2 V_2^*$ coincides with the natural projection.
In particular, it is surjective, and hence
the support of the cokernel of
the map $\sC_3\to \sC_4$ is contained in $F_{\rho}\times \hchow$.
We postpone to determine the support until Lemma \ref{lem:sK2}.

\vspace{5pt}

\noindent {\bf Step 2.}
We determine the kernel of 
the map $\sC_3(-H+L)\to \sC_4(-H+L)$
at $(y,x)$.
We denote a generator of $U_1$
by $a(e_1)^2+b(e_1e_2)+c(e_2)^2$.
It is easy to see that this corresponds to the element 
$c(e_1^*)^2-b(e_1^*e_2^*)+a(e_2^*)^2$
of $U_1\otimes (\wedge^2 V_2^*)^{\otimes 2}\subset {\ft S}^2 V_2^*$
by the natural pairing
${\ft S}^2 V_2\times {\ft S}^2 V_2\subset (V_2\otimes V_2)\times (V_2\otimes V_2)
\to \wedge^2 V_2\otimes \wedge^2 V_2$．
Thus it is the image of 
$c(e_1^*\otimes e_1^*)-b(e_1^*\otimes e_2^*)+a(e_2^*\otimes e_2^*)
\in V^*\otimes V_2^*$ by the map (\ref{eqnarray:V2}).
Consequently,
the kernel of the map $\sC_3(-H+L)\to \sC_4(-H+L)$ at $(y,x)$ 
is a six-dimensional vector space
with a basis
\[
e_3^*\otimes e_1^*,\,
e_3^*\otimes e_2^*,\,
e_4^*\otimes e_1^*,\,
e_4^*\otimes e_2^*,
\]
\[
e_1^*\otimes e_2^*-e_2^*\otimes e_1^*,\,
c(e_1^*\otimes e_1^*)-b(e_1^*\otimes e_2^*)+a(e_2^*\otimes e_2^*),
\]
which
we denote this by $A$.

\vspace{5pt}

\noindent {\bf Step 3.}
Now we will complete
the proof of the claim (a).

We also consider the twisted map
\begin{equation}
\label{eq:twist}
\sC_1(-H+L)\to \sC_3(-H+L)\to \sC_4(-H+L)
\end{equation}
instead.
Note that, at the point $(y,x)$, we have
\[
\text{$\widetilde{\sQ}^*\subset \wedge^2 V^*$ and
$\sO_{\hchow}(-H+L)\simeq 
U_1\otimes \wedge^2 V_2^*$.}
\]
Therefore, to show (\ref{eq:twist})
is a $0$-map,
 we have only to see
the image of the map
\begin{equation}
\label{eq:U1}
\wedge^2 V^*\otimes
(U_1\otimes \wedge^2 V_2^*)\to
V^*\otimes V_2^*
\end{equation}
is contained in $A$.
By writing down the map (\ref{eq:U1}) explicitly like (\ref{eqnarray:V2}),
we see that the image of an element
\[
(e_i^*\otimes e_j^*-e_j^*\otimes e_i^*)\otimes
(a(e_1)^2+b(e_1e_2)+c(e_2)^2)\otimes (e_1^*\wedge e_2^*)\in
\wedge^2 V^*\otimes
(U_1\otimes \wedge^2 V_2^*)\ (i<j)
\]
by the map (\ref{eq:U1})
is
\begin{equation}
\label{eq:im}
\delta_{1i} e^*_j\otimes (ae_2^*-\frac{b}{2}e_1^*)+
\delta_{2i} e^*_j \otimes (\frac{b}{2}e_2^*-ce_1^*)
-\delta_{2j} e^*_i \otimes (\frac{b}{2}e_2^*-ce_1^*) \in V^*\otimes V_2^*
\end{equation}
($\delta_{lm}$ is Kronecker's delta),
thus, is contained in $A$ as desired.
Now we have proved that
(\ref{eqnarray:fin6}) is a complex.

\vspace{5pt}

\noindent {\bf Proof of the claim (b).}

By twisting $\sO_{\hchow}(-H)$,
we show
\begin{equation}
\label{eq:map2'}
\widetilde{\sQ}^* \boxtimes \sO_{\hchow}(-H)
\oplus
\widetilde{\sS}_L\boxtimes
\sO_{{\hchow}}(-L)
\to
\widetilde{\sT}^*\boxtimes \sF
\end{equation}
is injective.
 It suffices to show this at a general point since the sourse of this map
is locally free, hence is torsion free.

Let $W\subset \wedge^2 V^*$ be the fiber of $\widetilde{\sQ}^*$
at a point of $\widetilde{\hcoY}$.
Then, since the fiber of $\widetilde{\sS}_L^*$ is $(\wedge^2 V^* /W)^*$,
the fiber of $\widetilde{\sS}_L$ is 
the orthogonal space to $W$ with respect to the natural pairing
\[
\wedge^2 V^*\times \wedge^2 V^*\to \wedge^4 V^*.
\]
Thus we denote by 
$W^{\perp}$ the fiber of $\widetilde{\sS}_L$.
Now we take a basis $e_1,\dots, e_4$ of $V$ and consider 
\[
W=\langle e^*_1\wedge e_2^*,
e^*_1\wedge e_4^*-e^*_2\wedge e_3^*, e^*_3\wedge e_4^*\rangle.
\]
Then we have
\[
W^{\perp}=\langle e^*_2\wedge e_4^*,
e^*_2\wedge e_3^*+e^*_1\wedge e_4^*, e^*_1\wedge e_3^*\rangle.
\]
This $W$ corresponds to the point $y$
of $\widetilde{\hcoY}$ associated to the pair of 
the rank $4$ quadric $x_1x_4-x_2x_3=0$ and a family of lines on it,
where $x_1,\dots, x_4$ is the coordinate of $V$ associated to the basis
$e_1,\dots, e_4$. This $y$ is a general point of $\widetilde{\hcoY}$.
Take a point $x\in \hchow$ associated to a pair $(V_2, U_1)$ of 
$[V_2]\in G(2,V)$ and one-dimensional subspace
$U_1\subset {\ft S}^2 V_2$.
By generality, we assume that $x\not \in E_f$.
Therefore 
we can choose a basis $p=\sum_{i=1}^{4} p_i e_i$, $q=\sum_{i=1}^{4} q_i e_i$ 
of $V_2$ such that $U_1=\mC pq$.
Similarly to the computations in the proof of the claim (a), 
we can describe the images of 
$\widetilde{\sQ}^* \boxtimes \sO_{\hchow}(-H)
\to
\widetilde{\sT}^*\boxtimes \sF$ and
$\widetilde{\sS}_L\boxtimes
\sO_{{\hchow}}(-L)
\to
\widetilde{\sT}^*\boxtimes \sF$ at $(y,x)$
as follows:
\begin{itemize}
\item
The image of 
$W\otimes U_1 \to
V^*\otimes V_2$ is the subspace with a basis
\[
e_2^*\otimes (p_1 q+q_1 p)-
e_1^*\otimes (p_2 q+q_2 p),
\]
\[
e_4^*\otimes (p_1 q+q_1 p)-e_1^*\otimes (p_4 q+q_4 p)
-e_3^*\otimes (p_2 q+q_2 p)+e_2^*\otimes (p_3 q+q_3 p),
\]
\[
e_4^*\otimes (p_3 q+q_3 p)-
e_3^*\otimes (p_4 q+q_4 p).
\]
\item
The image of $W^{\perp} \otimes \wedge^2 V^* \to V^*\otimes V_2$
is the subspace with a basis
\[
e_4^*\otimes (p_2 q-q_2 p)-
e_2^*\otimes (p_4 q-q_4 p),
\]
\[
e_3^*\otimes (p_2 q-q_2 p)-e_2^*\otimes (p_3 q-q_3 p)
+e_4^*\otimes (p_1 q-q_1 p)-e_1^*\otimes (p_4 q-q_4 p),
\]
\[
e_3^*\otimes (p_1 q-q_1 p)-
e_1^*\otimes (p_3 q-q_3 p).
\]
\end{itemize}
Therefore, after elementary calculations, we conclude
that, if 
$p_3\not =q_4$ and $p_1 q_4+p_4 q_1\not =p_2 q_3+p_3 q_4$, then
the image of $W\otimes U_1\oplus W^{\perp} \otimes \wedge^2 V^* \to
V^*\otimes V_2$ is a $6$-dimensional vector space.
Therefore (\ref{eq:map2'}) is injective.
\end{proof}

We need more detailed 
descriptions of $\sK$ and $\sK_2$.

\begin{lem}
\label{lem:sK2}
\begin{enumerate}[$(1)$]
\item $\sK$ is an invertible sheaf on the variety $D$,
where we recall that $D$ is an irreducible component of 
$\Delta'$ $($the subsection $\ref{sub:Delta'})$.
\item The cokernel $\sK_2$ of the inclusion map 
$\sC_1\oplus \sC_2\hookrightarrow \sK_1$ 
is a coherent sheaf on $\sV$ of generically rank $2$.
\end{enumerate}
\end{lem}

\begin{proof}~

\noindent (1)
As we have already observed in the proof of 
Lemma \ref{lem:sK} (Step 1 in the proof of the claim (a)),
the support of $\sK$ is contained in
$F_{\rho}\times \hchow$.
We take a point $x\in \hchow$ 
associated to a pair $(V_2, U_1)$ of 
$[V_2]\in G(2,V)$ and one-dimensional subspace
$U_1\subset {\ft S}^2 V_2$.
We also take a point $y$ of 
$F_{\rho}$ lying over a point $[V_1]\in \Prt_{\rho}\simeq \mP(V)$.
Then the image of $\widetilde{\sT}^*\to V^*\otimes \sO_{\widetilde{\hcoY}}$ at $y$
is $(V/V_1)^*$．
We can proceed in the sequel by following the argument of the proof of
the claim (a) with replacing $V^*$ with $(V/V_1)^*$.
For example, we consider
the map $(V/V_1)^*\otimes V_2^*\to {\ft S}^2 V_2^*$ instead of 
(\ref{eqnarray:V2}).
We also follow the notation there.
Then we see that 
the map $\sC_3\to \sC_4$ is surjective in the case where $V_1\not \subset V_2$.
Suppose that  
$V_1\subset V_2$.
Then, by letting $e_1$ be a basis of $V_1$,
the image of the map $\sC_3\to \sC_4$ is generated by the images of 
$e_1^* e_2^*$, $(e_2^*)^2$ in
${\ft S}^2 V_2^*/(U_1\otimes (\wedge^2 V_2^*)^{\otimes 2})$.
It implies that 
the map $\sC_3\to \sC_4$ is surjective when $c\not =0$, and 
the cokernel of the map $\sC_3\to \sC_4$ is one-dimensional when 
$c=0$.
The condition $c=0$ means that
the point of $\hchow$
corresponds to 
a $0$-dimensional subscheme
whose support contains the point $[V_1]$, namely, $(y,x)\in D$.
Therefore we have shown the claim (1).

\vspace{5pt}

\noindent (2)
First we note that 
$\sK_2$ is a torsion sheaf on $\widetilde{\hcoY}\times \hchow$.
Indeed, $\sK_1$ is a sheaf on $\widetilde{\hcoY}\times \hchow$ generically of rank $6$ since $\sC_3\to \sC_4$ is generically surjective by the claim (a)
in the proof of Lemma \ref{lem:sK}, and 
$\rank \sC_3=8$ and $\rank \sC_4=2$.
Therefore $\sK_2$ is a torsion sheaf since $\rank \sC_1=\rank \sC_2=3$.

Now taking the duals of the short exact sequences (\ref{eq:rem})
in the remark after
Lemma \ref{lem:sK}, we obtain the following exact sequences:
\begin{align*}
&\begin{cases}
0\to \sC_4^*\to \sK_3^*\to \sE xt^{1}({\iota_{\Vs}}_*\sK,\sO_{\widetilde{\hcoY}\times \hchow})\to 0,\\ 
\sE xt^i(\sK_3,\sO_{\widetilde{\hcoY}\times \hchow})\simeq \sE xt^{i+1}({\iota_{\Vs}}_*\sK,\sO_{\widetilde{\hcoY}\times \hchow}) \ (i\geq 1),
\end{cases}\\
&\begin{cases}
0\to \sK^*_3\to \sC_3^*\to \sK^*_1\to \sE xt^1(\sK_3,\sO_{\widetilde{\hcoY}\times \hchow})\to 0,\\
\sE xt^i(\sK_1,\sO_{\widetilde{\hcoY}\times \hchow})\simeq \sE xt^{i+1}(\sK_3,\sO_{\widetilde{\hcoY}\times \hchow})\ (i\geq 1),
\end{cases}\\
&\begin{cases}
0\to \sK^*_1\to \sC_1^*\oplus \sC_2^*\to 
\sE xt^1(\sK_2,\sO_{\widetilde{\hcoY}\times \hchow})\to 
\sE xt^1(\sK_1,\sO_{\widetilde{\hcoY}\times \hchow})\to 0,\\
\sE xt^i(\sK_2,\sO_{\widetilde{\hcoY}\times \hchow})
\simeq \sE xt^{i}(\sK_1,\sO_{\widetilde{\hcoY}\times \hchow})\ (i\geq 2).
\end{cases}
\end{align*}
Note that $\sE xt^{i}({\iota_{\Vs}}_*\sK,\sO_{\widetilde{\hcoY}\times \hchow})\not =0$ only for $i=4$ since ${\iota_{\Vs}}_*\sK$ is the invertible sheaf on $D$ by (1), and $D$ is smooth and of codimension 4 in $\widetilde{\hcoY}\times \hchow$ by Lemma \ref{D} (see \cite[Cor.~3.5.11]{BH} for example). 
Therefore we see that
$\sE xt^{2}(\sK_2,\sO_{\widetilde{\hcoY}\times \hchow})\simeq 
\sE xt^{4}({\iota_{\Vs}}_*\sK,\sO_{\widetilde{\hcoY}\times \hchow})$, and
$\sE xt^{i}(\sK_2,\sO_{\widetilde{\hcoY}\times \hchow})\not =0$ only for $i=2$
and possibly $i=1$.
Moreover, we see that 
we can compute $\det \sK_2=2\Vs$ as a sheaf on 
$\widetilde{\hcoY}\times \hchow$ by using
the exact sequences (\ref{eq:rem}). 
Thus
$\sE xt^{1}(\sK_2,\sO_{\widetilde{\hcoY}\times \hchow})$
is actually nonzero and is supported on $\Vs$, and 
$\sK_2$ is generically of rank 2 on $\Vs$.
\end{proof}

We show that the locally free resolution as stated in 
Theorem $\ref{thm:newker}$
will be obtained by taking the dual of the complex as in Lemma \ref{lem:sK}.

\begin{proof}[{\bf Proof of Theorem $\ref{thm:newker}$}]~

By the proof of
Lemma \ref{lem:sK2} (2),
we obtain the following exact sequence: 
\begin{eqnarray}
0\to
\sO_{\widetilde{\hcoY}} 
\boxtimes \Omega^1_{\hchow/\mathrm{G}(2,V)}(-L)
\to
\widetilde{\sT}\boxtimes \sF^*(-H)
\to\\
\widetilde{\sQ} 
\boxtimes \sO_{\hchow}\oplus 
\widetilde{\sS}^*_L\boxtimes
\sO_{{\hchow}}(L-H)
\to
\sE xt^1 (\sK_2, \sO_{\widetilde{\hcoY}\times \hchow})
\to 0.
\nonumber
\end{eqnarray}
Therefore it remains to show
that ${\iota_{\Vs}}_*\sP\simeq \sE xt^1 (\sK_2, \sO_{\widetilde{\hcoY}\times \hchow})$.

By Lemma \ref{lem:sK}, the commutative diagram (\ref{eq:K})
is extended in the following one
with exact row and column except the third row:

\vspace{5pt}

\begin{equation}
\label{eq:extend}
\xymatrix{
& & & 0\ar[d] & 0\ar[d] & \\
&  &  & \sC_1(M+H)\ar[r]\ar[d] & {\iota_{\Vs}}_*\sJ(M+H+L)\ar[r]\ar[d] & 0\\
0\ar[r] & \sC_2\ar[r]\ar[d]& 
\sC_3\ar[r]\ar[d]
 & \sC_4 \oplus \sC_1(M+H)\ar[r]\ar[d] & 
{\iota_{\Vs}}_*\sI_{\Delta/\Vs}(M+H+L)\ar[r]\ar[d] & 0\\
0\ar[r] & \sC_1\oplus \sC_2\ar[r]& 
\sC_3\ar[r]
 & \sC_4 \ar[r]\ar[d] &
{\iota_{\Vs}}_*\mathcal{K}\ar[r]\ar[d] & 0\\
& & & 0 & 0.}
\end{equation}

From this diagram, we will construct a surjection $\sK_2\to \sP^*
(M+H)$.
Note that the diagram induces
a surjection from the kernel $\sK_1$ of 
the map $\sC_3\to \sC_4$
to the kernel of the composite 
\begin{eqnarray*}
0\oplus \sC_1(M+H)\to
\sC_4\oplus
\sC_1(M+H)
\to
{\iota_{\Vs}}_*\sI_{\Delta/\Vs}(M+H+L).
\end{eqnarray*}
The latter is also the kernel of
$\sC_1(M+H)\to 
{\iota_{\Vs}}_*\sJ(M+H+L)$.
Therefore, by (\ref{eq:P*}), we have a surjection
$\sK_1\to \sP^*(M+H)$.
Now we verify that
the composite of
\[
\sC_1\oplus \sC_2\to\sC_3\to
\sC_1(M+H)
\]
from the diagram (\ref{eq:extend})
and the restriction map $\sC_1(M+H)
\to
\sC_1(M+H)|_{\Vs}$ is a $0$-map.
Indeed, the map 
$\sC_2\to\sC_3\to
\sC_1(M+H)|_{\Vs}$ is clearly a $0$-map since it comes from
the second row of the diagram (\ref{eq:extend}).
By construction,
the map $\sC_1\to \sC_3\to \sC_1(M+H)$ is $\SL(V)$-equivariant.
Since the uniqueness of such a map 
(\cite[Lem.~5.10]{DerSym}),
this is 
the natural map which is obtained
from $\sO_{\widetilde{\hcoY}\times \hchow}(-\Vs)\to
\sO_{\widetilde{\hcoY}\times \hchow}$ by tensoring
$\sC_1(M+H)$.
Thus $\sC_1\to\sC_3\to
\sC_1(M+H)|_{\Vs}$ is also a $0$-map.

Consequently, we obtain a surjection
$\sK_2\simeq \sK_1/(\sC_1\oplus \sC_2)\to
\sP^*(M+H)$.
Let $\delta$ be the kernel of the surjection
$\sK_2\to \sP^*
(M+H)$.
By Lemma \ref{lem:sK2} (2),
$\sK_2$ is a coherent sheaf on $\Vs$ and is generically of rank $2$,
and $\sP^*$ has the same property by construction,
thus the support of $\delta$ is properly contained in $\Vs$.
Hence we have
$\sE xt^1(\sP^*
(M+H), \sO_{\widetilde{\hcoY}\times \hchow})
\simeq
\sE xt^1(\sK_2, \sO_{\widetilde{\hcoY}\times \hchow})$.
Finally, by Theorem \ref{cla:duality}, we have 
$\sE xt^1(\sP^*
(M+H), \sO_{\widetilde{\hcoY}\times \hchow})
\simeq {\iota_{\Vs}}_*\sP$
since $\sP$ is reflexive on $\Vs$.
\end{proof}

\subsection{Additional descriptions of $\sP$}
Now we collect some results about $\sP$, which are used in the section \ref{section:Corr}.

Applying Lemma \ref{lem:sK2},
we can compute the dual of (\ref{eq:P*}) explicitly,
in which $\sP$ and $\Delta$ are more directly related.
\begin{prop}
\label{prop:dualP*}
The dual of $(\ref{eq:P*})$ induces
the following exact sequence\,$:$
\begin{equation}
\label{eq:P}
0\to \sO_{\widetilde{\hcoY}}\boxtimes \sO_{\hchow}(-L)|_{\Vs}\to
\widetilde{\sQ}\boxtimes \sO_{\hchow}|_{\Vs}\to \sP\to \omega_{\Delta}\otimes \omega_{\Vs}^{-1}\otimes \sO_{\hchow}(-L)\to 0.
\end{equation}
Moreover,
$\sE xt^i ({\iota_{\Vs}}_*\sP,\sO_{\widetilde{\hcoY}\times \hchow})$ is 
nonzero only for $i=1,2$, and
$\sE xt^1 ({\iota_{\Vs}}_*\sP,\sO_{\widetilde{\hcoY}\times \hchow})\simeq 
{\iota_{\Vs}}_*(\sP^*)(M+H)$ and
$\sE xt^2 ({\iota_{\Vs}}_*\sP,\sO_{\widetilde{\hcoY}\times \hchow})$
is an invertible sheaf on $D$.
\end{prop}

\begin{proof}
Taking the dual of (\ref{eq:P*}), we obtain
\begin{equation}
\label{eq:lem1}
0\to \sO_{\widetilde{\hcoY}}\boxtimes \sO_{\hchow}(-L)|_{\Vs}\to
\widetilde{\sQ}\boxtimes \sO_{\hchow}|_{\Vs} \to \sP\to 
\sE xt^1(\sJ(L),\sO_{\Vs}) \to 0,
\end{equation}
and 
\begin{equation}
\label{eq:lem2}
\text{$\sE xt^i(\sP^*,\sO_{\Vs})\simeq \sE xt^{i+1} (\sJ(L), \sO_{\Vs})$ for $i\geq 1$}.
\end{equation}
By the definition of $\sK$ (see (\ref{eq:sK})),
we have the exact sequence
\begin{eqnarray}
\label{eqn:IJ}
\sE xt^i(\sK(-M-H),\sO_{\Vs})\to \sE xt^i(\sI_{\Delta/\Vs}(L),\sO_{\Vs})\to\\
 \sE xt^i(\sJ(L),\sO_{\Vs})\to \sE xt^{i+1}(\sK(-M-H),\sO_{\Vs})\nonumber
\end{eqnarray}
for any $i$.
By Lemma \ref{lem:sK2} (1) and \cite[Cor.~3.5.11]{BH},
we have 
\begin{eqnarray}
\label{eqn:lem'}
\text{$\sE xt^i(\sK(-M-H),\sO_{\Vs})\not =0$ only for $i=3$, and}\\
\text{$\sE xt^3(\sK(-M-H),\sO_{\Vs})$ is an invertible sheaf on $D$}.\nonumber
\end{eqnarray}
Therefore
\begin{equation}
\label{eq:lem3}
\text{$\sE xt^i(\sI_{\Delta/\Vs}(L),\sO_{\Vs})\simeq \sE xt^i(\sJ(L),\sO_{\Vs})$
for $i\not =2,3$}.
\end{equation}
On the other hand,
by the exact sequence
\[
0\to \sI_{\Delta/\Vs}(L)\to \sO_{\Vs}(L)\to \sO_{\Delta}(L)\to 0,
\]
we have 
\begin{equation}
\label{eq:lem4}
\text{$\sE xt^i(\sI_{\Delta/\Vs}(L),\sO_{\Vs})=0$ ($i\geq 2$),\
$\sE xt^1(\sI_{\Delta/\Vs}(L),\sO_{\Vs})\simeq 
\sE xt^2(\sO_{\Delta}(L),\sO_{\Vs})$}
\end{equation}
since $\Delta$ is smooth and of codimension two in $\Vs$
(Proposition \ref{prop:flat}).
We note that, by the standard computation of dualizing sheaf of $\Delta$
(Theorem \ref{cla:duality}),
we have 
\[
\sE xt^1(\sI_{\Delta/\Vs}(L),\sO_{\Vs})\simeq 
\sE xt^2(\sO_{\Delta}(L),\sO_{\Vs})\simeq
\omega_{\Delta}\otimes \omega_{\Vs}^{-1}\otimes
\sO_{\hchow}(-L).
\]  
Therefore we obtain the first assertion by (\ref{eq:lem1}), (\ref{eq:lem3}) and (\ref{eq:lem4}).
By (\ref{eq:lem2}), (\ref{eq:lem3}) and (\ref{eq:lem4}),
we have
$\sE xt^i(\sP^*,\sO_{\Vs})=0$ for $i\geq 3$.
By (\ref{eq:lem2}), (\ref{eqn:IJ}), (\ref{eqn:lem'}), and (\ref{eq:lem4}),
we also have $\sE xt^2(\sP^*,\sO_{\Vs})=0$, and
$\sE xt^1(\sP^*,\sO_{\Vs})$ is an invertible sheaf on $D$.
Since $\sP$ is reflexive on $\Vs$, we have
$\sP\simeq \sP^*(\det \sP)$
by \cite[Prop.~1.10]{H}, where $\det \sP$ is the 
invertible sheaf $(\det \widetilde{\sQ}+L)|_{\Vs}$ on $\Vs$ by (\ref{eq:P*}).
Therefore we have
$\sE xt^i(\sP,\sO_{\Vs})\simeq \sE xt^i(\sP^*,\sO_{\Vs})\otimes \det \sP$.  
Finally, by the Grothendieck-Verdier duality (Theorem \ref{cla:duality}),
we have
\[
\sE xt^i({\iota_{\Vs}}_*\sP,\sO_{\widetilde{\hcoY}\times \hchow})
\simeq 
{\iota_{\Vs}}_*\sE xt^{i-1}(\sP,\sO_{\Vs}(M+H)).
\]
Therefore we obtain the second assertion.
\end{proof}

We can read off
the following assertion about 
$\sE xt^{\bullet}(\sJ,\sO_{\widetilde{\hcoY}\times \hchow})$ from the proof of 
Proposition \ref{prop:dualP*} above.

\begin{prop}
\label{prop:Delta'}
$\sE xt^{\bullet}(\sJ,\sO_{\widetilde{\hcoY}\times \hchow})$
is nonzero only for $\bullet=1,2,3$.
$\sE xt^{1}(\sJ,\sO_{\widetilde{\hcoY}\times \hchow})$,
$\sE xt^{2}(\sJ,\sO_{\widetilde{\hcoY}\times \hchow})$, and
$\sE xt^{3}(\sJ,\sO_{\widetilde{\hcoY}\times \hchow})$
are invertible sheaves on $\Vs$, $\Delta$, and $D$,
respectively.
\end{prop}

Using (\ref{eq:P}), we derive flatness of $\sP$ from that of $\Delta$ (Propositions \ref{prop:flat} and \ref{prop:Deltasmooth}).

\begin{prop}
\label{prop:flatP}
$\sP$ is flat over $\hchow$ and over $\widetilde{\hcoY}\setminus \Prt_{\sigma}$
\end{prop}

\begin{proof}
We split (\ref{eq:P}) into two short exact sequences;
\[
0\to \sO_{\widetilde{\hcoY}}\boxtimes \sO_{\hchow}(-L)|_{\Vs}\to
\widetilde{\sQ}\boxtimes \sO_{\hchow}|_{\Vs}\to \sA\to 0,
\]
and 
\[
0\to \sA\to \sP\to \omega_{\Delta}\otimes \omega_{\Vs}^{-1}\otimes \sO_{\hchow}(-L)\to 0,
\]
where $\sA$ is the cokernel of 
$\sO_{\widetilde{\hcoY}}\boxtimes \sO_{\hchow}(-L)|_{\Vs}\to
\widetilde{\sQ}\boxtimes \sO_{\hchow}|_{\Vs}$.
By the first exact sequence and \cite[Thm.~22.5]{M},
we see that $\sA$ is flat over $\hchow$ and $\widetilde{\hcoY}$
since $\sO_{\widetilde{\hcoY}}\boxtimes \sO_{\hchow}(-L)|_{\Vs}\to
\widetilde{\sQ}\boxtimes \sO_{\hchow}|_{\Vs}$ is injective when
it is restricted to the fibers over $\hchow$ and $\widetilde{\hcoY}$.
Therefore the assertion follows in a standard way from flatness of $\sA$ and
Propositions \ref{prop:flat} and \ref{prop:Deltasmooth}.
\end{proof}

\subsection{Cutting $\Delta$, $D$ and $\sP$}
\label{cutting}
The aim of this subsection is
to derive the properties of 
the several restrictions of $\sP$
which we quote in the subsections
\ref{section:P}, \ref{section:full} and \ref{sub:IK}.

We take $\widetilde{Y}$ and $X$ as in the subsection 
\ref{section:XY}, and
$S$ and $\hW$ as in the subsection \ref{subsection:SW}.
We recall that 
\[
\Delta_1=\Delta|_{\widetilde{Y}\times X},\,
\Delta_2=\Delta|_{S\times \hW}
\]
and
now we also set \[
D_1=D|_{\widetilde{Y}\times X},\,
D_2=D|_{S\times \hW},
\]
where $D$ is an irreducible component of 
$\Delta'$ $($the subsection $\ref{sub:Delta'})$.
We denote by $\Delta_{x;1}$ and $D_{x;1}$
the fibers of $\Delta_1\to X$ and $D_1\to X$ over $x\in X$ respectively, and
similarly for the fibers over a point of $\widetilde{Y}$, and for $\Delta_2$ and $D_2$.

\begin{prop}
\label{lem:DeltaDflat2}
\begin{enumerate}[$(1)$]
\item
For any $s\in S$,
$\Delta_{s;2}$ is cut out from
the fiber of $\Delta\to \widetilde{\hcoY}$ over $s$
by regular sequences.
For any $s\in \Gamma$, $D_{s;2}$ is cut out from
the fiber of $D\to F_{\rho}$ over $s$ by regular sequences,
where
we recall that $\Gamma$ is the ramification locus of the double cover
$S\to P_2$.
\item
$\Delta_2$ and $D_2$ are cut out from
$\Delta$ and $D$, respectively by regular sequences.
\end{enumerate}
\end{prop}

\begin{proof}
By \cite[Thm.~17.4 (iii)]{M}, 
$\Delta|_{S\times \hchow}$ and $D|_{\Gamma\times \hchow}$ are cut out from
$\Delta$ and $D$, respectively by regular sequences
since $\Delta\to \widetilde{\hcoY}$ and 
$D\to F_{\rho}$ are equidimensional, and
$S$ and $\Gamma$ are 
cut out from
$\widetilde{\hcoY}$ and $F_{\rho}$, respectively by regular sequences.

Therefore, for (2), we have only to show that
$\Delta_2$ and $D_2$ are cut out from
$\Delta|_{S\times \hchow}$ and $D|_{\Gamma\times \hchow}$, 
respectively by regular sequences.
Then, by the flatness of 
$\Delta|_{S\times \hchow}\to S$ and $D|_{\Gamma\times \hchow}\to \Gamma$,
it suffices to show the assertion (1) by
\cite[Cor.~to Thm.~22.5]{M}.
Let $s\in S$ be a point. Since $\hW$ is cut out regularly by
2 members of $|H|$ from $\Vs_s$,
it suffices to show that
$\Delta_{s;2}$ of $\Delta_2\to S$ over $s$
is $1$-dimensional,
and  
$D_{s;2}$ of $D_2\to \Gamma$ over $s$ (if $s\in \Gamma$)
is $1$-dimensional (here we use 
the descriptions of fibers of 
$\Delta\to \widetilde{\hcoY}$ and 
$D\to F_{\rho}$ as in Proposition \ref{prop:Deltasmooth} (2) and
Lemma \ref{D}, and
\cite[Thm.~17.4 (iii)]{M}).

As for $\Delta_2$,
it is isomorphic to 
the blow-up
of $Z_S$ along all the flopped curves
by Lemma \ref{prop:Delta2descrip}.
Since $Z\to S$ is a $\mP^1$-bundle, we see that
any fiber of $\Delta_2\to S$ is $1$-dimensional.

As for $D_2$, we assume by contradiction that
$D_{s;2}\subset \hW$ is $2$-dimensional.
Let $[V_1]$ be the vertex of the rank $3$ quadric corresponding to $s$. 
By the definition of $D$ and $\dim D_{s;2}=2$,
we see that the image $\overline{D}_{s;2}\subset \overline{W}$
of $D_{s;2}$ coincides with the $\rho$-plane $\mathrm{P}_{V_1}$.
Let $\mathrm{P}'_{V_1}\subset Z_S$ be the strict transform of $\mathrm{P}_{V_1}$. If $\mathrm{P}'_{V_1}$ contains a fiber $q$ of $Z_S\to S$,
then $q$ parameterizes lines in $\mP(V)$ through $[V_1]$, hence
the image of $q$ on $S$ is contained in $\Gamma$ and corresponds to a rank $3$ quadric with $[V_1]$ the vertex. Since the genus of $\Gamma$ is three, $\Gamma$ is not the image of 
the rational surface $\mathrm{P}'_{V_1}$.
Therefore $\mathrm{P}'_{V_1}$ contains at most a finite number of 
fibers of $Z_S\to S$, and hence $\mathrm{P}'_{V_1}\to S$ is dominant.
This means that any fiber of $Z_S\to S$ contains at least 
one point corresponding to 
a line through $[V_1]$. Therefore any quadric in $P_2$ passes through $[V_1]$.
Since $\Bs P_2$ coincides with $v_2(\mP(V))\cap P^{\perp}_2$,
$[V_1]$ is one of $w_i$'s.
This is a contradiction as in Step 5 of the proof of 
Proposition \ref{prop:flop} (2). 
\end{proof}

\begin{prop}
\label{lem:DeltaDflat}
\begin{enumerate}[$(1)$]
\item
For any $x\in X$,
$\Delta_{x;1}$ and $D_{x;1}$ are cut out from
the fibers of $\Delta\to \hchow$ and
$D\to\hchow$
over $x$, respectively,
by regular sequences.
\item
$\Delta_1$ and $D_1$ are cut out from
$\Delta$ and $D$, respectively by regular sequences.
\end{enumerate}

\end{prop}

\begin{proof}
Let $x\in X$ be a point.
Arguing as in the proof of Proposition
\ref{lem:DeltaDflat2},
it suffices to show that
the fiber $\Delta_{x;1}$ of $\Delta_1\to X$ over $x$
is $1$-dimensional, and
the fiber $D_{x;1}$ of $D_1\to X$ over $x$
is $0$-dimensional
since $\widetilde{Y}$ is cut out regularly by
5 members of $|H|$ from $\Vs_x$.

As for $\Delta_1$, 
the assertion follows by Lemma \ref{Delta1fiber}.

We consider $D_1\to X$.
Let $x_1+x_2$ be the $0$-cycle corresponding to $x$. 
By the definition of $D$,
the fiber of $D\to \hchow$ over $x$
consists of points corresponding to $\rho$-conics
contained in the $\rho$-plane ${\rm P}_{x_1}$ or ${\rm P}_{x_2}$.
By \cite[Prop.~4.20]{Geom},
such $\rho$-conics correspond to singular quadrics with
$x_1$ or $x_2$ contained in the vertices.
If $D_{x;1}$ is positive dimensional,
then $P_3\subset \mP({\ft S}^2 V^*)$ contains a pencil of singular quadrics
with $x_1$ or $x_2$ contained in the vertices.
In particular, $H$ contains a line, a contradiction to the assumption
(\ref{eq:XYcond}).
\end{proof}

In the sequel, we use the notation
defined above Proposition
\ref{prop:restP'}.

\begin{prop}
\label{prop:restP}
$\sP_x$ is a reflexive sheaf on $\Vs_x$.
$\sE xt^i ({\iota_x}_*\sP_x,\sO_{\widetilde{\hcoY}})$ is 
nonzero only for $i=1,2$, and
$\sE xt^1 ({\iota_x}_*\sP_x,\sO_{\widetilde{\hcoY}})\simeq 
{\iota_x}_*(\sP_x^*)(M)$ and
$\sE xt^2 ({\iota_x}_*\sP_x,\sO_{\widetilde{\hcoY}})$
is an invertible sheaf on $D|_{\Vs_x}$.
The locally free resolution $(\ref{eqnarray:onVs'})$
restricts to $\widetilde{\hcoY}$, namely, the following is exact on $\widetilde{\hcoY}:$
\begin{eqnarray}
\label{eqnarray:onVsx'}
0\to
\sO_{\widetilde{\hcoY}}^{\oplus 2} 
\to
\widetilde{\sT}^{\oplus 2}
\to
\widetilde{\sQ} \oplus 
\widetilde{\sS}^*_L
\to
{\iota_x}_*\sP_x\to 0.
\end{eqnarray}

A similar assertion holds for $\sP_y$ with $y\in \widetilde{\hcoY}\setminus \Prt_{\sigma}$. In particular, we have the following exact sequence on $\hchow:$ 
\begin{eqnarray}
\label{eqnarray:onVsy'}
0\to
\Omega^1_{\hchow/\mathrm{G}(2,V)}(-L)
\to
\sF^*(-H)^{\oplus 4}
\to\\
\sO_{\hchow}^{\oplus 3}\oplus 
\sO_{{\hchow}}(L-H)^{\oplus 3}
\to
{\iota_y}_*\sP_y\to 0.
\nonumber
\end{eqnarray}

\end{prop}
\begin{proof}
We only show the assertions for $\sP_x$. 
The first assertion follows from the second assertion of Proposition \ref{prop:dualP*} and \cite[Lem.~1.1.13]{HL} since $\Vs$ and $D$ are restricted by regular sequences
to the fiber of $\Vs\to \hchow$ over $x$.

Note that the ideal sheaf $\sJ_x$ 
of $\Delta'|_{\Vs_x}$ is $\sJ\otimes \sO_{\Vs_x}$
by Proposition \ref{prop:Delta'} and [ibid.]
since 
$\Vs$, $\Delta$ and $D$ are restricted by regular sequences
to the fiber of $\Vs\to \hchow$ over $x$.
Therefore (\ref{eq:P*}) induces an exact sequence
\[
0\to \sP^*_x\to \widetilde{\sQ}^*|_{\Vs_x}\to \sJ_x\to 0.
\]
Using this, we obtain the assertions for
$\sE xt^{\bullet} ({\iota_x}_*\sP_x,\sO_{\widetilde{\hcoY}})$
in a similar way to the proof of 
Proposition \ref{prop:dualP*}.

By splitting (\ref{eqnarray:onVs'}) into two short exact sequences,
we can prove the second assertion in a very similar way to 
the proof of Proposition \ref{prop:flatP} by \cite[Thm.~22.5]{M}
using the flatness of $\sP$.
\end{proof}

We recall that
\[
\sP_1:=\sP|_{\widetilde{Y}\times X},\,
\sP_2:=\sP|_{S\times \hW}.
\]

\begin{prop}
\label{prop:flatP1}
$\sP_1$ is flat over $X$, and $\sP_2$ is flat over $S$.
\end{prop}

\begin{proof}
Proofs for both assertions are similar, hence we only consider
$\sP_1$.

By Proposition \ref{prop:flatP}, $\sP\times_{\hchow} X$ is flat over $X$.
Let $x$ be any point of $\hchow$.
By Proposition \ref{lem:DeltaDflat} (1),  
the descriptions of 
$\sE xt^{\bullet} ({\iota_x}_*\sP_x,\sO_{\widetilde{\hcoY}})$
in Proposition \ref{prop:restP}, 
and \cite[Lem.~1.1.13]{HL},
$\widetilde{Y}$ is cut out from $\widetilde{\hcoY}$ by
$\sP_x$-regular sequences.
Therefore $\sP_1$ is flat over $X$ by \cite[Cor.~to Thm.~22.5 (2)$\Rightarrow$(1)]{M}.
\end{proof}

We note that it is possible to derive Proposition \ref{cla:tower}
by the proof of Proposition \ref{prop:flatP1}.

\begin{proof}[{\bf Deriving (\ref{eq:ShW}) and (\ref{eq:PYX})}]
We only derive (\ref{eq:PYX}).
Note that the ideal sheaf $\sJ_1$ 
of $\Delta'|_{\widetilde{Y}\times X}$ is $\sJ\otimes \sO_{\widetilde{Y}\times X}$
by Proposition \ref{prop:Delta'} and \cite[Lem.~1.1.13]{HL}
since 
$\Vs$, $\Delta$ and $D$ are restricted by regular sequences
by Proposition \ref{lem:DeltaDflat} (2).
Therefore (\ref{eq:P*}) induces an exact sequence
\[
0\to \sP^*_1\to \widetilde{\sQ}^*|_{\widetilde{Y}\times X}\to \sJ_1(L)\to 0.
\]
Taking the dual of this exact sequence,
we obtain (\ref{eq:PYX}) in the same way 
to derive (\ref{eq:P}).
\end{proof}


\vspace{1cm}
\noindent {\footnotesize Department of Mathematics, Gakushuin University, 
Toshima-ku,Tokyo 171-8588,$\,$Japan }{\footnotesize \par}

\noindent {\footnotesize e-mail address: hosono@math.gakushuin.ac.jp} 

\vspace{5pt}

\noindent {\footnotesize Graduate School of Mathematical Sciences,
University of Tokyo, Meguro-ku,Tokyo 153-8914,$\,$Japan }{\footnotesize \par}

\noindent {\footnotesize e-mail address: takagi@ms.u-tokyo.ac.jp} 

\end{document}